\documentclass[11pt]{article}

\usepackage{graphicx}
\usepackage{latexsym}

\usepackage{lipsum}
\usepackage{graphicx}
\usepackage{authblk}
\usepackage[english]{babel}
\usepackage{color}
\usepackage{hyperref}
\hypersetup{
    colorlinks=true,
    linkcolor=blue,
    filecolor=blue,      
    urlcolor=blue,
    citecolor=blue,
    }
\usepackage{dsfont}
\usepackage{geometry}
\usepackage{float}
\usepackage{bbm}

\usepackage[nottoc]{tocbibind}
 \usepackage{indentfirst}

\graphicspath{{./figure/}}

\usepackage{pgfplots}

\usepackage{tikz}
\usepackage{caption}
\usepackage{subcaption}
\DeclareCaptionLabelFormat{custom}
{%
      \textbf{Figure #2}
}
\DeclareCaptionLabelSeparator{custom}{ }
\DeclareCaptionFormat{custom}
{%
\centering
    #1 #2 #3 
}
\usepackage{amsmath,amsfonts,amssymb,amsthm}

\numberwithin{equation}{section}
\usepackage{etoolbox}
\apptocmd{\thebibliography}{}{}{}
\newcommand{\supp}{\mbox{supp}}

\theoremstyle{plain}
\newtheorem{theorem}{Theorem}[section]
\newtheorem{lemma}[theorem]{Lemma}

\newtheorem{proposition}[theorem]{Proposition}

\newtheorem{reminder*}{[theorem]Reminder}

\newtheorem{details*}[theorem]{Details}

\newtheorem{comm*}{Comment}
\newtheorem{definition}[theorem]{Definition} 
\newtheorem{definition*}{[theorem]Definition}

\newtheorem{notation*}{Notation}

\newtheorem{remark}[theorem]{Remark}
\newcommand{\der}{\textup{d}}

\date{}
\title{Long-time asymptotics for coagulation equations with injection that do not have stationary solutions}
\author[1]{Iulia Cristian}
\author[2]{Marina A. Ferreira}
\author[3]{Eugenia Franco}
\author[4]{Juan J. L. Vel\'{a}zquez}
\affil[1]{Institute for Applied Mathematics, University of Bonn, Endenicher Allee 60, 53115 Bonn, Germany
\href{mailto:cristian@iam.uni-bonn.de}{cristian@iam.uni-bonn.de}}
\affil[2]{Department of Mathematics and Statistics, University of Helsinki, P.O. Box 68, FI-00014 Helsingin yliopisto, Finland

\href{mailto:marina.ferreira@helsinki.fi}{marina.ferreira@helsinki.fi}}
\affil[3]{Institute for Applied Mathematics, University of Bonn, Endenicher Allee 60, 53115 Bonn, Germany
\href{mailto:franco@iam.uni-bonn.de}{franco@iam.uni-bonn.de}}
\affil[4]{Institute for Applied Mathematics, University of Bonn, Endenicher Allee 60, 53115 Bonn, Germany
\href{mailto:velazquez@iam.uni-bonn.de}{velazquez@iam.uni-bonn.de}}

\begin{document}

\maketitle
\abstract{
In this paper we study a class of coagulation equations including a source term that injects in the system clusters of size of order one.
 The coagulation kernel is homogeneous, of homogeneity $\gamma < 1$, such that $K(x,y)$ is approximately $x^{\gamma + \lambda} y^{-\lambda}$, when $x$ is larger than $y$.
 We restrict the analysis to the case $\gamma + 2 \lambda \geq 1 $. In this range of exponents, the transport of mass toward infinity is driven by collisions between particles of different sizes. 
 This is in contrast with the case considered in \cite{ferreira2021self} where $\gamma + 2 \lambda <1$. In that case,
 the transport of mass toward infinity is due to the collision between particles of comparable sizes. 
 In the case $\gamma+2\lambda\geq 1$, the interaction between particles of different sizes leads to an additional transport term in the coagulation equation that approximates the solution of the original coagulation equation with injection for large times.
We prove the existence of a class of self-similar solutions for suitable choices of $\gamma $ and $\lambda$ for this class of coagulation equations with transport.
We prove that for the complementary case such self-similar solutions do not exist.
}
\\

\tableofcontents


\section{Introduction}
\subsection{Aim of the paper} 
In this paper we study the long-time behavior of the coagulation equation with injection:
\begin{equation} \label{eq:coag with inj} 
\partial_{t}f\left( t, x\right)  =\mathbb{K}[f](t,x)+\eta\left(  x\right),
\end{equation}
where
\begin{equation}\label{coagulation-kernel-def}
 \mathbb{K}[f](t,x):=\frac{1}{2}\int_{0}^{x}K\left(  x-y,y\right)
f\left( t, x-y\right)  f\left(t,  y\right)  \der y-\int_{0}^{\infty}K\left(
x,y\right)  f\left(t,  x\right)  f\left( t, y\right)  \der y
\end{equation}
and where $\eta\geq 0$. We will assume in all the following that $\eta(x) \not\equiv 0$ and that it is either compactly supported or it decays fast enough with the cluster size.
 
The study of problems with this form arises naturally in problems of aerosols and atmospheric science (\cite{Friedlander,Vehkamaki+others,Vehkamaki}). In this context the function $f(t,x)$ denotes
the density of clusters with size $x$ at time $t$.

The collision operator $\mathbb{K}[f](t,x)$ in (\ref{coagulation-kernel-def}) was introduced by Smoluchowski (see \cite{Smoluchowski}). The kernel $K(x,y)$ encodes the information about the mechanism driving the coagulation of clusters. 
In this paper we are interested in kernels arising in atmospheric science applications. A common feature of these kernels is the homogeneity property. Indeed, in many cases the rate of aggregation scales like a power law with the cluster size. This means that the coagulation kernel $K(x,y)$ satisfies
\begin{equation}\label{kernel-homogeneity}
    K(a x,a y)=a^{\gamma}K(x,y) \textup{ for } a >0 \textup{ and } (x,y)\in\mathbb(0,\infty)^{2},
\end{equation}
for some $\gamma\in\mathbb{R}$. 

On the other hand, since the coagulation process does not depend on the order in which the clusters of size $x$ and $y$ are chosen we have the symmetry property
\begin{equation}\label{kernel-symmetry}
    K(x,y)=K(y,x) \textup{ for } (x,y)\in(0,\infty)^{2}.
\end{equation}
We will assume that the coagulation kernel $K$ satisfies 
\begin{equation} \label{kernel bounds} 
c_2 \left[  \frac{x^{\gamma+\lambda}}{y^{\lambda}}+\frac
{y^{\gamma+\lambda}}{x^{\lambda}} \right]  \leq K\left(  x,y\right)   \leq c_1 \left[  \frac{x^{\gamma+\lambda}}{y^{\lambda}}+\frac
{y^{\gamma+\lambda}}{x^{\lambda}} \right] \textup{ and } \gamma+2\lambda\geq 0,
\end{equation}
with $0<c_{2}\leq c_{1}<\infty$, where $\gamma$ is the homogeneity parameter introduced in (\ref{kernel-homogeneity}) and $\lambda\in\mathbb{R}$.
 The polynomial bounds \eqref{kernel bounds} are satisfied by many of the most relevant collision kernels arising in aerosol and atmospheric science, such as the \textit{diffusive} coagulation kernel and the \textit{free molecular} kernel (see for instance \cite{Friedlander}).

Notice that the condition $\gamma+2\lambda\geq 0$ in \eqref{kernel bounds} does not imply any loss of generality. This can be seen from the fact that the function $x^{\alpha}y^{\beta}+x^{\beta}y^{\alpha}$ can be written as $\left[  \frac{x^{\gamma+\lambda}}{y^{\lambda}}+\frac
{y^{\gamma+\lambda}}{x^{\lambda}} \right]$ with $\gamma+\lambda=\max\{\alpha,\beta\}$ and $-\lambda=\min\{\alpha,\beta\}$.

Since we will consider solutions of (\ref{eq:coag with inj}), (\ref{coagulation-kernel-def}) in which $f(t, \cdot)$ is a Radon measure, it is convenient to impose the following condition on the kernel $K$:
\begin{equation}\label{kernel cont}
    K\in\textup{C}((0,\infty)^{2}).
\end{equation}

It is well known that part or all the mass of the solutions of (\ref{eq:coag with inj}), (\ref{coagulation-kernel-def}) can escape towards $x=\infty$ in finite, or even zero time. This phenomenon is known as \textit{gelation}. For a more detailed analysis on this matter, see for example \cite{bookBanasiak,gelation_example}. In order to guarantee that gelation does not take place, we will assume in all the following that
\begin{equation}\label{avoid_gelation_parameters}
    \gamma<1 \textup{ and } \gamma+\lambda<1.
\end{equation}
See  \cite{davies1999smoluchowski} for a more detailed discussion on the gelation regimes. 

We expect that gelation does not take place under the weaker assumptions $\gamma\leq 1, \gamma+\lambda\leq 1$ (see \cite{bookBanasiak}). However, in the critical cases $\gamma=1$ or $\gamma+\lambda=1$, the self-similar solutions are not defined using power laws to scale the particle sizes but most likely using exponential functions. Given that the analysis of these solutions would require arguments different from the ones in this paper, we will not consider this case here. 

The existence of solutions for equation \eqref{eq:coag with inj}, \eqref{coagulation-kernel-def} has been considered in \cite{dubovskiui1994mathematical} and in \cite{escobedo2006dust}. In this paper we study the long time  behavior of the solutions of equation \eqref{eq:coag with inj}.
Due to the presence of the source we can expect the solutions to (\ref{eq:coag with inj}), (\ref{coagulation-kernel-def}) to converge to a stationary non-equilibrium solution $\overline{f}=\overline{f}(x)$ as $t\rightarrow\infty,$ i.e. to a solution of
\begin{equation}\label{stationary_equation_initial}
    \mathbb{K}[f](x)+\eta(x)=0.
\end{equation}
However, it turns out that if $\eta\not\equiv 0$ and $\gamma+2\lambda\geq 1,$ a solution for (\ref{stationary_equation_initial}) does not exist. In \cite{hayakawa1987irreversible}, the discrete stationary coagulation model
\begin{equation}\label{discrete_equation_with_injection}
    \sum_{k=1}^{n-1}K_{k,n-k}f_{k}f_{n-k}- \sum_{k=1}^{\infty}K_{k,n}f_{k}f_{n}+\overline \delta_{k,n}=0, \textup{ } k\geq 1,
\end{equation}
 has been studied for the explicit coagulation kernel
\begin{equation}
    K_{k,n}=k^{\gamma + \lambda }n^{-\lambda }+k^{-\lambda}n^{\gamma + \lambda }, \textup{ } k,n\in\mathbb{N}
\end{equation} 
when $|\gamma + \lambda |<1$, $|\lambda |<1 $, $|\gamma |< 1$. 
 Formal asymptotics for the large size behavior of the solutions $f_k $ of \eqref{discrete_equation_with_injection} has been obtained in  \cite{hayakawa1987irreversible}. The results in that paper indicate that a solution of (\ref{discrete_equation_with_injection}) exists if and only if $\gamma+2\lambda<1$.
 
In the case of general kernels satisfying the assumptions (\ref{kernel-symmetry}),
(\ref{kernel bounds}), (\ref{kernel cont}), and source terms $\eta$ decreasing
fast enough, it has been proved in \cite{ferreira2019stationary} that the solutions of (\ref{stationary_equation_initial}) (as well as its discrete counterpart) exist, if and only if $\gamma+2\lambda<1.$

It is worth to remark that it has been proved in \cite{ferreira2019stationary} that the solutions of  (\ref{stationary_equation_initial}) can be estimated, up to a multiplicative constant, from above and below by the power law $x^{- \frac{3+\gamma }{2}}$ for large values of $x$.

Since in the case $\gamma + 2 \lambda \geq 1 $ a stationary solution of \eqref{stationary_equation_initial} does not exist, we cannot expect the solutions to (\ref{eq:coag with inj}), (\ref{coagulation-kernel-def}) to behave as the stationary solution $\overline{f}$ of
(\ref{stationary_equation_initial}) as $t\rightarrow\infty$ for $x$ of order one. 
It is then natural to ask what is the long time asymptotics of the solutions
to (\ref{eq:coag with inj}), (\ref{coagulation-kernel-def}) for large values of $t$ and $x.$ The scaling hypothesis that has been extensively used in the study of coagulation equations suggests that the mass of the particle distributions $f\left(t,x\right)$ is concentrated in cluster sizes $x$ of order $t^{p}$ for a suitable exponent $p$ that would be determined from dimensional considerations, which take into account the way in which the mass rescales in time. In the case of kernels $K\left(  x,y\right)=x^{\gamma + \lambda }y^{-\lambda}+y^{-\lambda}x^{\gamma + \lambda }$, $ -1 <  \lambda <0$ with $0\leq\gamma + \lambda<1$ and $\gamma <1$, it was suggested in \cite{davies1999smoluchowski}, using a combination of matched asymptotics and
numerical simulations, that the long time behavior of the solutions of (\ref{eq:coag with inj}), (\ref{coagulation-kernel-def}) is given by self-similar solutions with the form
\begin{equation}
f_s\left(  t,x\right)  =\frac{1}{t^{\frac{3+\gamma  }{1-\gamma}}}\Phi\left(
\xi \right)  \ ,\ \ \xi=\frac{x}{t^{\frac
{2}{1-\gamma }}}. \label{SelfSimFormOther}
\end{equation}

The approximation (\ref{SelfSimFormOther}) can be expected to be valid for large
cluster sizes, i.e. $x\gg1$. We will use from now the notation with $x\gg1$ to indicate large cluster sizes $x$.
In the case considered in \cite{davies1999smoluchowski} we have that $\gamma+2\lambda<1$ and
therefore stationary solutions $f_{s}$ solving (\ref{stationary_equation_initial}) exist. In this case, (\ref{stationary_equation_initial}) and (\ref{SelfSimFormOther}) suggest that $\Phi\left(
\xi\right)$ behaves for small values of $\xi$ as $K\xi^{-\frac{3+\gamma}{2}}$, for a suitable constant $K>0.$ More precisely, plugging (\ref{SelfSimFormOther}) in (\ref{eq:coag with inj}), (\ref{coagulation-kernel-def}), it follows that $\Phi$ solves
\begin{equation}
-\frac{2}{1-\gamma}\xi\Phi_{\xi}-\frac{3+\gamma}{1-\gamma}\Phi=\mathbb{K}
\left[  \Phi\right], \label{SelfSimFlux}
\end{equation}
where $\Phi$ satisfies the following boundary condition at $\xi\rightarrow0$ that guarantees that there is a constant flux of particles from the origin: 
\begin{equation}
\lim_{R\rightarrow0}\int_{0}^{R}\der\xi\int_{R-\xi}\der\eta K\left(  \xi
,\eta\right)  \xi\Phi\left(  \xi\right)  \Phi\left(  \eta\right)  =J,
\label{SelfSimFlBC}
\end{equation}
with $J=\int_{0}^{\infty}x\eta\left(  x\right)  \der x.$
The existence of solutions of equation
(\ref{SelfSimFlux}) satisfying the constant flux solution condition at $\xi=0$, \eqref{SelfSimFlBC}, 
has been rigorously proved in \cite{ferreira2021self} for kernels $K$
satisfying (\ref{kernel-symmetry}), (\ref{kernel bounds}), (\ref{kernel cont}),
(\ref{avoid_gelation_parameters}) with $\gamma + 2 \lambda<1$.

The picture described above, which combines the stationary behavior $\overline{f}$, (cf. (\ref{stationary_equation_initial})) for cluster sizes $x$ of order one, and the self-similar behavior (\ref{SelfSimForm}) for large cluster sizes, provides a rather complete description of the long time behavior of the solutions to (\ref{eq:coag with inj}), (\ref{coagulation-kernel-def}) in the case $\gamma+2\lambda<1.$
However, the same scenario cannot yield a description of the long time asymptotics of the solutions to (\ref{eq:coag with inj}), (\ref{coagulation-kernel-def}) if $\gamma+2\lambda\geq1,$ because, as explained above, in this case a solution of (\ref{stationary_equation_initial})
does not exist.

Notice that the existence/non-existence of solutions to \eqref{stationary_equation_initial} is related to the existence of stationary solutions of \eqref{stationary_equation_initial} yielding a constant flux of particles with the form of a power law, i.e. $\overline{f}(x)= cx^{- \frac{\gamma + 3 }{2}}$. These solutions exist for $\gamma + 2 \lambda < 1 $ and they do not exist for $\gamma + 2 \lambda \geq 1 $. In the framework of the wave turbulence it would be stated that the kernels $K$ with $\gamma + 2 \lambda < 1 $ satisfy the locality property, while the kernels $K$ with $\gamma + 2 \lambda \geq 1 $ do not have the locality property, see \cite{zel2002physics}. 

In this paper we study the long term behavior of the
solutions to (\ref{eq:coag with inj}), (\ref{coagulation-kernel-def}) in the case $\gamma + 2 \lambda \geq 1 $ (and the non-gelling regime $\gamma < 1 $ and $\gamma + \lambda < 1$) when the kernel is homogeneous, hence it can be expressed as 
\begin{equation}
K\left(  x,y\right)  =\left(  x+y\right)  ^{\gamma} F \left(  \frac{x}%
{x+y}\right)  \ \ ,\ \ F \left(  s\right)  =F\left(  1-s\right)
\ \ \text{for }s\in\left(  0,1\right).  \label{KernPhi}%
\end{equation}%. 
Notice that \eqref{kernel bounds} implies that
$\overline c_1 \leq F(s) \leq \overline c_2 $
for some $\overline c_1 , \overline c_2 >0$.
We will assume in the following a condition that is more restrictive than \eqref{kernel bounds}, namely 
\begin{equation}
\lim_{s\rightarrow0^{+}}\left[  s^{\lambda}F \left(  s\right)  \right]  =1. 
\label{PhiAsymp}%
\end{equation}
In order to prove that in the case $\gamma + 2 \lambda \geq 1$ there are no solutions of \eqref{stationary_equation_initial}, the main idea used in  \cite{ferreira2019stationary} is based on the fact that, for this range of exponents, the transfer of clusters
of size $x$ of order one towards very large cluster sizes is so fast
that the concentration of clusters with size of order one would become zero.

In the case of time dependent solutions having initially finite mass, this
increases linearly due to the fact that
\[
\partial_t \left( \int_0^\infty x f(t, x) \der x \right) = \int_0^\infty x \eta(x) \der x .
\] 
For large times, due to the
increase of the average cluster size, we might expect that, if $\gamma
+2\lambda\geq1$, there should be a fast transport of the newly injected
clusters of order one towards much larger cluster sizes. This almost instantaneous transport
results in small concentrations of clusters of order one for large times.

We now remark that the part of the coagulation operator
which describes the coagulation between particles of different sizes can be
approximated by means of a transport operator in the space of cluster sizes.
These arguments, that will be described in detail using formal asymptotics in
Section \ref{sec:heuristics}, show that the solutions of
(\ref{eq:coag with inj}), (\ref{coagulation-kernel-def}) for large cluster sizes can be
approximated by means of the following equation if $\gamma+2\lambda\geq
1$
\begin{equation}
\partial_{t}f\left(  t,x\right)  +\frac{\partial_{x}\left(  x^{\gamma+\lambda
}f\left(  t,x\right)  \right)  }{\int_{0}^{\infty}z^{\gamma+\lambda}f\left(
t,z\right) \der z}=\mathbb{K}\left[  f\right]  (t,x),\ \ \text{for }x\gg1.
\label{coagTransp}%
\end{equation}

We emphasize that the non-local transport term $\frac{\partial_{x}\left(
x^{\gamma+\lambda}f\left(  t,x\right)  \right)  }{\int_{0}^{\infty}%
z^{\gamma+\lambda}f\left(  t,z\right)  \der z}$ is a consequence of the presence
of the source $\eta\left(  x\right)  $ in (\ref{eq:coag with inj}). The fact
that the contributions to the coagulation operator $\mathbb{K}\left[
f\right]  $ which are due to the aggregation of particles with very different
sizes can be approximated by a differential operator has been extensively used in the literature of coagulation equations
(cf. \cite{Friedlander}). The resulting first order terms are often
referred to represent \textit{heterogeneous condensation} (cf. \cite{Friedlander}). 
On the other side, the term $\mathbb{K}\left[  f\right]  $ in
(\ref{coagTransp}) describes the aggregation of clusters of comparable sizes.

The solutions of \eqref{coagTransp} are expected to describe the asymptotic behavior of the solution of equation \eqref{eq:coag with inj} both when $\gamma + 2 \lambda > 1$ and when $\gamma + 2 \lambda =1 $, even if in these two cases we have two slightly different scenarios.
Namely, when $\gamma +2\lambda >1$, we will have that if $x$ is of order $1 $, then $f(t, x) \rightarrow 0 $ as $t \rightarrow \infty $, while when $\gamma + 2 \lambda =1 $ we will have that if $x$ is of order $1 $, then  $f(t, x) \rightarrow f_s(x)$ as $t \rightarrow \infty $, where $f_s $ is a solution of 
\begin{equation}\label{eq:stationary}
\mathbb K[f](x) + \eta(x)- x^{-\lambda} f(x)=0, 
\end{equation}
(compare with \eqref{stationary_equation_initial}). Due to the presence of the term $x^{-\lambda}f(x)$, equation \eqref{eq:stationary} might have solutions when $\gamma+2\lambda=1$.

The scaling properties of (\ref{coagTransp}) suggest that this equation is
compatible with the existence of self-similar solutions for equation \eqref{coagTransp} with the form
\begin{equation} \label{SelfSimForm}
f_s\left(t, x \right)  =\frac{1}{t^{\frac{3+\gamma}{1-\gamma}}}\Phi\left(
\frac{x}{t^{\frac{2}{1-\gamma}}}\right).
\end{equation} 
Here the self-similar profile $\Phi$ satisfies the following equation, obtained by substituting equality \eqref{SelfSimForm} in equation \eqref{coagTransp} and using the self-similar change of variables $\xi=\frac{x}{t^{\frac{2}{1-\gamma}}}$,
\begin{equation}\label{eq:ss} 
-\frac{3+\gamma}{1-\gamma}\Phi\left(  \xi\right)  -\frac{2}{1-\gamma}\xi \partial_\xi
\Phi\left(  \xi\right)  +\frac{1}{\int_{0}^{\infty}\eta
	^{\gamma+\lambda}\Phi\left(  \eta\right)  \der \eta}\frac{\partial}{\partial\xi
}\left(  \xi^{\gamma+\lambda}\Phi\left(  \xi\right)  \right)  =\mathbb{K}%
\left[  \Phi\right]  \left(  \xi\right)  \ \ ,\ \ \xi>0. 
\end{equation}

 The main result of this paper is to determine the range
of exponents $\gamma$ and $\lambda$ satisfying $\gamma+2\lambda\geq1$ and the non-gelation conditions \eqref{avoid_gelation_parameters} for which self-similar solutions of
(\ref{coagTransp}) with the form (\ref{SelfSimForm})
exist (see Figure \ref{fig1} for a classification of these exponents). Specifically, we will prove the following. Suppose that $\gamma
+2\lambda\geq1$  and that (\ref{avoid_gelation_parameters}) holds. Then

\begin{itemize}
\item If $\gamma>-1,$ there exists at least one self-similar solution of
(\ref{coagTransp}) with the form (\ref{SelfSimForm}).
\item If $\gamma\leq -1$ and $\gamma +2\lambda >1$ , no solutions of (\ref{coagTransp}) with the form
(\ref{SelfSimForm}) exist (See Figure \ref{fig1}).
\item If $\gamma\leq -1$ and $\gamma +2\lambda =1$, we prove that there are no self-similar solutions $f_{s}$ of the form \eqref{SelfSimForm} such that $\int_0^1 x^{-\lambda} \Phi(x) \der x  < \infty$.
\end{itemize}

The meaning of the condition $\int_0^1 x^{-\lambda} \Phi(x) \der x < \infty$ is that the number of clusters removed by the coagulation process in any bounded time interval is finite.
We therefore do not exclude the existence of a self-similar solution of equation \eqref{coagTransp} with $\int_0^1 x^{-\lambda} \Phi(x) \der x = \infty$. 

A remarkable property of the self-similar solutions of (\ref{coagTransp}), that we constructed in this paper, is that they vanish identically for $0<\xi\ < 
\rho(M_{\gamma+\lambda}):=\left( \frac{1-\gamma}{2 M_{\gamma +\lambda}}\right)^{\frac{1}{1-\gamma -\lambda} }$, where
\begin{align*}
    M_{\gamma +\lambda} :=\int_{(0,\infty)}\xi^{\gamma+\lambda}\Phi(\xi) \der\xi.
\end{align*}
The fact solutions $\Phi$ of (\ref{eq:ss}) vanish in an interval $(0,\rho(M_{\gamma+\lambda}))$ means that, for large
times $t,$ the injected particles are transferred almost instantaneously to
clusters with sizes $x\geq\rho(M_{\gamma+\lambda}) t^{\frac{2}{1-\gamma}}.$ Moreover, the fraction
of clusters with sizes $x< \rho(M_{\gamma+\lambda}) t^{\frac{2}{1-\gamma}}$ becomes negligible
for very long times. The existence of this "minimal" cluster size for large
times is a remarkable feature that, to our knowledge, has not been observed in the literature on self-similar solutions for the coagulation equation (see for instance \cite{escobedo2006dust}, \cite{escobedo2005self}, \cite{fournier2006local}, \cite{menon2004approach}, \cite{hungtruong}).
In particular, it is worth to notice that this
behavior of the self-similar solutions is very different from the one
exhibited by the self-similar solutions obtained in the case $\gamma
+2\lambda<1$ in \cite{ferreira2021self}.

The results that we obtain in this paper in the critical case $\gamma+2\lambda=1$ are more fragmentary than those obtained for $\gamma+2\lambda>1$. In the case $\gamma+2\lambda=1$, we only prove the existence of a solution $\Phi$ of (\ref{eq:ss}) for $\gamma>-1$ with $\Phi=0$ on $(0,\rho(M_{\gamma+\lambda}))$, but we do not prove that any solution vanishes in the interval $(0,\rho(M_{\gamma+\lambda}))$.

On the other hand, we will prove in this paper that the function $\Phi$ which
describes the self-similar profile in (\ref{SelfSimForm}) decreases
exponentially as $\xi\rightarrow\infty$ in the same manner as the self-similar
solutions constructed in \cite{ferreira2021self} and as it usually happens for
the self-similar solutions of coagulation equations in problems without
injection, see \cite{fournier2006local} and \cite{escobedo2006dust}. 

As indicated above, if $\gamma\leq -1$ and $\gamma+ 2\lambda >1$, there are no self-similar
solutions of (\ref{coagTransp}) with the form (\ref{SelfSimForm}) and, if $\gamma\leq -1$ and $\gamma+ 2\lambda =1$, there are no self-similar
solutions of (\ref{coagTransp}) with the form (\ref{SelfSimForm}) satisfying $\int_0^1 x^{-\lambda} \Phi(x) \der x <\infty$. It is
natural to ask what is the long time asymptotics of the solutions of
(\ref{eq:coag with inj}), (\ref{coagulation-kernel-def}) in this case. This question will be
the subject of study of a future work.

\begin{figure}[H]%
\centering
\includegraphics[width=0.7\textwidth, height=12cm]{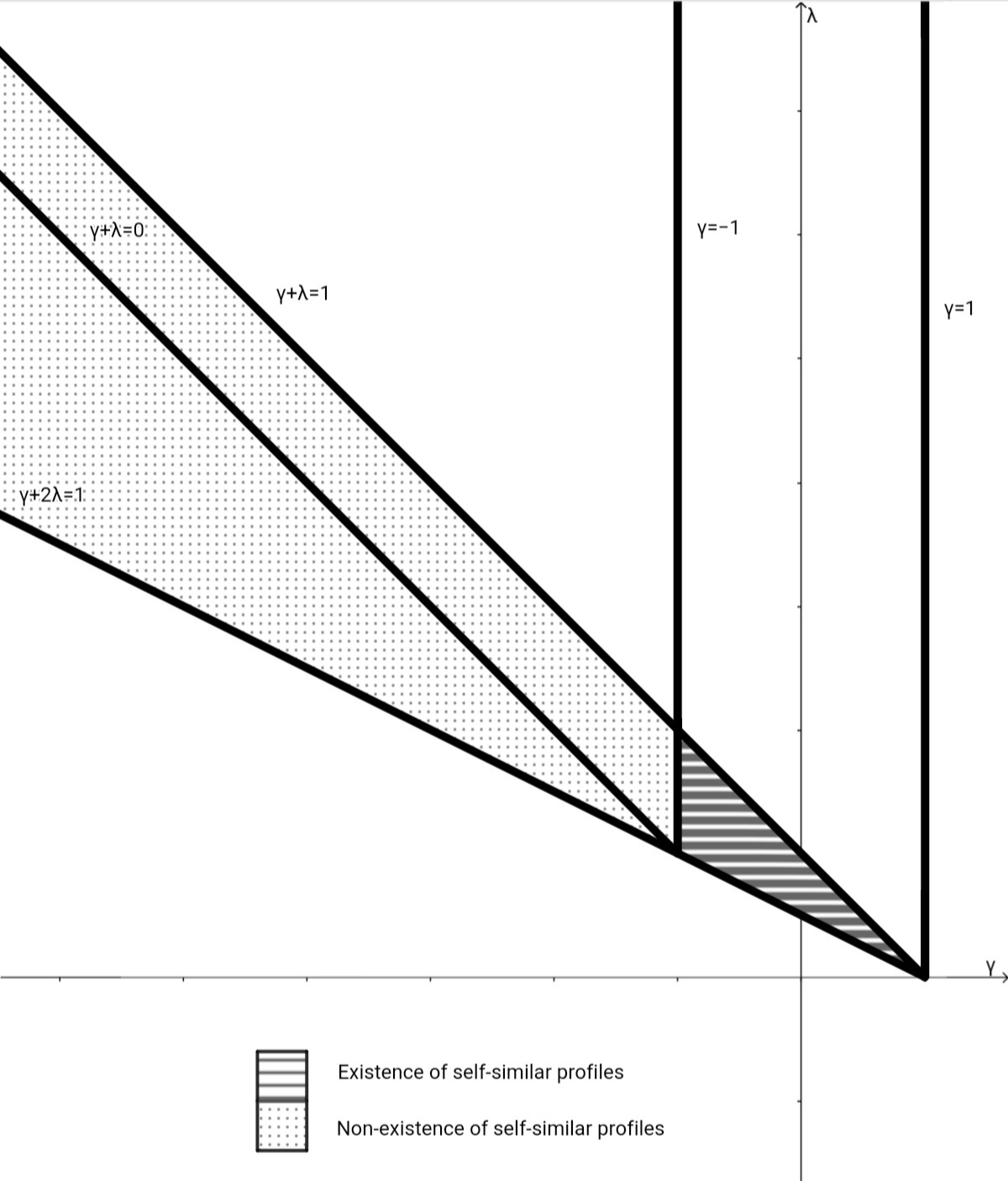}
\caption{
Coefficients for existence and  non-existence of self-similar\linebreak profiles  in the case $\gamma+2\lambda>1$}
\label{fig1}
\end{figure}

In both cases the
coagulation term $\mathbb{K}\left[  f\right]  $ in (\ref{coagTransp}) which
describes the aggregation of particles with comparable sizes is negligible for
large times, and the long time behavior of the distribution of clusters is
determined by the coagulation of particles of size $x$ of order one with large
particles. The main difference between the cases $\gamma+\lambda \geq 0$
and $\gamma+\lambda < 0$  arises from the fact that in the second case
the transfer of clusters (not monomers) from the region where $x$ is of order
one to $x\gg1$, \textit{heterogeneous condensation}, is relevant. 

There are several results in the physical literature which are related, and are consistent with the ones in this paper, see for instance \cite{BC12}, \cite{KC12}, \cite{KMR98}, \cite{KMR99}. 
In \cite{BC12} a coagulation model with kernels satisfying \eqref{kernel bounds} with parameters $\gamma , \lambda $ such that $\gamma + 2 \lambda \geq 1 $ and including also a source term and a removal of particles term has been studied. In that problem, the cluster concentrations as $t \rightarrow \infty $ are determined by the coalescence of particles of very different sizes.
It is then possible to approximate the coagulation-removal model for large clusters by means of the equation
\[
\partial_t f(t, x)= - \partial_x \left( x^{\gamma + \lambda} f(t, x) \right) +\int_1^{x/2 } y^{1-\lambda} f(t, y ) \der y - f(t,x) \int_x^K y^{\gamma + \lambda} f(t, y) \der y,
\]
where $K$ is the maximum particle size in the system. 
The numerical simulations in \cite{BC12} show that the concentration of cluster sizes of order one approach to a stationary solution that converges to zero if $K$ is sent to infinity.

In \cite{KC12} a coagulation model with injection and with kernels satisfying \eqref{kernel bounds} with $\gamma =0$ and $\lambda \in (1/2, 1] $ has been considered. Numerical simulations and formal computations in \cite{KC12} suggest that the cluster concentration for clusters of order one tends to zero as $t \rightarrow \infty$. 
In addition, it is suggested that the concentration for large clusters are described by a  self-similar solution. 

In \cite{KMR98}, \cite{KMR99} coagulation equations with injection are considered with kernels satisfying \eqref{kernel bounds} with  $-\lambda =\gamma \leq -1 $. It is seen there that for $\gamma + 2 \lambda \geq 1 $, the solutions of the corresponding coagulation equation behave in a non self-similar manner and decay logarithmically. This is in agreement with the non-existence of self-similar solutions that we obtained in this paper for $\gamma \leq -1$.

\subsection{Notation and plan of the paper}

We use the notation $\mathbb R_* :=(0, \infty )$ and $\mathbb R_+ :=[0, \infty ). $
Given an interval $I \subset \mathbb R_+ $ we denote with $C_c(I)$ the Banach space of the functions on $I$ that are continuous and compactly supported. We endow the space $C_c(\mathbb R_*)$ with the supremum norm denoted with $\|\cdot \|_\infty.$
We denote with $\mathcal M_+ (I)$ the space of the non-negative Radon measures on $I$. 
Given a measure $\mu \in \mathcal M_+(I)$ we  denote with $\|\mu \|_{TV}$ the total variation norm of $\mu.$
For a compact interval $I\subseteq\mathbb{R}_{\ast}$ and two bounded measures $\mu,\nu,$ we denote the Wasserstein metric by $W_{1},$ namely
 \begin{align}
W_{1}(\mu,\nu)=\sup_{||\varphi||_{\textup{Lip}}\leq 1}\int_{I}\varphi(x)(\mu-\nu)(\der x),\label{W metric}
\end{align}
 where the supremum is taken over the Lipschitz functions and where $\| f \|_{\textup{Lip}}=\|f\|_\infty + [f]_{\textup{Lip}}$, with $[f]_{\textup{Lip}}=\sup_{\underset{x\neq y}{x,y\in I}}\frac{|f(x)-f(y)|}{|x-y|}$.

To keep the notation lighter, we denote with $C$ and $c$ constants that might change from line to line in the computations. In addition, we use the notation $f\lesssim g,$ for two functions $f,g$, to mean that there exists a constant $C>0$ such that $f\leq C g$.

Moreover, given a measure $\mu$, we denote with $M_\alpha(\mu)$ the $\alpha$ moment of $\mu $, i.e.
\[
M_{\alpha} (\mu):=\int_{\mathbb R_*} x^\alpha \mu(\der x ). 
\] 
To simplify the notation, in some cases, we write $M_\alpha $ instead $M_\alpha (\mu) $ if the choice of the measure $\mu$ is clear in the argument.

We use the notation $f \sim g $ as $x \rightarrow x_0$ to indicate that $\lim_{x \rightarrow x_0} \frac{f(x)}{g(x)} =1$, while we use the notation $f \approx g $ to say that there exists a constant $M>0$ such that 
\[
\frac{1}{M } \leq \frac{f}{g} \leq M. 
\]

As previously mentioned, we use the notation $x\gg 1$ for large cluster sizes $x$. Additionally, for cluster sizes $x,y$,  we denote  $x\gg y$ or $y\ll x$ to mean that $x$ is much larger than $y$. For two terms, $A$ and $B$, we use the notation $A\simeq B$  to mean informally that $A$ can be approximated in terms of $B$ in the region under consideration in the respective formula.

The paper is organized as follows.
In Section \ref{sec:heuristics} we present a heuristic motivation to study the existence of the self-similar solutions considered in this paper.
In Section \ref{sec:main} we present the main results of the paper regarding the existence and non-existence of a self-similar profile and its properties.
In Section \ref{sec:ideas} we explain the main ideas behind the proofs of existence and non-existence, skipping the technical difficulties of the proofs. 
Section \ref{sec:exifstence} deals with the proof of the existence of a self-similar solution for equation \eqref{eq:ss} when $\gamma > -1$. We also prove in this section that the solution decays exponentially for large values. Section \ref{sec:non existence} deals with the non-existence of self-similar solutions for equation \eqref{eq:ss} when $\gamma \leq -1$.

\section{Asymptotic description of the long time behavior
\label{sec:heuristics}}

\subsection{The case $\gamma + 2 \lambda > 1$}

In this section we describe the long time asymptotics of the solutions of
\eqref{eq:coag with inj}, \eqref{coagulation-kernel-def} if $\gamma+2\lambda> 1$ using
formal asymptotic arguments.

As indicated in the Introduction, we expect $f\left(  t, x\right)  $ to
converge to zero as $t\rightarrow\infty$ for $x$ of order one. Therefore, the
contribution due to the term $\frac{1}{2}\int_{0}^{x}K\left(  x-y,y\right)
f\left(  t,x-y\right)  f\left(  t,y\right)  \der y$ can be expected to be
negligible in this region since this term is quadratic in $f$ and we can
expect the linear term $\int_{0}^{\infty}K\left(  x,y\right)   f\left(  t,y\right)  \der y f\left(
t,x\right) $ to give a larger contribution. We will
check that these assumptions are self-consistent, in the sense that they will
predict an asymptotic behavior for $f$ for which the assumptions made hold.

We examine the asymptotic behavior of the linear term $\int_{0}^{\infty
}K\left(  x,y\right)   f\left(  t,y\right)  \der y f\left(  t,x\right)$
when $x$ is of order one. 
Due to the effect of the coagulation, we expect the
distribution $f\left(  t,y\right)  $ to be concentrated for long times in the
larger cluster sizes $y$ as $t\rightarrow\infty.$ Using the assumptions
(\ref{KernPhi}) and (\ref{PhiAsymp}) we obtain the following asymptotic
behavior of $K\left(  x,y\right)  $ for $x\ll y$%
\begin{equation}
K\left(  x,y\right) \simeq  y^{\gamma+\lambda}x^{-\lambda}\label{B1}.%
\end{equation}

We then expect to have the following asymptotics as $t\rightarrow\infty,$ due
to the concentration of $f$ in the large cluster sizes%
\[
\int_{0}^{\infty}K\left(  x,y\right)  f\left(  t,y\right)  \der y\sim x^{-\lambda
}\int_{0}^{\infty}y^{\gamma+\lambda}f\left(  t,y\right)  \der y\ \ \text{as\ \ }%
t\rightarrow\infty
\]
for $x$ of order one.
Then, considering the dominant terms in \eqref{eq:coag with inj}, \eqref{coagulation-kernel-def} for $x$ of order one, we obtain the following equation%
\begin{equation}
\partial_{t}f\left(  t,x\right)  =-x^{-\lambda}f\left(  t,x\right)  \int
_{0}^{\infty}y^{\gamma+\lambda}f\left(  t,y\right)  \der y+\eta\left(  x\right).
\label{A1}%
\end{equation}

As explained in the introduction, if $\gamma+2\lambda> 1,$ 
since steady
states describing the cluster concentrations with $x$ of order one do not exist, we expect
to have $f\left(  t,x\right)  \rightarrow0$ as $t\rightarrow\infty.$ This
suggest that we should have $M_{\gamma+\lambda}=\int_{0}^{\infty}%
y^{\gamma+\lambda}f\left(  t,y\right)  \der y\rightarrow\infty$ as $t\rightarrow
\infty.$ Suppose that $\partial_{t}M_{\gamma+\lambda}\ll M_{\gamma+\lambda}$
as $t\rightarrow\infty$ (something that would happen if $M_{\gamma+\lambda}$
behaves like a power law, as we will see to be the case). Then (\ref{A1})
implies the following asymptotic behavior for $f\left(  t,x\right)  $%
\begin{equation}
f\left(  t,x\right)  \sim\frac{x^{\lambda}\eta\left(  x\right)  }{\int
_{0}^{\infty}y^{\gamma+\lambda}f\left(  t,y\right)  dy}\ \ \text{as\ \ }%
t\rightarrow\infty\label{A2}%
\end{equation}
for $x$ of order one. We notice that the higher order contributions to
(\ref{A2}) due to the contributions of the term $\frac{1}{2}\int_{0}%
^{x}K\left(  x-y,y\right)  f\left(  t,x-y\right)  f\left(  t,y\right)  \der y$ are
of order $\frac{1}{ \left( M_{\gamma+\lambda}\right)^{2}}$ or smaller.
Therefore, these contributions will be neglected in the following.

The equation (\ref{A2}) yields an approximate formula for the concentration of
clusters with $x$ of order one. We now approximate the part of the coagulation
operator which is due to the aggregation of particles with size one with very
large particles. To this end we introduce a characteristic length $L=L\left(
t\right)  \gg1$ such that we can approximate $f\left(  t,y\right)  $ by means
of (\ref{A2}) for $y\leq L.$ We attempt to approximate the evolution of
the distribution $f\left(  t,x\right)  $ for $x\gg L.$
Then, the coagulation
operator in (\ref{coagulation-kernel-def}) can be approximated, using the symmetry
properties of the first term on the right of (\ref{coagulation-kernel-def}), as follows%
\begin{align*}
\mathbb{K}\left[  f\right]  \left(  t,x\right)    & =\frac{1}{2}\int_{0}%
^{x}K\left(  x-y,y\right)  f\left(  t,x-y\right)  f\left(  t,y\right)
\der y-\int_{0}^{\infty}K\left(  x,y\right)  f\left(  t,x\right)  f\left(
t,y\right)  \der y\\
& =\left[ \int_{0}^{L}K\left(  x-y,y\right)  f\left(  t,x-y\right)  f\left(
t,y\right)  \der y-\int_{0}^{L}K\left(  x,y\right)  f\left(  t,x\right)  f\left(
t,y\right)  \der y\right]  \\
& +\left[  \frac{1}{2}\int_{L}^{x-L}K\left(  x-y,y\right)  f\left(
t,x-y\right)  f\left(  t,y\right)  \der y-\int_{L}^{\infty}K\left(  x,y\right)
f\left(  t,x\right)  f\left(  t,y\right)  \der y\right].
\end{align*}

We can rewrite this formula as%
\begin{equation}
\mathbb{K}\left[  f\right]  \left(  t,x\right)  =\left[  \int_{0}^{L}\left[
K\left(  x-y,y\right)  f\left(  t,x-y\right)  -K\left(  x,y\right)  f\left(
t,x\right)  \right]  f\left(  t,y\right)  \der y\right]  +\mathbb{K}\left[
f\chi_{\left[  L,\infty\right)  }\right]  \left(  t,x\right),  \label{A3}%
\end{equation}
where $\chi_{\left[  L,\infty\right)  }$ denotes the characteristic function
of the interval $\left[  L,\infty\right)  .$ We now approximate $\mathbb{K}%
\left[  f\right]  \left(  t,x\right)  $ for large values of $x$ and
$t\rightarrow\infty$ and more precisely for $x\gg L.$ To this end we use
(\ref{A2}) and we assume also that $K$ and $f$ are sufficiently regular for
large values of $x.$ Then, using the fact that $y\ll x$ we obtain the
following approximation%
\begin{equation}
K\left(  x-y,y\right)  f\left(  t,x-y\right)  -K\left(  x,y\right)  f\left(
t,x\right)  \simeq-y\frac{\partial}{\partial x}\left[  K\left(  x,y\right)
f\left(  t,x\right)  \right].  \label{A4}%
\end{equation}

Moreover, (\ref{B1}) yields an approximation for $K\left(  x,y\right)  $ if
$x\ll y.$ Exchanging the roles of $x$ and $y$ we obtain $K\left(  x,y\right)
\sim x^{\gamma+\lambda}y^{-\lambda}$ as $\frac{x}{y}\rightarrow\infty$ and
plugging this formula into (\ref{A4}) we obtain%
\[
K\left(  x-y,y\right)  f\left(  t,x-y\right)  -K\left(  x,y\right)  f\left(
t,x\right)  \simeq-y^{1-\lambda}\frac{\partial}{\partial x}\left[
x^{\gamma+\lambda}f\left(  t,x\right)  \right].
\]

Using this approximation in (\ref{A3}) we then obtain%
\[
\mathbb{K}\left[  f\right]  \left(  t,x\right)  =-\left[  \int_{0}%
^{L}y^{1-\lambda}f\left(  t,y\right)  \der y\right]  \left[  \frac{\partial
}{\partial x}\left[  x^{\gamma+\lambda}f\left(  t,x\right)  \right]  \right]
+\mathbb{K}\left[  f\chi_{\left[  L,\infty\right)  }\right]  \left(
t,x\right)
\]
for $x\gg1$ and $t\rightarrow\infty.$ We can now use (\ref{A2}) to derive a
formula for $\int_{0}^{L}y^{1-\lambda}f\left(  t,y\right)  \der y.$
We then obtain the approximation%
\begin{equation}
\mathbb{K}\left[  f\right]  \left(  t,x\right)  \simeq -\frac{\int_{0}^{\infty}%
y\eta\left(  y\right)  \der y}{\int_{0}^{\infty}y^{\gamma+\lambda}f\left(
t,y\right)  \der y}\frac{\partial}{\partial x}\left[  x^{\gamma+\lambda}f\left(
t,x\right)  \right]  +\mathbb{K}\left[  f\chi_{\left[  L,\infty\right)
}\right]  \left(  t,x\right).  \label{A5}%
\end{equation}

Notice that we use that $\int_{0}^{L}y\eta\left(  y\right)  \der y\simeq\int
_{0}^{\infty}y\eta\left(  y\right) \der y$ since by assumption $\eta\left(
y\right)  $ decreases sufficiently fast for large values of $y.$

We will denote as $f_{out}$ the distribution of particles in the region where
$x\gg1.$ More precisely we write $f_{out}=f\chi_{\left[  L,\infty\right)  }.$
Combining (\ref{eq:coag with inj}), (\ref{coagulation-kernel-def}) with (\ref{A5}) we obtain
the following evolution equation for $f_{out}$%
\begin{equation}
\partial_{t}f_{out}\left(  t,x\right)  +\frac{\int_{0}^{\infty}y\eta\left(
y\right)  \der y}{\int_{0}^{\infty}y^{\gamma+\lambda}f_{out}\left(  t,y\right)
\der y}\frac{\partial}{\partial x}\left[  x^{\gamma+\lambda}f_{out}\left(
t,x\right)  \right]  =\mathbb{K}\left[  f_{out}\right]  \left(  t,x\right).
\label{A6}%
\end{equation}

Notice that we use the approximation $\int_{0}^{\infty}y^{\gamma+\lambda
}f\left(  t,y\right)  \der y\simeq\int_{0}^{\infty}y^{\gamma+\lambda}%
f_{out}\left(  t,y\right)  \der y$ that might be expected because $f\left(
t,y\right)  \rightarrow0$ for $y\leq L$ (cf. (\ref{A2})).

In the rest of the paper we will study the properties of the self-similar
solutions associated to the equation (\ref{A6}). It is worth to remark that
the transport term on the left of (\ref{A6}) is the way in which the injection of
particles with size $x$ of order one affects the outer distribution of
clusters $f_{out}$. Indeed, multiplying (\ref{A6}) by $x$ and integrating we
obtain%
\begin{align} 
& \partial_{t}\left(  \int_{0}^{\infty}xf_{out}\left(  t,x\right)  \der x\right)
+\frac{\int_{0}^{\infty}y\eta\left(  y\right)  \der y}{\int_{0}^{\infty}%
y^{\gamma+\lambda}f_{out}\left(  t,y\right)  \der y}\int_{0}^{\infty}%
x\frac{\partial}{\partial x}\left[  x^{\gamma+\lambda}f_{out}\left(
t,x\right)  \right]  \der x \nonumber \\
&=\int_{0}^{\infty}x\mathbb{K}\left[  f_{out}\right]
\left(  t,x\right)  \der x.\label{A7}%
\end{align}

The mass conservation property associated to the coagulation kernel yields
\[ \int_{0}^{\infty}x\mathbb{K}\left[  f_{out}\right]  \left(  t,x\right)
\der x=0.
\]
On the other hand, integrating by parts in the second term on the left
of (\ref{A7}) we obtain $\int_{0}^{\infty}x\frac{\partial}{\partial x}\left[
x^{\gamma+\lambda}f_{out}\left(  t,x\right)  \right]  \der x=-\int_{0}^{\infty
}x^{\gamma+\lambda}f_{out}\left(  t,x\right)  \der x.$ Combining these results we
obtain%
\begin{equation}
\partial_{t}\left(  \int_{0}^{\infty}xf_{out}\left(  t,x\right)  \der x\right)
=\int_{0}^{\infty}x\eta\left(  x\right)  \der x.\label{A8}%
\end{equation}

The identity (\ref{A8}) states that the total mass of the clusters in the
outer region is equal to the injection rate. This result, that holds for long
times, could be expected because in the regime described in this section, the
injected particles are transferred instantaneously to large cluster sizes. This is also consistent with $f(t,x)\rightarrow 0$ as $t\rightarrow\infty$ when $x\approx 1$.

\subsection{The case $\gamma+2\lambda=1$} \label{sec:asympt =1}

In the case $\gamma+2\lambda=1$, the approximation of the concentrations of clusters $f(t, \cdot)$ in
the region where $x$ is of order one must be obtained in a different manner.
The
reason is that in this case we cannot expect the moment $M_{\gamma+\lambda}=\int_0^\infty x^{\gamma +\lambda} f(t,x) \der x $
to converge to infinity as $t\rightarrow\infty.$ Indeed, suppose that most
of the mass of the monomers is distributed in a characteristic length $L(t)$
that increases as $t\rightarrow\infty.$ Then, if we denote as $M_{0}$ and
$M_{1}$ the moments of $f$ of order zero and one, respectively, (i.e. $\int
_{0}^{\infty}f\left(  t,x\right) \der x$ and $\int_{0}^{\infty}xf\left(
t,x\right)  \der x $, respectively), we have that $M_{1}=M_{0} L(t)$ and $M_{\gamma+\lambda}%
=M_{0}L(t)^{\gamma+\lambda}.$ 

Assuming that $\int_{0}^{\infty}%
x\eta\left(  x\right) \der x=1$ and hence that $M_{1}\simeq t$ as $t\rightarrow\infty$, we deduce that $M_{0}L(t)=t$. 
The rescaling properties of
(\ref{eq:coag with inj}), (\ref{coagulation-kernel-def}) suggest that $\frac{M_{0}}{t}%
\simeq\left(  M_{0}\right)^{2}L(t)^{\gamma}.$ Therefore, plugging the
identity $M_{0}=\frac{t}{L(t)}$ in this formula we obtain $L(t)=
t^{\frac{2}{1-\gamma}}.$ Hence
\begin{equation} \label{scaling argument}
M_{\gamma+\lambda}=M_{0}L(t)^{\gamma+\lambda}=tL(t)^{\gamma+\lambda-1}=  t^{1+\frac
{2}{1-\gamma}\left(  \gamma+\lambda-1\right)  }=1. 
\end{equation}

It then follows that the self-similar rescaling ansatz implies that
$M_{\gamma+\lambda}$ remains of order one for large times. A consequence of
this is that we cannot approximate (\ref{eq:coag with inj}), (\ref{coagulation-kernel-def})
for clusters of order one by means of the equation (\ref{A1}). Instead of this
we will use a different approximation by splitting $f$ in an outer part which
describes the cluster distribution for $x$ of order $L$ and an inner part
that describes the cluster distribution for $x$ of order one. More precisely,
we write%
\begin{equation}
f\left(t,x\right)  =f_{inner}\left( t, x \right)  +f_{outer}\left(
t, x\right),  \label{B2}%
\end{equation}
where $f_{out} = f \chi_{[L, \infty ) }$, while $f_{inn} =f \chi_{(0, L] }$, for a constant $L>0.$

Using (\ref{eq:coag with inj}), (\ref{coagulation-kernel-def}) we would then obtain the
following evolution equation for $f_{inner}$ 
\begin{equation}
\partial_{t}f_{inner}\left(  t,x\right)  =\mathbb{K}\left[  f_{inner}\right]
\left(  t,x\right)  -f_{inner}\left(  t,x\right)  \int_{0}^{\infty}K\left(
x,y\right)  f_{outer}\left(  t,y\right)  \der y+\eta\left(  x\right),  \label{B3}%
\end{equation}
where we have used the decomposition (\ref{B2}) in the lost term of the
coagulation operator $\mathbb{K}\left[  f\right]  $. We use also the fact that
in order to compute the gain term for $x$ of order one we need to use only
$f_{inner}.$

By assumption, the main contribution of $f_{outer}\left(  t,y\right)  $ is due
to clusters with size $L \gg1.$ On the other hand, for $x$ of order one, we have the
approximation%
\[
K\left(  x,y\right)  f_{outer}\left(  t,y\right)  \simeq y^{\gamma+\lambda
}x^{-\lambda}f_{outer}\left(  t,y\right)
\]
and (\ref{B3}) becomes%
\begin{equation}
\partial_{t}f_{inner}\left(  t,x\right)  =\mathbb{K}\left[  f_{inner}\right]
\left(  t,x\right)  -f_{inner}\left(  t,x\right)  x^{-\lambda}\int_{0}%
^{\infty}y^{\gamma+\lambda}f_{outer}\left(  t,y\right) \der y+\eta\left(
x\right).  \label{B4}%
\end{equation}

The scaling argument above, \eqref{scaling argument}, suggests that $\int_{0}^{\infty}y^{\gamma+\lambda
}f_{outer}\left(  t,y\right) \der y$ approaches to a positive constant as
$t\rightarrow\infty$ if $f_{outer}$ behaves in a self-similar manner and
$\gamma+2\lambda=1.$ We will write%
\begin{equation}
M_{\gamma+\lambda}=\int_{0}^{\infty}y^{\gamma+\lambda}f_{outer}\left(
t,y\right)  \der y. \label{B5}%
\end{equation}

On the other hand, we will assume that the function $f_{inner}$, which is
described by means of \eqref{B4}, approaches to a stationary solution for large
times. We then obtain the following equation which would be expected to
describe the behavior of the solutions of (\ref{B4}) for long times%
\begin{equation}
\mathbb{K}\left[  f_{inner}\right]  \left(  x\right)  -M_{\gamma+\lambda
}f_{inner}\left(  x\right)  x^{-\lambda}+\eta\left(  x\right)  =0.\label{B6}%
\end{equation}

We can now derive an equation describing the
evolution of $f_{outer}.$ To this end we argue as in the case of
$\gamma+2\lambda>1$ in order to approximate the effect in $f_{outer}$ due to
the collisions of clusters with size $L$ with clusters with size of order
one. Using approximations analogous to the ones used in the derivation of
(\ref{A6}) (cf. (\ref{A3}), (\ref{A4})) we obtain the following approximation
for $\mathbb{K}\left[  f\right]  \left(  t,x\right)  $
\begin{equation}
\mathbb{K}\left[  f\right]  \left(  t,x\right)  \simeq-\left[  \int
_{0}^{\infty}y^{1-\lambda}f_{inner}\left(  t,y\right)  \der y\right]  \left[
\frac{\partial}{\partial x}\left[  x^{\gamma+\lambda}f_{outer}\left(
t,x\right)  \right]  \right]  +\mathbb{K}\left[  f_{outer}\right]  \left(
t,x\right)  \label{B7}%
\end{equation}
for $x$ of order $L.$

On the other hand, multiplying (\ref{B6}) by $x$ and integrating in $\left(
0,\infty\right)  $ we obtain, since $f_{inner}$ is zero for $x>L$, that%
\[
M_{\gamma+\lambda}\int_{0}^{\infty}y^{1-\lambda}f_{inner}\left(  y\right)
\der y=\int_{0}^{\infty}y\eta\left(  y\right)  \der y.
\]

Using this formula to eliminate $\int_{0}^{\infty}y^{1-\lambda}f_{inner}%
\left(  y\right)  \der y$ in (\ref{B7}) we obtain%
\[
\mathbb{K}\left[  f\right]  \left(  t,x\right)  \simeq-\frac{\int_{0}^{\infty
}y\eta\left(  y\right) \der y}{M_{\gamma+\lambda}}\left[  \frac{\partial
}{\partial x}\left[  x^{\gamma+\lambda}f_{outer}\left(  t,x\right)  \right]
\right]  +\mathbb{K}\left[  f_{outer}\right]  \left(  t,x\right). 
\]
Therefore, we obtain that $f_{outer}$ satisfies the equation (\ref{A6}).

The whole asymptotic behavior derived here relies on the
existence of solutions of the equation (\ref{B6}).
In this equation, the value
of $M_{\gamma+\lambda}$ is chosen as one associated to a self-similar solution
of (\ref{A6}). The equation (\ref{B6}) can then be interpreted as a stationary
solution for a coagulation equation with source $\eta$ and a removal term
$-M_{\gamma+\lambda}f_{inner}\left(  x\right)  x^{-\lambda}.$ The
results in \cite{ferreira2019stationary} imply that no solutions of (\ref{B6}) exist if
$M_{\gamma+\lambda}=0$ and $\gamma+2\lambda=1$, but when $M_{\gamma+\lambda} >0$ the existence/non-existence of stationary solutions for equation \eqref{B6} is still an open problem.

\section{Setting and main results} \label{sec:main}
The main results of this paper concern the existence and the non-existence of solutions to the following equation 
\begin{equation} \label{eq: ss strong} 
-\frac{3+\gamma}{1-\gamma}\Phi\left(  \xi\right)  -\frac{2}{1-\gamma}\xi
\Phi_{\xi}\left(  \xi\right)  +\frac{1}{\int_{0}^{\infty}\eta^{\gamma+\lambda
}\Phi\left(  \eta\right)  \der \eta}\frac{\partial}{\partial\xi}\left(
\xi^{\gamma+\lambda}\Phi\left(  \xi\right)  \right)  =\mathbb{K}\left[
\Phi\right]  \left(  \xi\right) ,\ \ \xi>0.
\end{equation}
We write now precisely what we mean by a solution of equation \eqref{eq: ss strong} and then we state the main theorems on the  existence of a self-similar profile under certain assumptions on the parameters $\gamma $ and $\lambda $ (Theorem \ref{thm:existence}), on its properties (Theorem \ref{thm:properties}) and on the non-existence of the self-similar solutions under different assumptions on the parameters $\gamma $ and $\lambda $ (Theorem \ref{thm:non-existence}).  

\begin{definition} \label{def:self-similar profile}
Let $K$ be a homogeneous symmetric coagulation kernel satisfying
\eqref{kernel bounds}, \eqref{kernel cont},  with homogeneity $\gamma < 1 $ and with $\gamma + \lambda < 1$ and $\gamma + 2 \lambda \geq 1 $. A self-similar profile of equation \eqref{A6} with respect to the kernel $K$, is a measure $\Phi \in \mathcal M_+(\mathbb R_*) $ such that
\begin{equation}\label{minimal boundary cond for ss} 
0 < \int_{\mathbb R_*}x^{\gamma +\lambda}\Phi(\der x) < \infty \text{ and } \int_{(0,1)}x^{1-\lambda}\Phi(\der x) < \infty 
    \end{equation}
    and such that it satisfies the following equation 
\begin{align}\label{weak steady state eq}
&  \int_{\mathbb R_*}  \varphi'(x) \left[ \frac{2}{1-\gamma}x - \frac{x^{\gamma + \lambda}}{\int_{\mathbb R_*}z^{\gamma +\lambda}\Phi(\der z)}  \right] \Phi(\der x) 
 - \frac{1+\gamma}{1-\gamma} \int_{\mathbb R_*}  \varphi(x) \Phi(\der x)\\
& =\frac{1}{2}\int_{\mathbb R_*} \int_{\mathbb R_*} K(x,y) \left[ \varphi(x+y)- \varphi(x) - \varphi(y) \right] \Phi(\der x) \Phi(\der y) \nonumber 
\end{align} 
for every test function $\varphi \in C^1_c(\mathbb R_*).$
\end{definition} 
\begin{remark}
For every test function $\varphi \in C^1_c(\mathbb R_*)$ (hence such that $\varphi=0$ near zero) all the integrals in equation \eqref{weak steady state eq} are finite. The integrals in the left-hand side of equation \eqref{weak steady state eq} are bounded due to the fact that $\varphi$ is compactly supported and that $\Phi$ is a Radon measure. 
We analyse now the right hand side. 
Since $\gamma + 2 \lambda \geq 0 $
\begin{align*} 
&\int_{\mathbb R_*} \int_{\mathbb R_*} K(x,y) \left| \varphi(x+y)- \varphi(x) - \varphi(y) \right| \Phi(\der x) \Phi(\der  y) \\
 &\leq 4 c_1  \int_{\mathbb R_*} \int_{(0, y]} x^{- \lambda }y^{\gamma + \lambda} \left| \varphi(x+y)- \varphi(x) - \varphi(y) \right| \Phi(\der x) \Phi(\der y) \\ 
 &\leq c  \int_{\mathbb R_*} \int_{(0, y]} x^{- \lambda }y^{\gamma + \lambda} \left(\left| \varphi(x+y) - \varphi(y) \right| + | \varphi(x)| \right) \Phi(\der x) \Phi(\der y)  \\
  &\leq c ||\varphi'||_{\infty}\int_{\mathbb R_*} \int_{(0, 1)} x^{1 - \lambda }y^{\gamma+ \lambda}  \Phi(\der x) \Phi(\der y ) \\
  &+ c  \int_{\mathbb R_*} \int_{[1, y]} x^{ - \lambda }y^{\gamma + \lambda} \left| \varphi(x+y) - \varphi(y) \right| \Phi(\der x) \Phi(\der y ) 
  + c  \int_{\mathbb R_*} y^{\gamma + \lambda}  \Phi(\der y) \int_{\mathbb R_*} x^{- \lambda } |\varphi(x)| \Phi(\der x).
\end{align*}
Using \eqref{minimal boundary cond for ss}, the fact that $- \lambda \leq \gamma + \lambda $, as well as the fact that $\varphi$ is compactly supported the desired conclusion follows. 
\end{remark} 
\begin{theorem}[Existence of the self-similar profiles] \label{thm:existence and properties} 
Let $K$ be a homogeneous symmetric coagulation kernel, of homogeneity $\gamma$, satisfying
\eqref{kernel bounds}, \eqref{kernel cont}, with $\gamma , \lambda $ such that  \eqref{avoid_gelation_parameters} holds and such that
\[
-1 < \gamma, \quad \gamma+2\lambda \geq 1.
\] 
Then there exists a self-similar profile $\Phi$ as in Definition \ref{def:self-similar profile}.
Moreover, $\Phi$ is such that $\Phi((0, \rho(M_{\gamma+\lambda})))=0$ for 
\begin{equation}\label{rhoM} 
 \rho(M_{\gamma+\lambda}) :=  \left( \frac{1-\gamma }{2 \int_{\mathbb R_*} x^{\gamma + \lambda } \Phi(\der x ) }\right)^{\frac{1}{1-\gamma -\lambda }}.
\end{equation}

Additionally, $\Phi$ it is such that
\[
\int_{\mathbb R_*} e^{Lx } \Phi(\der x) < \infty 
\] 
for some $L>0$ and it is absolutely continuous with respect to the Lebesgue measure. 
Then $\Phi(\der x )=\phi(x) \der x $ and the density $\phi$ is such that 
\[
\limsup_{x \in \mathbb R_*} \phi(x) e^{M x} < \infty
\]
for a positive constant $M$. 
 \end{theorem} 
 \begin{remark}
In this paper we do not prove the uniqueness of the self-similar profiles. Therefore, it makes sense to understand if the proven properties for the self-similar profile constructed in Theorem \ref{thm:existence} hold for each self-similar profile as in Definition \ref{def:self-similar profile}. 

When $\gamma + 2 \lambda >1 $ we prove that each self-similar profile as in Definition \ref{def:self-similar profile} is zero in the set $(0, \rho(M_{\gamma+\lambda}))$, where $\rho(M_{\gamma+\lambda})$ is given by \eqref{rhoM}, see Theorem \ref{thm:properties} for more details. 
 In contrast, when $\gamma + 2 \lambda =1 $, we only prove that the self-similar profile constructed in the proof of Theorem \ref{thm:existence and properties} is such that $\Phi((0, \rho(M_{\gamma+\lambda})))=0$. 
However, we do not know if this property holds for every self-similar profile as in Definition \ref{def:self-similar profile}. 
 \end{remark}

\begin{theorem}[Non-existence of the self-similar profiles] \label{thm:non-existence} 
Let $K$ be a homogeneous symmetric coagulation kernel, of homogeneity $\gamma$, satisfying
\eqref{kernel bounds}, \eqref{kernel cont}, with $\gamma , \lambda $ such that  \eqref{avoid_gelation_parameters} holds.
\begin{enumerate}
    \item If $ \gamma \leq  -1, \gamma + 2\lambda> 1$, then a self-similar profile $\Phi$ as in Definition \ref{def:self-similar profile} does not exist. 
 \item If $ \gamma \leq  -1, \gamma + 2\lambda=1 $, then a self-similar profile $\Phi$ as in Definition \ref{def:self-similar profile} with the additional property
 \begin{align}\label{-lambda moment}
     \int_{(0,1]} x^{-\lambda} \Phi(\der x ) < \infty 
 \end{align}
 does not exist. 
 \end{enumerate}
\end{theorem}

\begin{remark}\label{rem:non existence gamma+2lambda=1}
Notice that if $\gamma + 2 \lambda > 1$ we prove that self-similar solutions as in Definition \ref{def:self-similar profile} do not exist when $\gamma \leq -1 $. Instead, when $\gamma + 2 \lambda =1 $ and $\gamma \leq -1$, we do not exclude the existence of a self-similar solution $\Phi$ as in Definition \ref{def:self-similar profile} with 
\[
  \int_{(0,1]} x^{-\lambda} \Phi(\der x ) = \infty.
\] 
\end{remark}

\section{Main ideas of the proofs}\label{sec:ideas}
In this section we explain the main ideas for the proofs of existence/non-existence of self-similar solutions. 
Both in the case $\gamma + 2 \lambda >1$ and $\gamma + 2 \lambda =1 $, to prove that a self-similar solution exists, we find an invariant region for the evolution equation corresponding to \eqref{eq: ss strong}, namely the following equation
\begin{align} \label{time dependent eq}
\partial_t \Phi(t, \xi) - \frac{3+\gamma}{1-\gamma}\Phi(t, \xi) - \frac{2\xi}{1-\gamma} \partial_\xi \Phi(t, \xi) + \frac{\partial_\xi \left( \xi^{\gamma+\lambda} \Phi(t, \xi)\right)  }{\int_{\mathbb R_*} x^{\gamma+\lambda} \Phi(t, \der x ) } =\mathbb K [\Phi](t, \xi) .
\end{align} 
By Tychonoff fixed point theorem the existence of an invariant region implies that there exists a solution of equation \eqref{eq: ss strong}.

We prove that the set
\begin{equation*}
P=\left\{ H : \quad \begin{aligned}& \int_0^\infty x H(x) \der x =1,\ \frac{1}{C_{2}} \leq  \int_0^\infty x^{\gamma + \lambda} H (x) \der x \leq C_1, \\
&H((0, \rho(C_1) ))=0, \ \int_0^\infty x^{2-\gamma - \lambda} H (x) \der x \leq C_{2}
\end{aligned} 
\right\} 
\end{equation*} 
is invariant when $\gamma + 2 \lambda \geq 1 $ and $\gamma > -1$ for suitable constants  $C_1, C_2>0 $ and $\rho(C_1)$ given by 
\begin{equation}\label{delta1formula}
\rho(C_1):= \left(\frac{1-\gamma }{2 C_1} \right)^{\frac{1}{1-\gamma -\lambda}}.
\end{equation} 
To prove that the set $P$ is invariant we proceed as follows.
\begin{enumerate}
    \item We prove that $\int_0^\infty \Phi_0(x) \der x =1 $ implies  $\int_0^\infty \Phi(t,x) \der x =1 $, for every $t >0$. This is done multiplying by $x$ equation \eqref{time dependent eq}, integrating from zero to infinity and then studying the ODE for the zeroth moment obtained in this manner. 
    \item As a second step we prove that there exists an upper bound for $M_{\gamma + \lambda}$. 
 To this end we multiply both sides of equation \eqref{eq: ss strong} by $x^{\gamma + \lambda}$ and we integrate over $x$ in $(0, \infty)$ to obtain an ODE for the $\gamma + \lambda $ moment.
Using Gr\"onwall's lemma, the fact that $\gamma + 2 \lambda  \geq 1 $ and that $\gamma > -1$, the desired conclusion follows. 
\item As a third step we use the fact that $M_{\gamma + \lambda} \leq C_1$ to prove that $\Phi(t, (0, \rho(C_{1}))=0.$ 
Indeed, the evolution described by equation \eqref{time dependent eq} is driven by two mechanisms: coagulation, which increases the average size of the particles in the system, and the growth term,
\[
 \frac{\xi^{\gamma + \lambda }}{\int_0^\infty x^{\gamma + \lambda } \Phi(t,x ) \der x } - \frac{2}{1-\gamma} \xi,  
\] 
which is positive for every $\xi < \rho(C_1) $. 
Hence, if we start from an initial data $\Phi_0$ such that $\Phi_0((0,   \rho(C_1) ))=0$, then we will have that $\Phi(t,(0,  \rho(C_1) ))=0$, for every $t >0.$
\item Using the fact that $\gamma + \lambda <1$, $\gamma > -1$, $\Phi(t,(0,  \rho(C_1)))=0$, as well as the fact that $\int_0^\infty x^{\gamma+\lambda} \Phi(t, \der x) \leq C_1  $, we prove that $ \int_0^\infty x^{2-\gamma-\lambda} \Phi(t,x) \der x \leq C_{2}$. 
\item Finally, from the upper bound for $M_{2-\gamma -\lambda}$, we derive a lower bound for $M_{\gamma + \lambda}$. Indeed, Cauchy-Schwarz inequality implies that
\begin{align}\label{CS}
  1= \left(   \int_{\mathbb R_*} x \Phi(t, x) \der x  \right)^2\leq \int_{\mathbb R_*} x^{\gamma+\lambda} \Phi(t,  x) \der x \int_{\mathbb R_*} x^{2-\gamma - \lambda } \Phi(t, x) \der x .
\end{align}
Hence $\frac{1}{C_{2}} \leq M_{\gamma + \lambda}. $
\end{enumerate}

In order to prove non-existence we proceed by contradiction. 
Due to the contribution of the coagulation operator in equation \eqref{coagTransp}, we expect the zeroth moment of $f$ to decay in time. 
However, assuming the self-similar change of variable \eqref{SelfSimForm}, we have that 
\[
\int_0^\infty f(t, x) \der x = t^{- \frac{1+\gamma}{1-\gamma} } \int_0^\infty \Phi(x) \der x. 
\] 
If $\gamma \leq  -1 $, then $\int_0^\infty f(t, x) \der x $ is constant or increasing in time and this gives a contradiction. Hence we cannot expect self-similar solutions to exist. 
To make the argument rigorous we will have to prove that $0< \int_0^\infty \Phi(\der x ) < \infty $. 
When $\gamma + 2 \lambda > 1 $ we do this by proving that for each self-similar profile there exists a $\delta >0$ such that $\Phi((0, \delta))=0$ and that $ \Phi$ tends to zero sufficiently fast as $x \rightarrow \infty $. 
Instead, when $\gamma + 2 \lambda =1 $, the methods used in this paper do not allow to prove that $\Phi$ is equal to zero near the origin, hence we use \eqref{-lambda moment} as well as \eqref{minimal boundary cond for ss} to prove that $0< \int_0^\infty \Phi(\der x ) < \infty $.

\section{Existence of a self-similar profile}\label{sec:exifstence}
We aim to prove the existence of a solution of equation \eqref{time dependent eq}. 
Namely, we will prove the following theorem, which is just a reformulation of Theorem \ref{thm:existence and properties}.
\begin{theorem} \label{thm:existence}
Let $K$ be a homogeneous symmetric coagulation kernel, of homogeneity $\gamma$, satisfying
\eqref{kernel bounds}, \eqref{kernel cont}, with $\gamma , \lambda $ such that  \eqref{avoid_gelation_parameters} holds and such that
\[
-1 < \gamma, \quad \gamma+2\lambda \geq 1.
\] 
Then there exists a self-similar profile $\Phi$ as in Definition \ref{def:self-similar profile}.
Moreover, $\Phi$ is such that $\Phi((0, \rho(M_{\gamma+\lambda})))=0$ with $\rho(M_{\gamma+\lambda})$ given by \eqref{rhoM}. 
\end{theorem}

To this end we prove the existence of a (time-dependent) solution for a suitably  truncated and regularized version of equation \eqref{time dependent eq}:
\begin{align}
\label{truncated time dependent eq}
&\partial_{t} \Phi \left( t, \xi \right)- \frac{3+\gamma}{1-\gamma}\Phi(t, \xi) - \frac{2\xi}{1-\gamma} \partial_\xi \Phi(t, \xi) + \frac{\partial_\xi \left( \xi^{\gamma+\lambda} \Phi(t, \xi)\right)  }{\int_{\mathbb R_*} x^{\gamma+\lambda} \Phi(t, \der x ) } =\mathbb K_R [\Phi](t, \xi).
\end{align}
The operator $\mathbb K_R$ is the truncated operator of parameter $R>0$ defined as 
\begin{align} \label{truncated operator}
\mathbb K_R[\Phi](t,x):= &\frac{1}{2} \int_{0}^{x} K_R\left(  x-y,y\right)
\Phi\left( t, x-y\right)  \Phi\left( t, y\right)  \der y\nonumber\\
&-\int_{0}^{\infty}K_R\left(
x,y\right)  \Phi\left(t,  x\right)  \Phi \left(t,  y\right)  \der y, \end{align}
 where the kernel $K_R$ is a \textit{truncated kernel}, i.e. it is a continuously differentiable, bounded and symmetric function $K_R: \mathbb R_*^2 \rightarrow \mathbb R_+$, such that for every $(x,y ) \in \mathbb R_*^2 $ we have that $K_R(x,y) \leq K(x,y)$ and
 \begin{equation}\label{truncated kernel} 
 \begin{cases}
 K_R(x,y)=0, &\text{ if } x>R \text{ or } y > R; \\
  | K_R(x,y) - K(x,y) |\leq e^{-R}, &\text{ if } (x,y) \in \left[ \frac{1}{2 R}, \frac{R}{4}\right]^2.
 \end{cases} 
 \end{equation} 
Notice that the kernel $K_R$ can be obtained starting from the kernel $K$ by means of standard truncations and mollifying arguments.

Using Tychonoff fixed point theorem we prove the existence of a stationary solution $\Phi_{R} $ for equation \eqref{truncated time dependent eq} and we will prove that there exists the limit $\Phi$ of $\{ \Phi_{R} \}_{R}$ as $R$  tends to infinity. To conclude  we will prove that the measure $\Phi $ satisfies equation \eqref{eq:ss}. 

\subsection{Existence of a time dependent solution for the truncated equation} 
Since in this section we work only with the truncated equation, we omit the label $R$ in $\Phi_{R}$. We will reintroduce the label $R$ in Section \ref{sec:existence self-similar profile}. 
We start this section by introducing a definition of solutions for equation \eqref{truncated time dependent eq}. 
\begin{definition}\label{def:truncated time dep} 
Let $K_R$ be the truncated kernel defined as in \eqref{truncated kernel} as a function of the homogeneous symmetric coagulation kernel $K$ satisfying
\eqref{kernel bounds}, \eqref{kernel cont},  with homogeneity $\gamma < 1 $ and with $\gamma + \lambda < 1$ and $\gamma + 2 \lambda \geq 1 $. A function $\Phi \in C^1([0,T]; \mathcal M_{+,b}(\mathbb R_*))$ is a solution of equation \eqref{truncated time dependent eq} if 
\[
0 < \inf_{  t \in [0, T]  }\int_{\mathbb R_*} x^{\gamma+\lambda } \Phi(t, \der x ) \ \text{ and } \ \sup_{  t \in [0, T]  }\int_{\mathbb R_*} x^{\gamma+\lambda } \Phi(t, \der x ) < \infty 
\] 
and if $\Phi$ satisfies 
\begin{align}\label{weak truncated time dependent eq}
& \int_{\mathbb R_*} \varphi(x) \dot{\Phi}(t, \der x)- \frac{1+\gamma}{1-\gamma} \int_{\mathbb R_*}  \varphi(x) \Phi(t, \der x) +
\frac{2}{1-\gamma} \int_{\mathbb R_*}  \varphi'(x) x \Phi(t, \der x) \\
& - \frac{1}{\int_{\mathbb R_*}x^{\gamma +\lambda}\Phi(t,\der x)}  \int_{\mathbb R_*}  \varphi'(x) x^{\gamma + \lambda}\Phi(t, \der x) \nonumber \\
& =\frac{1}{2}\int_{\mathbb R_*} \int_{\mathbb R_*} K_R(x,y) \left[ \varphi(x+y)- \varphi(x) - \varphi(y) \right] \Phi(t, \der x) \Phi(t, \der y), \nonumber 
\end{align} 
for every test function $\varphi \in C^1_c(\mathbb R_*).$
\end{definition}
Given the positive numbers $k_1, k_2, \rho^*$, we define the subset $S(k_1, k_2, \rho^*) $ of $\mathcal M_+(\mathbb R_*) $ as 
\begin{equation} \label{invariant set} 
  S(k_1, k_2, \rho^*) :=\left\{ H  \in \mathcal M_+(\mathbb R_*): \quad \begin{aligned}& H((0, \rho^*))=0, \int_{\mathbb R_*} xH(\der x ) =1, \\
  &k_1 \leq \int_{\mathbb R_*} x^{\gamma+\lambda} H(\der x ) \leq k_2 
  \end{aligned} \right\}.
\end{equation}
\begin{theorem}\label{thm:time dep truncated}
Assume $K_{R}$ to be a truncated kernel defined as in \eqref{truncated kernel} as a function of a homogeneous symmetric kernel $K$ satisfying \eqref{kernel bounds},  \eqref{kernel cont},  with parameters $\gamma, \lambda \in \mathbb R$ such that \eqref{avoid_gelation_parameters} holds and $\gamma + 2\lambda \geq 1 $ and $\gamma > -1.$ 
Then there exist two constants $k_1, k_2>0$ and a constant $R_0>0$ such that if $\Phi_{0} \in S\left(2 k_1 , \frac{k_2}{2} , \rho(k_2)\right) \cap \left\{ H \in \mathcal M_+(\mathbb R_*): H((4R, \infty) )=0 \right\} $ for $R >R_0$ and for $\rho(k_2)$ given by \eqref{delta1formula},
then there exists a unique solution $\Phi \in C^{1}([0,T];\mathcal M_{+}(\mathbb R_*))$ of equation \eqref{truncated time dependent eq} in the sense of Definition \ref{def:truncated time dep}, for $T>0$ sufficiently small.  
For every $ t \in [0, T]$, we have that $\Phi(t, \cdot ) \in S\left(k_1, k_2, \rho( k_2 )\right) \cap \left\{ H \in \mathcal M_+(\mathbb R_*): H((4R, \infty) )=0 \right\}.$ 
\end{theorem}

\begin{lemma}  \label{lem:characteristics}
Let $T>0$ and let $R>0$. Assume $\gamma < 1 $ and $\gamma + \lambda <1.$
Assume $\alpha  \in C([0, T]) $ satisfying 
\begin{equation} \label{bound alpha}
0 < k_1  \leq \alpha(t)  \leq k_2 < \infty, \ \text{ for all } \ t \in [0, T], 
\end{equation}
for some positive constants $k_1$ and $k_2$. 
Consider the ODE 
\begin{equation} \label{ODE} 
\frac{d x(t) } {dt}=  V(x(t), \alpha(t)), \quad x(0)=x_0 >0,
\end{equation} 
with $V(x, \alpha):=\frac{x^{\gamma +\lambda} }{\alpha}- \frac{2}{1-\gamma} x $. 
If $x_0\geq \rho(k_2)$ with $\rho(k_2)$ given by \eqref{delta1formula},
then \eqref{ODE} has a unique solution $X(t,x_0, \alpha) \geq \rho(k_2)$.
Let $X(t, x_{0,1}, \alpha_1)$ and $X(t, x_{0,2}, \alpha_2)$ be the solutions of equation \eqref{ODE} with respect to the functions $\alpha_1$ and $\alpha_2$ satisfying \eqref{bound alpha} and $x_{0,1},x_{0,2} \in[\rho(k_2),4R]$.
Then 
\begin{equation} \label{bounds for the ODE}
|X(t,x_{0,1}, \alpha_1) - X(t,x_{0,2}, \alpha_2) |\leq L_1 |x_{0,1}-x_{0,2}| +  T L_2 \|\alpha_1-\alpha_2 \|_{[0,T]}, 
\end{equation}
 for every  $t \in [0, T]$, where $L_1=L_1(T, \gamma, \lambda,k_1, k_2, R)>0,$ and $ L_2=L_2(T, \gamma, \lambda,k_1, k_2, R)>0$ and 
 where we denote by $\|\cdot\|_{[0,T]}$ the norm $\| f \|_{[0,T]}:= \sup_{ t \in [0, T]} | f(t) |$, for $f\in C([0,T])$. 
\end{lemma} 
\begin{proof} 
Since $V(x, \alpha)>0$, for every $x \leq \rho(k_2)$, we deduce that the set $\{ x \geq \rho(k_2)\} $ is an invariant region of \eqref{ODE}. This also implies that a unique solution exists. 

%To prove \eqref{bounds for the ODE} we notice that 
% \begin{align*}
%&\left| X(t, x_{0,1} , \alpha_1(t) ) - X(t,  x_{0,2} , \alpha_2(t))  \right| \leq \int_0^t \left| \frac{X^{\gamma + \lambda} (s,  x_{0,1} , \alpha_1(s) )}{ \alpha_1(s) } - \frac{X^{\gamma + \lambda}  (s, x_{0,2} , \alpha_2(s)) } {\alpha_2(s) }  \right| \der s   \\
%&+ \frac{2}{1- \gamma } \int_0^t \left| X(s,x_{0,1} , \alpha_1(s)) - X(s, x_{0,2} , \alpha_2(s))  \right| \der s  \\
% & \leq \frac{1}{k_1^{\alpha_2}  k_1^{\alpha_1} }\int_0^t \left| \alpha_2(s) X^{\gamma + \lambda}(s,x_{0,1} , \alpha_1(s)) - \alpha_1(s) X^{\gamma + \lambda}(s,x_{0,2} , \alpha_2(s))    \right| \der s   \\
% &+ \frac{2}{1- \gamma } \int_0^t \left|  X(s,x_{0,1} , \alpha_1(s)) - X(s, x_{0,2} , \alpha_2(s))   \right| \der s  \\
% &  \leq c_4 \int_0^t  \left|\int_{\mathbb R_*} y^{\gamma + \lambda} \Phi_1(s,\der y) - \Phi_2 (s,\der y) \right| \der s   + c_3 \int_0^t \left| X_{\Phi_1}(s, x ) - X_{\Phi_2}(s, x)  \right| \der s\\
% & \leq c_4 \int_0^t d_{s}(\Phi_{1}, \Phi_{2}) \der s   + c_3 \int_0^t \left| X_{\Phi_1}(s, x ) - X_{\Phi_2}(s, x)  \right| \der s.
%\end{align*} 

Equation \eqref{ODE} is a Bernoulli equation that can be reduced to the linear ODE
\[
\frac{ dy }{ dt} =\frac{ 1-(\gamma+\lambda) }{\alpha(t)} - \frac{2(1-(\gamma+\lambda))}{1-\gamma } y , \quad y_0:= x_0^{1-(\gamma+\lambda)}
\] 
via the change of variable $y=x^{1-(\lambda + \gamma)}$. 
We deduce that 
\[
Y(t,x_0^{1 -(\gamma+\lambda) }, \alpha)= x_0^{1 -(\gamma+\lambda) } e^{- \frac{2(1-(\gamma+\lambda))}{1-\gamma} t } + \int_0^t \frac{(1-\gamma-\lambda)}{\alpha(s) } e^{- \frac{2(1-(\gamma+\lambda))}{1-\gamma} (t-s) } \der s.
\] 
Since $\alpha_1, \alpha_2 $ satisfy \eqref{bound alpha} and since $x_{0,1}, x_{0,2} \in [\rho(k_2), 4R]$, the above formula implies that
\begin{align*}
\left|Y(t,x_{0,1}^{1 -(\gamma+\lambda) }, \alpha_1)-Y(t,x_{0,2}^{1 -(\gamma+\lambda) }, \alpha_2)\right| &\leq \left| x_{0,1}^{1 -(\gamma+\lambda) } - x_{0,2}^{1 -(\gamma+\lambda) }\right| \\
&+ C(T,\gamma,\lambda, k_1) \| \alpha_1 -\alpha_2 \|_{[0, T]}. 
\end{align*}

Since the set $ \{x \geq \rho(k_2)\} $ is invariant for \eqref{ODE}, the above inequality, together with the definition of $y$ as a function of $x$, implies \eqref{bounds for the ODE}.
\end{proof} 

We will solve equation \eqref{truncated time dependent eq} using Lagrangian coordinates. To this end we introduce the notation 
\begin{align*}
\chi^\alpha_{\varphi}(x,y,t):=\varphi(\ell_{\alpha}(x,y,t))- \varphi(x) - \varphi(y), 
\end{align*}
for $\varphi \in C_c(\mathbb R_*)$, where $\ell_\alpha$ is defined for $x >0$, $y>0$ and $t \geq 0$ as 
\begin{align}\label{define l phi}
X(t, \ell_\alpha(x,y,t), \alpha)=X(t,x, \alpha) + X(t,y, \alpha).
\end{align}
The function $\ell_\alpha $ is well defined because for every fixed time $t\geq 0$ the function $x \mapsto X(t,x, \alpha )$ is increasing. 

\begin{lemma}\label{lem:auxiliary time dep truncated} 
Let $K_R$ be a truncated kernel defined as in \eqref{truncated kernel} as a function of a homogeneous symmetric kernel $K$ satisfying \eqref{kernel bounds},  \eqref{kernel cont},  with $\gamma , \lambda \in \mathbb R$ satisfying \eqref{avoid_gelation_parameters}. 
Let $k_1, k_2, R$ be three positive constants. 
Assume that the initial condition $\Phi_0$ is such that $ 2k_1 \leq \int_{\mathbb R_*} x^{\gamma+\lambda} \Phi_0(\der x ) \leq \frac{k_2}{2}$ and that $\Phi_0((0,\rho(k_2)) \cup (4R, \infty ))=0$ for $\rho(k_2)$ given by \eqref{delta1formula}. 

For a sufficiently small time  $T>0$, depending on $R$, 
there exists a function $F \in C([0, T] , \mathcal M_{+}(\mathbb R_*))$, with  $F (0, \cdot)=\Phi_0$, that satisfies  
\begin{align} \label{eq for F}
&\int_{\mathbb R_*} \varphi(x) F(t,  \der x) =  \int_{\mathbb R_*} \varphi(x) \Phi_0( \der x)  \\
&+ \frac{1}{2} \int_0^t \int_{\mathbb R_*} \int_{\mathbb R_*} K_R(X(s,x, \alpha),X(s,y,\alpha))e^{\frac{1+\gamma}{1-\gamma}s}\chi^\alpha_{\varphi}(x,y,s)F(s,  \der x) F(s,  \der y) \der s, \nonumber
\end{align}
 for every $\varphi \in C_c(\mathbb R_*)$, with
 \begin{equation}\label{alfa}
\alpha(t):=e^{\frac{1+\gamma}{1-\gamma} t}\int_{\mathbb R_*} \left(X(t,x, \alpha)\right)^{\gamma+\lambda} F(t, \der x ), \quad \forall t \in [0,T].
 \end{equation} 

The function $F$ is such that $ k_1 \leq \alpha (t) \leq k_2$ and such that
\begin{equation} \label{support FPhi}
F(t,(0,\rho(k_2) )\cup (4R, \infty ))=0, \quad  \forall t \in[0,T].
\end{equation} 
\end{lemma}
\begin{proof}
We define the set 
\[ 
\tilde{X}:= \left\{ H \in \mathcal M_{+}(\mathbb R_*) : H((0,\rho(k_2)) \cup (4R, \infty))=0\right\}
\]
and the set 
\[
X_T:=\left\{ (F,\alpha) \in C([0,T]; \mathcal M_{+}(\mathbb R_*)) \times C([0,T]) :
\begin{aligned}\textup{ } & F(t,\cdot) \in \tilde{X},  k_1 \leq  \alpha(t) \leq k_2,  \forall t \in [0, T],\\
&\sup_{t\in[0,T]}\int_{\mathbb{R}_{\ast}}F(t,\der x)\leq 1+\int_{\mathbb{R}_{\ast}}\Phi_{0}(\der x)
\end{aligned} 
\right\}.
\]

We endow the set $\mathcal M_{+}(\mathbb R_*) $ with the Wasserstein metric $W_{1}$. 
The reason for this choice will become clear in the proof of the contractivity of the evolution operator corresponding to \eqref{eq for F}, \eqref{alfa}. 
We endow the set $C([0,T], \mathcal M_+(\mathbb R_*))$ with the metric induced by the distance 
\[
d_T(\mu, \nu):=\sup_{t \in [0,T] } W_1(\mu(t, \cdot) , \nu(t, \cdot)).
\] 
Similarly, we endow $C([0,T])$ with the norm $\| f \|_{[0,T]}:= \sup_{ t \in [0, T]} | f(t) |$. 

For every $(F, \alpha) \in X_T$, we define the operator $\mathcal T[F,\alpha](t): C_c(\mathbb R_*) \rightarrow \mathbb R_{+}$ as $\mathcal T [F,\alpha](t):= \mathcal T_1[F,\alpha](t)+\mathcal T_2[F,\alpha](t)$  with 
\[
\langle \mathcal T_1[F,\alpha](t), \varphi \rangle=  \int_{\mathbb R_*}
\varphi (x) e^{-\int_{0}^{t}a[F,\alpha](s,x)\der s} \Phi_0( \der x),
\] 
\[
a[F,\alpha](t,x):=e^{\frac{\gamma +1 }{1-\gamma }t} \int_0^\infty K_R(X(t,x,\alpha), X(t,y,\alpha)) F(t, \der y),
\]
\begin{align*}
 \langle \mathcal  T_2[F,\alpha](t), \varphi \rangle=&\frac{1}{2}  \int_0^t \int_{\mathbb R_*} \int_{\mathbb R_*} e^{- \int_s^t a[F](v,x) dv}  K_R(X(s,x,\alpha),X(s,y, \alpha ))e^{\frac{1+\gamma}{1-\gamma}s}\varphi(\ell_\alpha(x,y,s)) \cdot \\
 & \cdot F(s, \der x) F(s, \der y)\der s. 
\end{align*}
Moreover, given $(F, \alpha) \in X_T $, we define the operator $\mathcal A[F,\alpha]: [0, T] \rightarrow \mathbb R_*$ as 
\[
\mathcal A[F,\alpha](t) := e^{\frac{1+\gamma}{1-\gamma} t}\int_{\mathbb R_*} \left(X(t,x, \alpha)\right)^{\gamma+\lambda} F(t, \der x ), \quad \forall t \in [0,T].
\] 
We can now rewrite \eqref{eq for F}, \eqref{alfa} in a fixed point form, namely as $(F,\alpha)=\mathcal F[F, \alpha]$ where
$\mathcal F[F,\alpha]:= (\mathcal T[F,\alpha], \mathcal A[\mathcal T[F,\alpha] , \alpha ])$.

We prove now that $\mathcal F: X_T \rightarrow X_T$. 
Since for every $(F,\alpha) \in X_T$, the operator $\mathcal T[F,\alpha](t)$ is linear and continuous, we deduce that it can be identified with an element of $\mathcal M_{+, b} (\mathbb R_*) .$ Moreover, the operator $t \mapsto \mathcal T[F,\alpha](t)$ is a continuous map from $\mathbb R_{+}$ to $\mathcal M_{+, b} (\mathbb R_*)$, hence $\mathcal T [F,\alpha] \in C([0, T], \mathcal M_{+,b}(\mathbb R_*))$. 
Similarly, for every $(F,\alpha) \in X_T $, we have that $\mathcal A[F, \alpha ] \in C([0,T])$. 

Consider a test function $\varphi $ such that $\varphi(x)=0$ for every $ x \in [\rho(k_2), \infty)$. Then, since $\Phi_0((0,\rho(k_2)))=0$, we deduce that
\[
\langle \mathcal T_1[F,\alpha](t), \varphi \rangle= \int_{[\rho(k_2), \infty)}\varphi(x) e^{-\int_{0}^{t}a[F,\alpha](s,x)\der s} \Phi_0(\der x) =0. 
\] 
Similarly, since $\ell_\alpha(s,x,y) \geq x $ and since $(F,\alpha) \in X_T$, we deduce that
\begin{align*}
 \langle \mathcal T_2[F,\alpha](t), \varphi \rangle \leq \frac{C(R)}{2} \int_0^t \int_{[\rho(k_2), \infty)} \int_{[\rho(k_2), \infty)} e^{\frac{1+\gamma}{1-\gamma}s}\varphi(\ell_\alpha(x,y,s))F(s, \der x) F(s, \der y)\der s=0.
\end{align*}
Therefore, $(\mathcal T[F,\alpha])(t,(0,\rho(k_2)))=0$, for every $t \in [0, T].$

Consider any test function $\varphi$ such that $\varphi(x)=0$ in $[0,4R]$. Since $\Phi_0((4R, \infty))=0$, we have that
\[
\langle \mathcal T_1[F,\alpha](t), \varphi \rangle  \leq 
\int_{(0, 4 R]} \varphi(x)  e^{-\int_{0}^{t}a[F,\alpha](s,x)\der s}  \Phi_0(\der x)
=0.
\] 
Moreover, using the notation $S_1(s):=\{ (x,y) \in \mathbb R_*^2  : X(s,x,\alpha) \leq R, X(s,y,\alpha) \leq R  \} $,
$S_2(s):=\{ (x,y) \in \mathbb R_*^2  : X(s,x,\alpha) > R, X(s,y,\alpha) \leq R  \} $ and $S_3(s):=\{ (x,y) \in \mathbb R_*^2  : X(s,x,\alpha) > R, X(s,y,\alpha) > R  \} $, we have
\begin{align*}
 & \langle \mathcal  T_2[F,\alpha](t), \varphi \rangle \leq \\
 &\leq\frac{1}{2} \int_0^t  \iint_{S_1(s)} e^{\frac{1+\gamma}{1-\gamma}s}
  K_R(X(s,x,\alpha),X(s,y,\alpha))\varphi(\ell_\alpha(x,y,s))F(s, \der x) F(s,\der y)\der s\\
    &+ \frac{1}{2} \int_0^t  \iint_{S_2(s)} e^{\frac{1+\gamma}{1-\gamma}s}
  K_R(X(s,x,\alpha),X(s,y,\alpha))\varphi(\ell_\alpha(x,y,s))F(s, \der x) F(s,\der y)\der s \\ 
  &+ \frac{1}{2} \int_0^t \iint_{S_3(s)} e^{\frac{1+\gamma}{1-\gamma}s}
  K_R(X(s,x,\alpha),X(s,y,\alpha))\varphi(\ell_\alpha(x,y,s))F(s, \der x) F(s,\der y)\der s.
\end{align*}
For $x,y$ such that $X(s,x,\alpha) \geq R$ or $X(s,y,\alpha) \geq R$, we have that $K_R(X(s,x,\alpha),X(s,y,\alpha))=0$. Thus, the second and the third terms above are equal to zero.

 To see that the first term is equal to zero, notice that $X(t, x, \alpha) \geq x e^{- \frac{2}{1-\gamma } t}$. We can thus select $T$ sufficiently small so that $X(t, x, \alpha) \geq \frac{x}{2}$. 
This implies that 
 $ \frac{1}{2} \ell_\alpha(x,y,s) \leq  X(s, \ell_\alpha(x,y,s), \alpha )= X(s, x, \alpha ) + X(s, y, \alpha ) \leq 2 R  $.
 Hence $\ell_\alpha(x,y,s) \leq 4R$. Since $\varphi(x)=0$ for every $x \leq 4 R$, the desired conclusion follows.
 
We now prove that $k_1 \leq \mathcal A[ \mathcal T[F,\alpha] ,\alpha](t) \leq k_2$. To this end notice that 
\begin{align*}
  \left| \mathcal A[\mathcal T[F,\alpha],\alpha](t)  - \int_{\mathbb R_*} x^{\gamma+\lambda} \Phi_0(\der x) \right| 
  & \leq \int_{\mathbb R_*} \left| e^{\frac{1+\gamma}{1-\gamma } t }  (X(t,x,\alpha))^{\gamma+\lambda}-x^{\gamma+\lambda} \right| \mathcal T[F,\alpha](t, \der x) \\
  &+  \int_{\mathbb R_*} x^{\gamma+\lambda} \left| \mathcal T[F,\alpha](t, \der x)  -\Phi_0(\der x) \right| .
\end{align*}
Notice that
\begin{align*}
\left| e^{\frac{1+\gamma}{1-\gamma } t }  (X(t,x,\alpha))^{\gamma+\lambda}-x^{\gamma+\lambda} \right| &\leq  e^{\frac{1+\gamma}{1-\gamma } t } \left| (X(t,x,\alpha))^{\gamma+\lambda}-x^{\gamma+\lambda} \right| + \frac{1+\gamma}{1-\gamma } T x^{\gamma+\lambda}\\
&\leq T C(T, R, k_2) .
\end{align*}
Finally, we have that the term 
\begin{align*}
     \int_{\mathbb R_*} x^{\gamma+\lambda} \left| \mathcal T[F,\alpha](t, \der x)  - \Phi_{0}( \der x)  \right| 
\end{align*}
can be arbitrarily small by taking $T$ small. This is due to the definition of the map $T[F,\alpha]$ and the upper bound for $F $ in the definition of the space $X_{T}$.

We deduce that, for small time $T$, we have that $ \left| \mathcal A[F,\alpha](t) - \int_{\mathbb R_*} x^{\gamma+\lambda} \Phi_0(\der x ) \right| \leq  C(T),$ for every $t \in [0, T] $, where the constant $C(T) $ can be made arbitrarily small. 

We now check that the map $\mathcal F$ is a contraction. 
To this end, we use the fact that $|e^{-x_{1}}-e^{-x_{2}}|\leq|x_{1}-x_{2}|,$ for $x_{1},x_{2}\geq 0.$
Consider $(F,\alpha), (G,\beta) \in X_{T}$. 
Then, for every Lipschitz function $\varphi$ with $||\varphi||_{\textup{Lip}}\leq 1$, we have that
\begin{align*}
\bigg|\int_{\mathbb R_*} \varphi(x) [\mathcal T_1 [F,\alpha ]-\mathcal T_{1}[G,\beta]](t,\der x)\bigg|&\leq \int_{\mathbb R_*} \int_0^t \left| \varphi(x)  \left[a[F,\alpha](s,x)- a[G,\beta ](s,x) \right] \right| \der s \Phi_0(\der x) \\
&\leq \int_{\mathbb R_*} \int_0^t \left| \varphi(x)  \left[ \int_{\mathbb R_*}  K_R(X(s,x,\alpha),X(s,y,\alpha)) F(s, \der y)  \right. \right. \\
& -\left.\left.  \int_{\mathbb R_*}  K_R(X(s,x,\beta),X(s,y,\beta )) G(s, \der y) \right] \der s \Phi_0(\der x)  \right|\\
&\leq I_1[\varphi] + I_2[\varphi],
\end{align*}
where 
\begin{align*}
I_1[\varphi] &:=\int_{\mathbb R_*} \int_0^t \bigg| \varphi(x) \int_{\mathbb R_*} [ K_R(X(s,x,\alpha),X(s,y,\alpha))  \\
&  - K_R(X(s,x,\beta),X(s,y,\beta )) ]  F(s, \der y) \bigg| \der s \Phi_0(\der x)
\end{align*} 
and
\begin{align*} 
I_2[\varphi]:=  \int_{\mathbb R_*} \int_0^t \left| \varphi(x)  \int_{\mathbb R_*}  K_R(X(s,x,\beta),X(s,y,\beta )) \left( F(s, \der y) - G(s, \der y) \right) \der s \Phi_0(\der x ) \right|.
\end{align*} 
Then, the differentiability of the kernel $K_R$, together with inequality \eqref{bounds for the ODE}, implies that 
\begin{align*} 
&\sup_{||\varphi||_{\textup{Lip}\leq 1}} I_1[\varphi] \leq C(R) T \|\Phi_0\|_{TV} (1+ \|\Phi_0\|_{TV}) \|\alpha-\beta\|_{[0,T]}. 
\end{align*} 
Similarly, 
\begin{align*} 
&\sup_{||\varphi||_{\textup{Lip}\leq 1}} I_2[\varphi] \leq C(R) T \|\Phi_0\|_{TV} d_T(F,G). 
\end{align*} 

Using again the differentiability of the kernel $K_R$, inequality \eqref{bounds for the ODE} and the definition of $\ell_\alpha$, we obtain that
\begin{align*}
\sup_{||\varphi||_{\textup{Lip}\leq 1}}  \bigg|\int_{\mathbb R_*}  \varphi(x)[ \mathcal T_2 [F,\alpha](t,\der x)-\mathcal T_{2}[G,\beta](t,\der x)]\bigg| &\leq C(T,R)(1+||\Phi_{0}||_{TV})^{2}  d_T(F,G) \\
&+C(T,R)(1+||\Phi_{0}||_{TV})^{2}   \| \alpha -\beta \|_{[0,T]} ,
\end{align*}
where $T$ can be selected so that $C(T,R)(1+||\Phi_{0}||_{TV})^{2}  \leq \frac{1}{4} $. 
 
 We remark that, in order to obtain this bound, we used the properties of the Wasserstain distance. Similar estimates cannot be obtained by endowing $\mathcal M_+(\mathbb R_*)$ with the total variation norm. 
 
To understand this better, let $\alpha,\beta\in C([0,T])$, with $k_{1}\leq \alpha,\beta\leq k_{2}$. Let $s\in[0,T]$. Using the fact that $||\varphi||_{\textup{Lip}}\leq 1$ and that we can assume we work on a sufficiently nice compact set due to the support of the measures, we can derive the following estimate
\begin{align*}
|\varphi(\ell_\alpha(x,y,s))-\varphi(\ell_\beta(x,y,s))|&\leq |\ell_\alpha(x,y,s)-\ell_\beta(x,y,s)|\\
&\lesssim  |X(s,\ell_\alpha(x,y,s),\alpha)-X(s,\ell_\beta(x,y,s),\alpha)|\\
&\leq|X(s,\ell_\alpha(x,y,s),\alpha)-X(s,\ell_\beta(x,y,s),\beta)|\\
&+|X(s,\ell_\beta(x,y,s),\beta)-X(s,\ell_\beta(x,y,s),\alpha)|.
\end{align*}
Using (\ref{bounds for the ODE}) and (\ref{define l phi}), we have
\begin{align*}
  |X(s,\ell_\alpha(x,y,s),\alpha)-X(s,\ell_\beta(x,y,s),\beta)|&   \leq  |X(s,x,\alpha)-X(s,x,\beta)|\\
  &+ |X(s,y,\alpha)-X(s,y,\beta)|\\
 &\leq 2TL_{2}||\alpha-\beta||_{[0,T]}
\end{align*}
and thus
\begin{align}\label{function lip estimate}
|\varphi(\ell_\alpha(x,y,s))-\varphi(\ell_\beta(x,y,s))|\lesssim  T ||\alpha-\beta||_{[0,T]}.
\end{align}
More precisely, inequality (\ref{function lip estimate}) was needed to prove the contractivity of the map $\mathcal{T}_{2}$.
 
In order to prove the upper bound for the map $T[F,\alpha],$ we first observe that
\begin{align*}
\langle \mathcal T_1[F,\alpha](t), 1 \rangle= \int_{[\rho(k_2), \infty)} e^{-\int_{0}^{t}a[F,\alpha](s,x)\der s} \Phi_0(\der x) \leq ||\Phi_{0}||_{TV}
\end{align*}
and then that
\begin{align*}
 \langle \mathcal T_2[F,\alpha](t), 1 \rangle &\leq \frac{C(R)}{2}e^{\frac{1+\gamma}{1-\gamma}t} \int_0^t \int_{[\rho(k_2), \infty)} \int_{[\rho(k_2), \infty)} F(s, \der x) F(s, \der y)\der s\\
 &\leq \frac{C(R)}{2}e^{\frac{1+\gamma}{1-\gamma}T}T(1+||\Phi_{0}||_{TV})^{2}\leq 1,
\end{align*}
for $T$ sufficiently small. Thus,
\begin{align*}
\sup_{t\in[0,T]}\langle \mathcal T[F,\alpha](t), 1 \rangle\leq 1+||\Phi_{0}||_{TV}.
\end{align*}
Finally, 
\begin{align*}
  \|  \mathcal A [\mathcal T[F,\alpha], \alpha ]-   \mathcal A [\mathcal T[G,\beta], \beta] \|_{[0,T]} &\leq  C(R) d_T([\mathcal T[F,\alpha],\mathcal T[G,\beta])\\
  &+ C_1(T,R)T(1+||\Phi_{0}||_{TV}) \|\alpha - \beta \|_{[0,T]} \\
  &\leq C(T,R)T(1+||\Phi_{0}||_{TV})^{2} [d_T(F,G)+\|\alpha - \beta \|_{[0,T]}],
\end{align*}
where the constant $TC(T,R)(1+||\Phi_{0}||_{TV})^{2}$ can be made smaller than $\frac{1}{4}$ by selecting $T$ small. 
The operator $\mathcal F$ is therefore a  contraction and, by Banach fixed point theorem, we conclude that it has a unique fixed point $(F, \alpha ) \in X_T$. 
\end{proof}

\begin{proof}[Proof of Theorem \ref{thm:time dep truncated}]
Assume $T>0$ to be as in Lemma \ref{lem:auxiliary time dep truncated}. 
Let $(F, \alpha ) \in C([0,T], \mathcal M_+(\mathbb R_*)) $ $\times C([0,T])$ be the solution of \eqref{eq for F}, \eqref{alfa}. 
Let $X(t,y, \alpha)$ be the solution of equation \eqref{ODE}. 
Let $ \Phi \in C^1([0,T], \mathcal M_+(\mathbb R_*))$ be the function defined by duality as 
\begin{align}\label{phi function of F}
\int_{\mathbb R_*} \varphi(x) \Phi(t, \der x)=\int_{\mathbb R_*} \varphi(X(t,y, \alpha))  e^{\frac{1+\gamma}{1-\gamma }t} F(t, \der y),
\end{align} 
for every $\varphi \in C_c(\mathbb R_*)$. 
We prove that the function $\Phi$ is such that
\begin{align}\label{support truncated time solution}
\Phi (t,(0, \rho(k_2)) \cup (4R, \infty ))=0, \quad \forall t \in [0,T],
\end{align}
for $\rho(k_2)$ defined as in (\ref{delta1formula}).
To this end notice that, since $F(t,(0,\rho(k_{2}))=0$, for every $t\in[0,T]$, and since $ X(t, y, \alpha) \geq \rho(k_2)$, for every $y \geq \rho(k_2)$,  we have that 
\[
\int_{\mathbb R_* } \varphi(x) \Phi(t, \der x) =\int_{[\rho(k_2),\infty)} \varphi(X(t,y, \alpha))  e^{\frac{1+\gamma}{1-\gamma }t} F(t, \der y)=0,
\]
for every test function $\varphi$ such that $\varphi(x)=0$ if $x \geq \rho(k_2)$. 
Similarly,
since $t \mapsto X(t, x, \alpha) $ is a decreasing function for large values of $x$, we deduce that $ X(t, x, \alpha)\leq 4R$, for every $x \leq 4R $. Hence, 
\[
\int_{\mathbb R_* } \varphi(x) \Phi(t, \der x) =\int_{[\rho(k_2), 4 R]} \varphi(X(t,y, \alpha))  e^{\frac{1+\gamma}{1-\gamma }t} F(t, \der y)=0,
\]
for every test function $\varphi$ such that $\varphi(x)=0$ if $x \leq 4 R$. 

We now prove that $\Phi$ satisfies \eqref{weak truncated time dependent eq}. 
By its definition as the fixed point of the operator $\mathcal F$, we deduce that $F \in C^1([0, T] , \mathcal M_+(\mathbb R_*))$. 
Differentiating both sides of the equality \eqref{phi function of F} in time, we deduce that 
\begin{align*}
 \frac{\der}{\der t} \int_{\mathbb R_*} \varphi(x) \Phi(t, \der x)& =   \int_{\mathbb R_*}  \varphi(X(t, y, \alpha)) e^{\frac{1+\gamma}{1-\gamma} t } \partial_t F(t,\der y)  \\
&+\int_{\mathbb R_*}  \varphi'(X(t, y,\alpha)) \bigg[ \frac{(X(t,y,\alpha))^{\gamma+\lambda}}{\alpha(t)}  -\frac{2}{1-\gamma} X(t,y,\alpha) \bigg] e^{\frac{1+\gamma}{1-\gamma} t }F(t, \der y)  \\
&+ \frac{1+\gamma}{1-\gamma}\int_{\mathbb R_*}  \varphi(X(t, y, \alpha)) e^{\frac{1+\gamma}{1-\gamma} t } F(t,\der y).
\end{align*}
Using the fact that $F$ satisfies equation \eqref{eq for F}, we deduce that 
\begin{align*}
   \int_{\mathbb R_*}   \varphi(X(t, y, \alpha))& e^{\frac{1+\gamma}{1-\gamma} t } \partial_t F(t,\der y) \\
   &= \frac{1}{2}\int_{\mathbb R_*} \int_{\mathbb R_*}e^{ \frac{2(1+\gamma)}{1-\gamma}t} F(t,\der x) F(t,\der y)K_R(X(t,x,\alpha),X(t,y,\alpha))\\
&\times[ \varphi(X(t, x,\alpha)+X(t, y,\alpha ))- \varphi(X(t,y,\alpha ))- \varphi(X(t, x,\alpha))].
\end{align*}
Hence, using again the equality \eqref{phi function of F}, we have that
\begin{align*}
 \frac{\der}{\der t} \int_{\mathbb R_*} \varphi(x) \Phi(t,\der x)&= \int_{\mathbb R_*}  \varphi'(X(t, y,\alpha)) \left[   \frac{(X(t,y,\alpha))^{\gamma+\lambda}}{\alpha(t)}  -\frac{2}{1-\gamma} X(t,y,\alpha )\right] e^{\frac{1+\gamma}{1-\gamma} t } F(t,\der y)  \\
&+ \frac{1+\gamma}{1-\gamma} \int_{\mathbb R_*}  \varphi(X(t, y,\alpha)) e^{\frac{1+\gamma}{1-\gamma} t } F(t,\der y)  \\
&+  \frac{1}{2}\int_{\mathbb R_*} \int_{\mathbb R_*} K_R(X(t,x,\alpha),X(t,y,\alpha)) e^{ \frac{2(1+\gamma)}{1-\gamma}t}F(t,\der x) F(t,\der y) \\
&\cdot\left[ \varphi(X(t, x,\alpha)+X(t, y,\alpha))- \varphi(X(t,y,\alpha))- \varphi(X(t, x,\alpha)) \right]  \\  &= \int_{\mathbb R_*}  \varphi'(y) \left[\frac{y^{\gamma+\lambda}}{M_{\gamma+\lambda}(\Phi(t))}   -\frac{2}{1-\gamma} y\right] \Phi(t,\der y)+ \frac{1+\gamma}{1-\gamma} \int_{\mathbb R_*}  \varphi(y) \Phi(t,\der y)  \\
&+  \frac{1}{2}\int_{\mathbb R_*} \int_{\mathbb R_*} K_R(x ,y)\left[ \varphi( x+y)- \varphi(y)- \varphi(x) \right] 
\Phi(t,\der x)\Phi(t,\der y).
\end{align*} 
We conclude that $\Phi$ satisfies \eqref{weak truncated time dependent eq}. 
\end{proof}

\subsection{Existence of a stationary solution for the truncated equation} 
\begin{theorem}\label{thm:stationary truncated}
Assume $K_R $ to be a truncated kernel as in \eqref{truncated kernel} defined as a function of the homogeneous symmetric kernel $K$ that satisfies \eqref{kernel bounds},  \eqref{kernel cont},  for parameters $\gamma, \lambda \in \mathbb R$ such that \eqref{avoid_gelation_parameters} holds and such that $\gamma + 2\lambda \geq  1$, $ \gamma >-1 $. 
There exists $\overline R>0$ such that, for every truncation parameter $R> \overline R$, there exists a $\Phi \in \mathcal M_{+}(\mathbb R_*)$ satisfying the equation
\begin{align}\label{weak truncated steady state eq}
&  \int_{\mathbb R_*}  \varphi'(x) \left[ \frac{2}{1-\gamma}x - \frac{x^{\gamma + \lambda}}{\int_{\mathbb R_*}z^{\gamma +\lambda}\Phi(\der z)}  \right] \Phi(\der x) 
 - \frac{1+\gamma}{1-\gamma} \int_{\mathbb R_*}  \varphi(x) \Phi(\der x)\\
& =\frac{1}{2}\int_{\mathbb R_*} \int_{\mathbb R_*} K_R(x,y) \left[ \varphi(x+y)- \varphi(x) - \varphi(y) \right] \Phi( \der x) \Phi(\der y), \nonumber 
\end{align} 
for every $\varphi \in C^{1}_c(\mathbb R_*)$. 
The solution is such that
\begin{equation} \label{uniform bounds self-sol}
 \int_{\mathbb R_*} x \Phi(\der x) = 1,  \quad  \int_{\mathbb R_*} x^{\gamma +\lambda} \Phi(\der x) \leq C_{1} \quad \textup{and} \quad\int_{\mathbb R_*} x^{2-\gamma-\lambda} \Phi(\der x) \leq C_{2},
\end{equation} 
 for some constants $C_{1}, C_{2}>0$ that do not depend on the truncation $R$. Additionally, $\Phi$ is such that $\Phi((0, \rho(C_1)))=0$ where $\rho(C_1)$ is given by \eqref{delta1formula}.
\end{theorem}
To prove the theorem we introduce the semigroup $\{S(t)\}_{\{ t \geq 0\}}$ defined as 
$S(0)\Phi_0=\Phi_0$ and $S(t) \Phi_0 = \Phi(t, \cdot)$, where $\Phi$ is the solution of equation \eqref{weak truncated time dependent eq} constructed in the previous section.
We want to apply Tychonoff fixed point theorem to prove the existence of a stationary solution for equation \eqref{truncated time dependent eq}.
To this end we need to find an invariant region for $S(t)$ (Proposition \ref{prop:inv region}) and prove the continuity in the weak-$^{\ast}$ topology of the map $\Phi \mapsto S(t)\Phi$ for every time $t$ (Proposition \ref{prop:weak star cont}).
In this way we prove that for every time $t$ there exists a fixed point $\hat{\Phi}_t$ such that $S(t)\hat{\Phi}_t = \hat{\Phi}_t.$
We will then prove that $S$ is continuous (Proposition \ref{prop:time cont}) in time and conclude that the limit as $t$ tends to zero of $\hat{\Phi}_t $ is a solution of equation \eqref{weak truncated steady state eq}. This is a standard method used in the study of coagulation equations to prove existence of self-similar profiles. For more details, see \cite[Theorem 1.2]{escobedo2005self}.

\begin{proposition}[Invariant region] \label{prop:inv region}
Assume $K_R $ to be a truncated kernel as in \eqref{truncated kernel} defined as a function of the homogeneous symmetric kernel $K$ that satisfies \eqref{kernel bounds},  \eqref{kernel cont},  for parameters $\gamma, \lambda \in \mathbb R$ such that \eqref{avoid_gelation_parameters} holds and such that $\gamma + 2\lambda \geq  1$, $\gamma >-1 $. Then there exist some positive constants $C_1, C_2$ that do not depend on the truncation $R$ such that the set
\begin{align} \label{invariant region}
P:=\left\{ 
H \in \mathcal M_{+}(\mathbb R_*) : \quad \begin{aligned}   & M_1(H )=1,\ 
H((0, \rho(C_1)) \cup (4R,\infty) )=0, \\
& \frac{1}{C_{2}} \leq M_{\gamma + \lambda}(H) \leq C_1
 \end{aligned} 
 \right\}
\end{align} 
is invariant under the evolution operator $S(t),$ where $\rho(C_1)$ was defined in \eqref{delta1formula}. 
\end{proposition}
\begin{proof}

Let $\Phi_0$ and $K_R$ be as in Lemma \ref{lem:auxiliary time dep truncated}. 
Let $T$ be as in Lemma \ref{lem:auxiliary time dep truncated}. We prove that there exist two constants $C_{1},C_{2}>0$ that do not depend on the truncation parameter $R$ such that the solution $\Phi$ obtained in Theorem \ref{thm:time dep truncated} is such that for every time $t \in [0, T] $ 
\begin{itemize}
        \item[(1)]
$\int_{\mathbb R_* } x \Phi(t,\der x) = 1$; 
    \item[(2)]$\int_{\mathbb R_* } x^{\gamma + \lambda } \Phi(t, \der x) \leq \max\{ C_1,\int_{\mathbb R_* } x^{\gamma + \lambda } \Phi_{0}( \der x)\}$; 
    \item[(3)] $\int_{\mathbb R_* } x^{2- \gamma - \lambda } \Phi(t, \der x) \leq  \max\{C_{2},\int_{\mathbb R_* } x^{2-\gamma-\lambda} \Phi_{0}( \der x)\}$.
\end{itemize}
Additionally, we prove that we can conclude from (2) that $\Phi(t,(0,\rho(C_{1}))=0$, for every $t\in[0,T]$.

We first notice that, since $\Phi(t, \cdot)$ has compact support, we can consider in equation \eqref{weak truncated time dependent eq} test functions $\varphi$ such that $\varphi(x)=x^{k} $, for $k\in\mathbb{R}.$

Let us prove (1). We consider in equation \eqref{weak truncated time dependent eq} a test function $\varphi$ such that $\varphi(x)=x $.
We deduce that 
\[
\frac{d}{dt} \int_0^\infty \xi \Phi(t, \der \xi)=1- \int_0^\infty \xi \Phi(t,\der \xi).
\]
Hence 
\[
\int_0^\infty \xi \Phi(t,\xi) d \xi = 1+ \left( \int_{\mathbb R_*} \xi \Phi_0(\der\xi) -1 \right) e^{-t} 
\]
which since $\int_{\mathbb R_*} \xi \Phi_0(\der \xi)=1 $ implies 
\begin{equation} \label{lower bound mass truncated}
 \int_{\mathbb R_*}\xi \Phi(t, \der \xi) = 1.
\end{equation}

We prove (2). 
We start by proving this for $\gamma + 2 \lambda >1. $
Consider a test function $\varphi $ such that $\varphi(x)=x^{\gamma +\lambda}$
in equation \eqref{weak truncated time dependent eq}. 
Then, we can see that the moment $M_{\gamma + \lambda}$ satisfies the following 
\begin{align*}
& \partial_{t}M_{\gamma +\lambda} \leq - \frac{2\lambda + \gamma -1 }{1-\gamma} M_{\gamma+\lambda}
+ \frac{(\gamma +\lambda)}{M_{\gamma+\lambda}}  M_{2(\gamma+\lambda)-1} \\
&+ c_{3}\int_{[\rho(2C_{1}), \frac{R}{4}] }
 \int_{[\rho(2C_{1}), \frac{R}{4} ] } (x^{\gamma + \lambda} y^{-\lambda} + x^{-\lambda} y^{\gamma + \lambda} ) [ (x+y)^{\gamma + \lambda} - x^{\gamma + \lambda} - y^{\gamma + \lambda}]  \Phi(t, \der x ) \Phi(t, \der y ) 
 \end{align*} 
since, due to the definition of $K_{R}$ in (\ref{truncated kernel}), we have that on the set $[\rho(2C_{1}),\frac{R}{4}]^{2}$, there exists a constant $C>0$ such that $K_{R}\geq\frac{K}{C}$. 

We denote by $z:=\frac{y}{x}$. Notice that since $\gamma + \lambda >0$ when $y \leq x $ we have that 
\begin{align*}
 \left[x^{\gamma + \lambda} y^{-\lambda} + x^{-\lambda} y^{\gamma + \lambda} \right] \left[ (x+y)^{\gamma + \lambda} - x^{\gamma + \lambda} - y^{\gamma + \lambda} \right]& \leq 2  x^{2(\gamma + \lambda)} y^{-\lambda} [ (1+z)^{\gamma + \lambda} - 1- z^{\gamma + \lambda}  ]\\
& \leq 2 x^{2(\gamma + \lambda)} y^{-\lambda} \left( \left(\gamma + \lambda \right)z- z^{\gamma + \lambda} \right)\\
&\leq 2 (\gamma + \lambda -1 ) y^{\gamma} x^{\gamma + \lambda}. 
\end{align*}
As a consequence, by symmetry, we deduce that 
 \begin{align} \label{epsilon appears}
&\partial_{t}M_{\gamma +\lambda}  \lesssim - \frac{2\lambda + \gamma -1 }{1-\gamma} M_{\gamma+\lambda}
+ \frac{(\gamma +\lambda)}{M_{\gamma+\lambda}}  M_{2(\gamma+\lambda)-1} - 2c_{3} \left(1-\gamma-\lambda \right)
 M_{\gamma } M_{\gamma+ \lambda}\nonumber\\
 &+2c_{3}(1-\gamma-\lambda)\int_{(\frac{R}{4},\infty)}x^{\gamma}\Phi(\der x)M_{\gamma+\lambda}+2c_{3}(1-\gamma-\lambda)\int_{(\frac{R}{4},\infty)}x^{\gamma+\lambda}\Phi(\der x)M_{\gamma}\nonumber\\
 &\leq - \frac{2\lambda + \gamma -1 }{1-\gamma} M_{\gamma+\lambda}
+ \frac{(\gamma +\lambda)}{M_{\gamma+\lambda}}  M_{2(\gamma+\lambda)-1} -2c_{3} \left(1-\gamma-\lambda \right)
 M_{\gamma } M_{\gamma+ \lambda}\nonumber\\
 &+4^{1-\gamma}R^{\gamma-1}2c_{3}(1-\gamma-\lambda)M_{\gamma+\lambda}+c_{4}M_{\gamma}\nonumber\\
 &= - \big[\frac{2\lambda + \gamma -1 }{1-\gamma}-\tilde{\varepsilon}\big]M_{\gamma+\lambda}
+ \frac{1}{M_{\gamma+\lambda}} [ (\gamma +\lambda) M_{2(\gamma+\lambda)-1} - 2c_{3} \left(1-\gamma-\lambda \right)
 M_{\gamma } M_{\gamma+ \lambda}^2\nonumber\\
 &+c_{4}M_{\gamma+\lambda}M_{\gamma}],
\end{align}
with $\tilde{\varepsilon}:=4^{1-\gamma}R^{\gamma-1}2c_{3}(1-\gamma-\lambda)$. Notice we can take $\tilde{R}$ sufficiently large so that $- \frac{2\lambda + \gamma -1 }{1-\gamma}+\tilde{\varepsilon}<0$, for every $R\geq \tilde{R}$.

Since $\gamma < 2(\gamma +\lambda ) -1 \leq 1 $ we deduce, by interpolation, using that $M_{1}(\Phi(t))=1$, for all $t\in[0,T]$, that there exists two positive constants $c_5$ and $c_6$ such that 
\[
M_{ 2(\gamma +\lambda ) -1 } \leq c_5 M_\gamma  + c_6. 
\]
Hence, since $\gamma + \lambda >0$ then 
\begin{align*}
\partial_{t} M_{\gamma +\lambda}& \leq- \big[\frac{2\lambda + \gamma -1 }{1-\gamma}-\tilde{\varepsilon}\big] M_{\gamma+\lambda}
+ \frac{1}{M_{\gamma+\lambda}} [ (\gamma +\lambda) \left[ c_5 M_\gamma  + c_6 \right]\\
&- 2c_3 \left(1-\gamma-\lambda \right)
 M_{\gamma } M_{\gamma+ \lambda}^2 +c_{4}M_{\gamma+\lambda}M_{\gamma}].
\end{align*}
Multiplying by $M_{\gamma +\lambda}$ the inequality we deduce that 
\begin{align*}
  M_{\gamma +\lambda}  \partial_{t} M_{\gamma +\lambda} \leq&- \big[\frac{2\lambda + \gamma -1 }{1-\gamma}-\tilde{\varepsilon}\big] M_{\gamma+\lambda}^2
+  [ (\gamma +\lambda) \left[ c_5 M_\gamma  + c_6 \right]\\
&- 2c_3 \left(1-\gamma-\lambda \right)
 M_{\gamma } M_{\gamma+ \lambda}^2+c_{4}M_{\gamma+\lambda}M_{\gamma}]
\end{align*}
which readjusting the constants implies 
\begin{align}\label{constant gamma lambda larger one}
&  \partial_{t} M_{\gamma +\lambda}^2 \leq- c_3 M_{\gamma+\lambda}^2
+  M_\gamma \left( c_4 +c_{5}M_{\gamma+\lambda}  - c_6 M_{\gamma+ \lambda}^2 \right) + c_7
\end{align}
for $c_3, c_4, c_5, c_6,c_{7}>0.$
This implies that the set $\left\{ \Phi \in \mathcal M_{+}(\mathbb R_*) : M_{\gamma +\lambda } \leq  \frac{ c_5+\sqrt{c_5^{2}+4c_{4}c_{6}}}{2 c_{6} } \right\} $ is invariant when $\gamma + 2 \lambda >1.$

We consider now the case $\gamma + 2 \lambda =1.$ 
First of all notice that for every $M < R $ we have 
\[
\int_{[M, \infty ) } x^{\gamma + \lambda} \Phi(t, \der x) \leq M^{\gamma+\lambda -1 } \int_{[M , \infty )} x \Phi(t, \der x ) \leq M^{\gamma + \lambda -1}.
\] 

Notice that an upper bound is obvious for the points $t\in[0,T]$ for which $M_{\gamma+\lambda}(\Phi(t))\leq 1$. 
If there exist $t_{1},t_{2}\in[0,T]$, $t_{1}<t_{2}$, such that  $M_{\gamma+\lambda}(\Phi(t_{1}))<1$ and $M_{\gamma+\lambda}(\Phi(t_{2}))>1$, by the continuity in time of $\Phi,$ we have that there exists $\overline{t}\in[t_{1},t_{2}]$ such that $M_{\gamma+\lambda}(\Phi(\overline{t}))=1$ and $\varepsilon_{1}\in(0,1)$ such that $M_{\gamma+\lambda}(\Phi(s))\geq 1$ on $[\overline{t},\overline{t}+\varepsilon_{1}]$.

On the interval $[\overline{t},\overline{t}+\varepsilon_{1}]$, we can apply the following logic.

We select $R$ large enough so that we can select $M$ to be such that $(1-\delta)^{\frac{1}{1-\gamma-\lambda}} < M < R$, but independent on $R$, and we deduce that there exists a $\delta >0$ such that
\begin{align*}
\int_{(0, M ) } x^{\gamma + \lambda}  \Phi(t,\der x)
&= \int_{\mathbb R_* } x^{\gamma + \lambda}\Phi(t, \der x) - \int_{[M, \infty ) } x^{\gamma + \lambda} \Phi(t, \der x) \\
&\geq \int_{\mathbb R_* } x^{\gamma + \lambda}\Phi(t, \der x)- M^{\gamma + \lambda -1 }\geq 1- M^{\gamma + \lambda -1 } \geq \delta. 
\end{align*}
As a consequence we deduce that 
\begin{equation} \label{gamma bound}
M_\gamma  \geq \int_{(0, M)} x^{\gamma} \Phi(t, \der x )  \geq M^{-\lambda} \int_{(0,M)} x^{\gamma+\lambda} \Phi(t, \der x)  \geq  \delta M^{- \lambda}. 
\end{equation}
Substituting the test function $\varphi(x)=x^{\gamma + \lambda} $ in equation \eqref{weak truncated time dependent eq},
we deduce that there exists a constant $c>0$ such that
\begin{align*}
&\partial_t M_{\gamma + \lambda } \leq \frac{\gamma + \lambda}{M_{\gamma+\lambda} } M_{2(\gamma + \lambda ) -1 } - c (1-\gamma -\lambda ) M_{\gamma +\lambda} M_\gamma+ \tilde{\epsilon}  M_{\gamma +\lambda} + c M_\gamma, 
\end{align*}
for $\tilde{\varepsilon}\in(0,1)$, which can be made sufficiently small as before. Hence, since for suitable constants $c_3, c_4 >0$ we have that $M_{2(\gamma+\lambda )-1} \leq c_3 M_\gamma + c_4$, similarly to (\ref{epsilon appears}), we deduce that
\begin{align*}
&\frac{1}{2}\partial_t M^2_{\gamma + \lambda } 
\leq c_5 M_{\gamma }   (\gamma + \lambda)  +(\gamma + \lambda) c_6 - c_7 M^2_{\gamma + \lambda } M_\gamma+\tilde{\varepsilon} M_{\gamma +\lambda}^{2}+c_{8} M_\gamma M_{\gamma+\lambda}.
\end{align*}
 Using \eqref{gamma bound} we deduce that there exists a constant $c_{9}>0$ such that
\begin{align*}
    &\frac{1}{2}\partial_t M^2_{\gamma + \lambda } 
\leq M_{\gamma }  \left[c_5   (\gamma + \lambda)  +(\gamma + \lambda) c_{9}  +c_{8}M_{\gamma+\lambda}-(c_7-\tilde{\varepsilon}\delta^{-1} M^{\lambda}) M^2_{\gamma + \lambda }  \right] .
\end{align*}
Choosing now $\tilde{\varepsilon}$ sufficiently small such that $-c_7+\tilde{\varepsilon}\delta^{-1} M^{\lambda}<0$, we have that there exists a constant $C_{1}>0$ such that 
\begin{align}\label{constant for gamma lambda}
\int_{\mathbb R_*} x^{\gamma + \lambda } \Phi(t, \der x ) \leq C_{1}. 
\end{align}

We now prove that $\left(S(t) \Phi_0 \right) ((0, \rho(C_1)))=0$. 
Notice that, by Lemma \ref{lem:auxiliary time dep truncated}, we know that $\left(S(t) \Phi_0 \right) ((0, \rho(2C_{1})))=0$, where we recall that $\rho( 2 C_1) \leq \rho(C_1)$. 
Consider now a sequence of test functions $ \{\varphi_n \}$ such that, for every $n\in\mathbb{N}$, $\varphi_{n}$ is decreasing and supported in $(0, \rho(C_1))$. Since $\varphi_n $ is decreasing, we deduce that the solution $F$ of equation \eqref{eq for F} satisfies 
\begin{align*} 
 \int_{\mathbb R_*} \varphi_n(x) F(t, \der x ) \leq \int_{\mathbb R_*} \varphi_n(x) \Phi_{0}(\der x )+ C(R)  \int_{0}^{t} \int_{\mathbb R_*} \varphi_n(x) F(s, \der x ) \der s, \quad \text{ for every } n \in \mathbb N, 
\end{align*} 
where we used in addition that $M_{1}(F)$ is bounded. Using Gr\"onwall’s inequality, we deduce that
\begin{align*}
    \int_{\mathbb R_*} \varphi_n(x) F(t, \der x ) =0,
\end{align*}
for every $t\in[0,T]$ and $n\in\mathbb{N}$.

We then let $\varphi_{n}$ converge pontwise to $\chi_{(0, \rho(C_1))}$ and obtain that
\begin{align*}
F(t, (0, \rho(C_1))) =0,
\end{align*}
for every $t\in[0,T]$.

Consider now a test function $\varphi$ such that $\varphi(x)=0$ for every $x \geq \rho(C_1)$ in equality \eqref{phi function of F}. Then 
\[
\int_{\mathbb R_* }\varphi(x) \Phi(t, \der x )= e^{\frac{1+\gamma}{1-\gamma}t} \int_{[\rho(C_1), \infty) }\varphi(X(t, y, \alpha) ) F(t, \der y). 
\] 
Using the fact that $X(t, y, \alpha) \geq \rho(C_1)$ for every $y \geq \rho(C_1)$, we deduce that 
$\Phi(t, (0, \rho(C_1)))=0$, for every $t \in [0, T]$. 

We conclude by proving (3). 
We consider a test function $\varphi $ equal to $x^{2-\gamma - \lambda} $ in equation \eqref{weak truncated time dependent eq} and deduce that
\begin{align} \label{eq weak form x 2-lambda-gamma} 
&\partial_t \int_{\mathbb{R}_{\ast}} \xi^{2-\gamma - \lambda }\Phi(t,\der \xi) - \frac{3-3\gamma- 2 \lambda}{1-\gamma}\int_{\mathbb{R}_{\ast}}\xi^{2-\gamma - \lambda }\Phi(t,\der \xi) 
- \frac{ 2-\gamma -\lambda }{M_{\gamma+\lambda}(\Phi(t)) } M_{1}(\Phi(t))  \\
&= \frac{1}{2}\int_{[ \rho(C_1), \infty)}\int_{[ \rho(C_1), \infty)} K_R(\xi,z) \left[(z+\xi)^{2-\gamma -\lambda }- \xi^{2-\gamma -\lambda}-z^{2-\gamma -\lambda} \right] \Phi(t,\der \xi)\Phi(t,\der z).\nonumber 
\end{align} 
By Cauchy-Schwarz inequality it follows that
\begin{equation}\label{eq following from CS}
\frac{1}{\int_{[ \rho(C_1), \infty)} \xi^{\gamma +\lambda} \Phi(t,\der \xi)} \leq 
\frac{\int_{[ \rho(C_1), \infty)}\xi^{2-\gamma - \lambda } \Phi(t,\der \xi)}{ \left( \int_{[ \rho(C_1), \infty)}\xi \Phi(t,\der \xi) \right)^2} = \int_{[ \rho(C_1), \infty)} \xi^{2-\gamma - \lambda } \Phi(t,\der \xi). 
\end{equation}
Plugging (\ref{eq following from CS}) into equation \eqref{eq weak form x 2-lambda-gamma} and using the fact that the total mass is equal to $1$ proved in \eqref{lower bound mass truncated}, we deduce that
\begin{align*}
&\partial_t \int_{[ \rho(C_1), \infty)} \xi^{2-\gamma - \lambda }\Phi(t,\der \xi) \leq \left( \frac{3\gamma-3+2\lambda}{1-\gamma} + 2-\gamma -\lambda\right) \int_{[ \rho(C_1), \infty)} \xi^{2-\gamma - \lambda }\Phi(t,\der \xi) \\
&+ \frac{1}{2}\int_{[ \rho(C_1), \infty)} \int_{[ \rho(C_1), \infty)}K_R(\xi,z) \left[(z+\xi)^{2-\gamma -\lambda }- \xi^{2-\gamma -\lambda}-z^{2-\gamma -\lambda} \right] \Phi(t,\der \xi)\Phi(t,\der z).
\end{align*} 

Hence, since $-1 < \gamma <1$ and $0 < \gamma + \lambda <1 $, we have
\begin{align*}
    \frac{3\gamma-3+2\lambda}{1-\gamma} + 2-\gamma -\lambda=- \frac{(\gamma+1)(1-\gamma-\lambda)}{1-\gamma } <0.
    \end{align*}
By symmetry, 
\begin{align*}
   &\frac{1}{2}\int_{[ \rho(C_1), \infty)} \int_{[ \rho(C_1), \infty)}K_R(\xi,z) \left[(z+\xi)^{2-\gamma -\lambda }- \xi^{2-\gamma -\lambda}-z^{2-\gamma -\lambda} \right] \Phi(t, \der\xi)\Phi(t,\der z) \\
   &\leq  \int_{[ \rho(C_1), \infty)} \int_{[\rho(C_1), z]} K_R(\xi,z) \left[(z+\xi)^{2-\gamma -\lambda }- \xi^{2-\gamma -\lambda}-z^{2-\gamma -\lambda} \right] \Phi(t,\der  \xi)\Phi(t,\der z).
\end{align*}
Assume $\xi\leq z$.  Denote $\eta:=\frac{\xi}{z}\in(0,1]$ and observe
\begin{align}\label{moments larger than one}
K_{R}(\xi,z)[(\xi+z)^{2-\gamma-\lambda}-\xi^{2-\gamma-\lambda}-z^{2-\gamma-\lambda}]\leq & K(\xi,z) z^{2-\gamma-\lambda}[(1+\eta)^{2-\gamma-\lambda}-1-\eta^{2-\gamma-\lambda}]\nonumber\\
\leq & C   K(\xi,z)z^{2-\gamma-\lambda}\eta \nonumber \nonumber\\
\leq& C z^{\gamma}(\eta^{\gamma+\lambda}+\eta^{-\lambda}) z^{2-\gamma-\lambda}\eta \leq 2C z^{2-\lambda}\eta^{1-\lambda}\nonumber \\
\leq & 4C (z\xi^{1-\lambda}+\xi z^{1-\lambda}).
\end{align}

Since $\rho(C_1)\leq \xi$ and $\rho(C_1)\leq z$, then $z^{1-\lambda}\leq\rho(C_1)^{-\lambda}z$. Hence 
\[
\frac{d }{dt } \int_{\mathbb R_*} \xi^{2-\gamma-\lambda} \Phi(t, \der \xi) \leq - c_3 \int_{\mathbb R_*} \xi^{2-\gamma-\lambda} \Phi(t, \der \xi) + c(\rho(C_1)),
\]
for suitable constants $c_3, c(\rho(C_1))>0$. Then, (3) follows.
\end{proof}

\begin{proposition}[Time-continuity of the semigroup]\label{prop:time cont} 
Assume $K_R $ to be a truncated kernel as in \eqref{truncated kernel} defined as a function of the homogeneous symmetric kernel $K$ that satisfies \eqref{kernel bounds},  \eqref{kernel cont},  for parameters $\gamma, \lambda \in \mathbb R$ such that \eqref{avoid_gelation_parameters} holds and such that $\gamma > -1$ and $\gamma + 2 \lambda \geq 1 $. 
Let $\Phi_0\in P$. Let $T>0$ be as in Theorem \ref{thm:time dep truncated}. 
The map $S(\cdot)\Phi_0:[0,T]\rightarrow  P$ is continuous in time, for every fixed $\Phi_{0}$, where $P$ was defined in (\ref{invariant region}).
\end{proposition}
\begin{proof}
Let $T>0$. We want to estimate the value of $|S(t) \Phi_0- S(s) \Phi_0|$, for $s,t\in[0,T]$. Assume without loss of generality that $s\leq t$. 
By the definition of the operator $S$ we know that
\begin{align*}
& \int_{\mathbb R_*} \varphi(x) [{\Phi}(t, \der x)-{\Phi}(s, \der x)]- \frac{1+\gamma}{1-\gamma} \int_{s}^{t}\int_{\mathbb R_*}  \varphi(x) \Phi(z, \der x)\der z  \\
&+ \frac{2}{1-\gamma} \int_{s}^{t}\int_{\mathbb R_*}  \varphi'(x) x \Phi(z, \der x)\der z 
 - \int_{s}^{t} \frac{1}{\int_{\mathbb R_*}x^{\gamma +\lambda}\Phi(z,\der x)} \int_{\mathbb R_*}  \varphi'(x) x^{\gamma + \lambda}\Phi(z, \der x)\der z \nonumber \\
& =\frac{1}{2}\int_{s}^{t}\int_{\mathbb R_*} \int_{\mathbb R_*} K_R(x,y) \left[ \varphi(x+y)- \varphi(x) - \varphi(y) \right] \Phi(z, \der x) \Phi(z, \der y)\der z. \nonumber 
\end{align*}
We have that there exists a constant $C_{2}>0$ such that $M_{\gamma+\lambda}(\Phi(r))\geq \frac{1}{C_{2}}$, for every $r\in[0,T]$.

As $\Phi$ has compact support, we can conclude that the exists a constant $C>0$, which may depend on $\rho(C_{1}),R,T, \gamma,\lambda$, but independent on $s,t$, such that 
\begin{align*}
    \bigg|& \int_{\mathbb R_*} \varphi(x) [{\Phi}(t, \der x)-{\Phi}(s, \der x)]\bigg|\leq C|t-s|,
\end{align*}
thus giving us the desired continuity of the semigroup.
\end{proof}

\begin{lemma}[Dual equation] \label{lem:dual eq} 
Assume $K_R$ to be a truncated kernel as in \eqref{truncated kernel} defined as a function of a homogeneous symmetric kernel $K$ that satisfies \eqref{kernel bounds},  \eqref{kernel cont},  for parameters $\gamma , \lambda \in \mathbb R $ satisfying \eqref{avoid_gelation_parameters} and with $\gamma > -1$ and $\gamma + 2 \lambda \geq 1 $. 
Let $\Phi_{1},\Phi_{2}\in P$ be two solutions of (\ref{weak truncated time dependent eq}) with initial conditions $\Phi_{\textup{in},1},\Phi_{\textup{in},2}\in P$, respectively, and $T>0$ be as in Theorem \ref{thm:time dep truncated}. Then there exists a unique solution $\varphi\in\textup{C}^{1}([0,T],\textup{C}^{1}_{c}([\rho(C_{1}),\infty))$, with $\varphi(T,\cdot)=\psi(\cdot)$, where $\psi$ is an arbitrary function in $\textup{C}^{1}_{c}([\rho(C_{1}),\infty))$, which solves the following equation: 
\begin{align} \label{eq:dual}  
    \partial_{t}\varphi(t, \xi)+ \frac{1+\gamma}{1-\gamma}   \varphi(t, \xi)  -
\frac{2\xi}{1-\gamma} \partial_\xi \varphi(t, \xi)  + \frac{   \xi^{\gamma + \lambda}}{M_{\gamma +\lambda}(\Phi_1)}\partial_\xi \varphi(t, \xi) +\mathbb{L}(\varphi)(t,\xi)=0,
\end{align}
where
\begin{align*}
     \mathbb{L}(\varphi)(t,\xi):=& -  \frac{\xi^{\gamma+\lambda} }{M_{\gamma+\lambda}(\Phi_1)  M_{\gamma+\lambda}(\Phi_2)  } \int_{\mathbb R_*} \partial_z \varphi(t, z) z^{\gamma+\lambda} \Phi_2(t, \der z)   \\ 
  & +  \frac{1}{2}\int_{\mathbb R_*} K_{R}(\xi,\eta)[\varphi(t,\xi+\eta)-\varphi(t,\xi)-\varphi(t,\eta)](\Phi_{1}(t, \der \eta)+\Phi_{2}(t, \der \eta)). 
\end{align*}
\end{lemma}
\begin{remark}
We prove the statement of the lemma for a modified operator  $\overline{\mathbb L}(\varphi)$ which preserves the continuity in the variable $\xi$, is equal to zero when $\xi\geq 8R$ and $\overline{\mathbb L}(\varphi)={\mathbb L}(\varphi)$ if $\xi\in[\rho(C_{1}),4R]$. Due to the support of $\Phi_{1}$ and $\Phi_{2}$, when proving the continuity of the map  $S(t)$ in the weak-$^\ast$ topology, it suffices to analyse the operator $\overline{\mathbb{L}}(\varphi)$. Thus, it is enough to prove the statement of the lemma only for the operator $\overline{\mathbb{L}}(\varphi)$. Notice that, for example, the operator $\mathbb{L}(\varphi)$ does not preserve compactness because of the presence of the term $\frac{\xi^{\gamma+\lambda} }{M_{\gamma+\lambda}(\Phi_1)  M_{\gamma+\lambda}(\Phi_2)  }\int_{\mathbb R_*} \partial_z \varphi(t, z) z^{\gamma+\lambda} \Phi_2(t, \der z)$. We keep the notation $\mathbb{L}(\varphi)$ for simplicity. We omit further details as the proof consists of standard methods used in the study of coagulation equations, see, for example, \cite{ferreira2019stationary}.
\end{remark}
\begin{proof}[Proof of Lemma \ref{lem:dual eq}]
First, we use the method of characteristics. 
We define $X(t, \xi)$ to be the solution of the ODE
\begin{align*}
   x'(t) &= -
\frac{2}{1-\gamma}x  + \frac{1}{\int_{\mathbb R_*}z^{\gamma +\lambda}\Phi_1(t,\der z)}  x^{\gamma + \lambda}
\end{align*} with initial condition $x(0)=\xi. $

In this way equation \eqref{eq:dual} can be rewritten in the following fixed point form: 
\begin{align} \label{eq:fixed point L}
  \varphi(t, X(t, \xi))  =   \mathcal L[\varphi]  (t, \xi),
  \end{align}
  where 
  \begin{align*} 
  \mathcal L [\varphi] (t, \xi) :=\varphi(T, X(T, \xi))  + \frac{1+\gamma}{1-\gamma}   \int_t^T  \varphi(s, X(s, \xi)) \der s 
+\int_t^T \mathbb{L}(\varphi)(s, X(s , \xi)) \der s. 
\end{align*}

Our strategy for proving the statement of the lemma is to apply Banach fixed point theorem. The operator $\mathcal L$ maps  ${C}^{1}([0,T],\textup{C}^{1}_{c}([\rho(C_{1}), \infty ) ))$ in itself. 

We prove that the operator $\mathcal L$ is a contraction if we endow $Y:= C([0,T], C^1_c([\rho(C_{1}), \infty ))$ with the norm $\| \varphi \|_Y := \sup_{ t \in [0, T] } \left(  \sup_{ x \in \mathbb R_*} |\varphi(t, x) |+ \sup_{x \in \mathbb R_*} | \partial_x \varphi(t, x) | \right) $. 

To this end we notice that
\begin{align*}
   \mathcal L [\varphi_1] (t, \xi) -  \mathcal L [\varphi_2 ] (t, \xi) &= \frac{1+\gamma}{1-\gamma}  \int_t^T  \left( \varphi_1 (s, X(s, \xi)) -  \varphi_2 (s, X(s, \xi)) \right) \der s \\
&+\int_t^T \left[ \mathbb{L}(\varphi_1)(s, X(s , \xi)) - \mathbb{L}(\varphi_2)(s, X(s , \xi)) \right]  \der s. 
\end{align*}
We notice that 
\begin{align*} 
&\int_t^T \left[ \mathbb{L}(\varphi_1)(s, X(s , \xi)) - \mathbb{L}(\varphi_2)(s, X(s , \xi)) \right]  \der s \\
&=  -\int_t^T  \frac{1}{M_{\gamma+\lambda}(\Phi_1)  M_{\gamma+\lambda}(\Phi_2)  } \int_{\mathbb R_*} \left(  \partial_z \varphi_1 (s, z) - \partial_z \varphi_2 (s, z) \right)z^{\gamma+\lambda} \Phi_2(s, \der z)    X(s, \xi)^{\gamma+\lambda} \der s  \\ 
  & +  \frac{1}{2}\int_t^T\int_{\mathbb R_*} K_{R}(X(s , \xi),\eta)[\varphi_1(s,X(s , \xi)+\eta)- \varphi_2(s,X(s , \xi)+\eta) -\varphi_1(s,X(s , \xi))\\
  &+\varphi_2(s,X(s , \xi)) -\varphi_1(s,\eta) + \varphi_2(s,\eta) ] (\Phi_{1}(s, \der \eta)+\Phi_{2}(s, \der \eta))\der s. 
\end{align*} 
From this we deduce that 
\[
\| \mathcal L [ \varphi_1] - \mathcal L [\varphi_2]  \|_Y \leq T c(\rho(C_{1}), R, \Phi_1, \Phi_2) \| \varphi_1 - \varphi_2 \|_Y
\] 
and hence $\mathcal L $ is a contraction for sufficiently small times $T.$
We can extend the solution to all possible times noting that the contraction constant $c(\rho(C_{1}),  \Phi_1, \Phi_2, R)$ does not depend on the final condition $\psi$. 
We thus deduce that there exists a solution $\varphi$ of the fixed point $\varphi= \mathcal L [\varphi].$
\end{proof}
We now prove that the found solution is Lipschitz continuous.
\begin{proposition}\label{dual function lipschitz}
Assume $K_R$ to be a truncated kernel as in \eqref{truncated kernel} defined as a function of a homogeneous symmetric kernel $K$ such that it satisfies \eqref{kernel bounds},  \eqref{kernel cont},  for parameters $\gamma , \lambda \in \mathbb R $ satisfying \eqref{avoid_gelation_parameters} and such that $\gamma > -1$ and $\gamma + 2 \lambda \geq 1 $. Let $T>0$ be as in Theorem \ref{thm:time dep truncated}. Let $\varphi\in \textup{C}^{1}([0,T],\textup{C}^{1}([\rho(C_1),8R]))$ with initial datum $\varphi(T,\cdot)$ be the function found in Lemma \ref{lem:dual eq}. Assume, in addition, that $\sup_{ \xi\in[\rho(C_1),8R]}|\varphi(T,\xi)|\leq 1$ and that $\varphi(T,\xi)$ is Lipschitz. Then $\varphi$ is Lipschitz continuous, in the sense that, for every $t\in[0,T],$ there exists $C(t)>0$ such that 
\begin{align*}
    \sup_{s\in[0,t]}|\varphi(s,\xi)-\varphi(s,\tilde{\xi})|\leq C(t)|\xi-\tilde{\xi}|,
\end{align*}
for every $\xi,\tilde{\xi}\in[\rho(C_1),8R]$. Moreover, $C(t)$ may depend on the norm of $\Phi_{1}$ and $\Phi_{2}$, but is otherwise independent of the choice of $\Phi_{1}$ and $\Phi_{2}$.
\end{proposition}
\begin{proof}
Notice first that, since $\sup_{ \xi\in[\rho(C_1),8R]}|\varphi(T,\xi)|\leq 1$, there exists a constant $C>0$, which depends on the norm of $\Phi_{1}$ and $\Phi_{2}$ and the parameters  $\rho(C_{1})$, $R$ and $t$, such that $\sup_{s\in[0,t],\xi\in [\rho(C_1),8R]}$ $|\varphi(s,\xi)|\leq C$. This can be seen by looking at (\ref{eq:fixed point L}). We will use Gr\"{o}nwall's inequality in (\ref{eq:fixed point L}) in order to prove that $\varphi$ is Lipschitz.
\begin{align*}
 |\varphi(t,X(t,\xi))-\varphi(t,X(t,\tilde{\xi}))|&\leq |\varphi(T,X(T,\xi))-\varphi(T,X(T,\tilde{\xi}))|\\
 &+\frac{1+\gamma}{1-\gamma}   \int_t^T  |\varphi(s, X(s, \xi))-\varphi(s,X(s,\tilde{\xi}))| \der s\\
&+\int_{t}^{T}|\mathbb L [\varphi] (s, X(s,\xi))-\mathbb L [\varphi] (s, X(s,\tilde{\xi}))|\der s.
\end{align*}

In order to bound the term
\begin{align*}
    \int_{t}^{T}|\mathbb L [\varphi] (s, X(s,\xi))-\mathbb L [\varphi] (s, X(s,\tilde{\xi}))|\der s,
\end{align*}
we need to estimate
\begin{align*}
    I_{1}=& \bigg|\frac{1}{2}\int_{\mathbb R_*} K_{R}(X(t,\xi),\eta)[\varphi(t,X(t,\xi)+\eta)-\varphi(t,X(t,\xi))-\varphi(t,\eta)](\Phi_{1}(t, \der \eta)+\Phi_{2}(t, \der \eta)) 
    \\&-\frac{1}{2}\int_{\mathbb R_*} K_{R}(X(t,\tilde{\xi}),\eta)[\varphi(t,X(t,\tilde{\xi})+\eta)-\varphi(t,X(t,\tilde{\xi}))-\varphi(t,\eta)](\Phi_{1}(t, \der \eta)+\Phi_{2}(t, \der \eta))\bigg|
\end{align*}
and
\begin{align*}
  I_{2}= \frac{|X(t,\xi)^{\gamma+\lambda}-X(t,\tilde{\xi})^{\gamma+\lambda}|}{M_{\gamma+\lambda}(\Phi_1)  M_{\gamma+\lambda}(\Phi_2)  } \int_{\mathbb R_*} |\partial_z \varphi(t, z)| z^{\gamma+\lambda} \Phi_2(t, \der z).
   \end{align*}
   For $I_{1},$ by the definition of
    $K_{R}$ in \eqref{truncated kernel}, we have $K_{R}$ is $C^{1}$ and we can assume it has compact support as we are only interested in the region $[\rho(C_{1}),8R]^{2}$. Thus,  we have that the first derivative of $K_{R}$ is bounded from above. Moreover, there exist constants $L_{1}(t), L_{2}(t)>0$ such that $L_{1}(t)|\xi-\tilde{\xi}|\leq |X(t,\xi)-X(t,\tilde{\xi})|\leq L_{2}(t)|\xi-\tilde{\xi}|$.
   
For $I_{2},$ we use
   \begin{align*}
   & \frac{1}{M_{\gamma+\lambda}(\Phi_1)  M_{\gamma+\lambda}(\Phi_2)  }\int_{\mathbb R_*} |\partial_z \varphi(t, z)| z^{\gamma+\lambda} \Phi_2(t, \der z)\leq \frac{C_{1}}{(C_{2})^{2}}\sup_{z\in[\rho(C_{1}),8R]}|\partial_z \varphi(t, z)|.
\end{align*}
If we prove that there exists a constant $C>0,$ which can depend on $\rho(C_{1}),R,$ and the norms of $\Phi_{1},$ and $\Phi_{2},$ but does not vary depending on the choice of $\Phi_{1},\Phi_{2}$ such that \begin{align}\label{bound for derivative in semigroup}
    \sup_{z\in[\rho(C_{1}),8R]}|\partial_z \varphi(t, z)|\leq C,
\end{align}
then there exists a constant $C>0,$ which can depend on $\rho(C_{1}),$ $R$ and the norms of $\Phi_{1}$ and $\Phi_{2}$ such that 
\begin{align}\label{last step semigroup}
    |\varphi(t,X(t,\xi))-\varphi(t,X(t,\tilde{\xi}))|&\leq C |X(t,\xi)-X(t,\tilde{\xi})|\nonumber\\
    &+ C\int_{t}^{T}    |\varphi(s,X(s,\xi))-\varphi(s,X(s,\tilde{\xi}))|.
\end{align}
Let us now prove (\ref{bound for derivative in semigroup}). From equation (\ref{eq:fixed point L}) and the fact that $K_{R}$ is in $C^{1}$ with compact support, we obtain an upper bound for $|\partial_{2}\varphi(t,X(t,\xi))|$.

We then use that there exist some constants $\overline{c}(t), \overline{\overline{c}}(t)>0$ such that $\overline{c}(t)\xi\leq  X(t,\xi)\leq \overline{\overline{c}}(t)\xi$, with $\xi\in[\rho(C_{1}),8R]$ in order to obtain the desired bound for $\sup_{z\in[\rho(C_{1}),8R]}|\partial_z \varphi(t, z)|$.

We use Gr\"{o}nwall's inequality in (\ref{last step semigroup}) and obtain that $|\varphi(t,X(t,\xi))-\varphi(t,X(t,\tilde{\xi}))|\leq C(t) |X(t,\xi)-X(t,\tilde{\xi})|.$ Thus, we can conclude that also $|\varphi(t,\xi)-\varphi(t,\tilde{\xi})|\leq C(t) |\xi-\tilde{\xi}|$ since $X(t,\xi)$ is Lipschitz continuous in the $\xi$ variable.
\end{proof}

\begin{proposition}[Continuity of the semigroup in the weak topology] \label{prop:weak star cont} 
Assume $K_R$ to be a truncated kernel as in \eqref{truncated kernel} defined as a function of a homogeneous symmetric kernel $K$ such that it satisfies \eqref{kernel bounds},  \eqref{kernel cont},  for parameters $\gamma , \lambda \in \mathbb R $ satisfying \eqref{avoid_gelation_parameters} and such that $\gamma > -1$ and $\gamma + 2 \lambda \geq 1 $. 
For every time $t>0$ the map 
\[
 S(t):P\rightarrow P
\] 
is continuous in the weak-$^{\ast}$ topology, where $P$ was defined in (\ref{invariant region}).
\end{proposition}
\begin{proof}
Let $\delta>0$. In order to prove continuity in the weak-$^{\ast}$ topology of the semigroup and because of the support of $\Phi_{1}$ and $\Phi_{2}$, it is enough to prove that, if  for every $\psi \in C_c([\rho(C_{1}),\infty))$, with $||\psi||_{\infty}\leq 1$, we have that 
\begin{align}\label{sufficiently small}
    \int_{\mathbb R_*} \psi (x) \left( \Phi_{\textup{in}, 1 } (\der x) - \Phi_{\textup{in}, 2} ( \der x ) \right) \text{ is sufficiently small,}
\end{align}
 then we have that for every $\psi \in C_c([\rho(C_{1}),\infty))$, with $||\psi||_{\infty}\leq 1,$
\begin{align*}
\int_{\mathbb R_*} \psi (x) \left( \Phi_1 (t, \der x) - \Phi_2(t, \der x ) \right) \leq \delta,
\end{align*}
where $\Phi_1 $ and $\Phi_2 $ are the solutions of equation \eqref{weak truncated time dependent eq} with initial conditions $\Phi_{\textup{in}, 1} $ and $\Phi_{\textup{in}, 2} $, respectively. 
To this end we notice that since $\Phi_1$ and $\Phi_2 $ satisfy equation \eqref{weak truncated time dependent eq} then 
\begin{align*}
&\int_{\mathbb R_*} \varphi(t,x) [{\Phi_{1}}(t, \der x)-{\Phi_{2}}(t, \der x)]-\int_{\mathbb R_*} \varphi(0,x) [{\Phi_{\textup{in},1}}( \der x)-{\Phi_{\textup{in},2}}( \der x)] \\
& -\int_{0}^{t}\int_{\mathbb R_*} \partial_{s}\varphi(s,x) [{\Phi_{1}}(s, \der x)-{\Phi_{2}}(s, \der x)]\der s
- \frac{1+\gamma}{1-\gamma} \int_{0}^{t}\int_{\mathbb R_*}  \varphi(s,x) [{\Phi_{1}}(s, \der x)-{\Phi_{2}}(s, \der x)]\der s \\ & +\frac{2}{1-\gamma}\int_{0}^{t} \int_{\mathbb R_*}  \partial_x \varphi(s,x) x [{\Phi_{1}}(s, \der x)-{\Phi_{2}}(s, \der x)]\der s \\
& - \int_{0}^{t}\frac{1}{\int_{\mathbb{R}_{\ast}}z^{\gamma+\lambda}\Phi_{1}(s,\der z)}  \int_{\mathbb R_*}  \partial_x \varphi(s, x) x^{\gamma + \lambda}\Phi_{1}(s, \der x)\der s\\
&+\int_{0}^{t}\frac{1}{\int_{\mathbb{R}_{\ast}}z^{\gamma+\lambda}\Phi_{2}(s,\der z)} \int_{\mathbb R_*}   \partial_x \varphi(s, x)x^{\gamma + \lambda}\Phi_{2}(s, \der x)\der s  \\
& =\frac{1}{2}\int_{0}^{t}\int_{\mathbb R_*} \int_{\mathbb R_*} K_R(x,y) \chi_{\varphi}(s,x,y) [{\Phi_{1}}(s, \der y)+{\Phi_{2}}(s, \der y)][{\Phi_{1}}(s, \der x)-{\Phi_{2}}(s, \der x)]\der s,
\end{align*} 
where $\chi_{\varphi}(s,x,y):=\varphi(s,x+y)-\varphi(s,x)-\varphi(s,y)$, for every $\varphi\in C([0, t], C_c^1(\mathbb R_*)).$
To simplify the computations, we adopt the notation $\tilde{\Phi}:= \Phi_1 - \Phi_2 $ and we notice that $\tilde{\Phi} $ satisfies the following equation for every $\varphi \in C^1([0, t], C_c^1(\mathbb R_*))$
\begin{align*}
&\int_{\mathbb R_*} \varphi(t,x) \tilde{\Phi}(t, \der x)-\int_{\mathbb R_*} \varphi(0,x) [{\Phi_{\textup{in},1}}( \der x)-{\Phi_{\textup{in},2}}( \der x)] \\
& -\int_{0}^{t}\int_{\mathbb R_*} \partial_{s}\varphi(s,x) \tilde{\Phi}(s, \der x)\der s
- \frac{1+\gamma}{1-\gamma} \int_{0}^{t}\int_{\mathbb R_*}  \varphi(s,x) \tilde{\Phi}(s, \der x)\der s
\\ &+\frac{2}{1-\gamma} \int_{0}^{t}\int_{\mathbb R_*}  \partial_x \varphi(s,x) x\tilde{\Phi}(s, \der x)\der s  - \int_{0}^{t}\frac{1}{\int_{\mathbb R_*}z^{\gamma +\lambda}\Phi_{1}(s,\der z)}  \int_{\mathbb R_*}  \partial_x \varphi(s, x) x^{\gamma + \lambda}\tilde{\Phi}(s, \der x)\der s \\ &+\int_{0}^{t}\int_{\mathbb R_*} \mathbb L[\varphi](s, x) \tilde{\Phi} (s, \der x)\der s.
\end{align*} 
First we make the following notation:
\begin{align*}
    C_{norm}:=\sup_{s\in[0,t]}\int_{\mathbb{R}_{\ast}}[\Phi_{1}+\Phi_{2}](s,\der x)<\infty.
\end{align*}
Notice that $C_{norm}$ can be bounded from above by 
\begin{align*}
    C_{norm}=\sup_{s\in[0,t]}\int_{\mathbb{R}_{\ast}}[\Phi_{1}+\Phi_{2}](s,\der x)\leq C(R,\rho(C_{1}),t) \int_{\mathbb{R}_{\ast}}[\Phi_{\textup{in},1}+\Phi_{\textup{in},2}](\der x).
\end{align*}
Let now $\varphi$ be the solution found in Lemma \ref{lem:dual eq} with coagulation kernel $K_{R}$. By Proposition \ref{dual function lipschitz}, we know that there exists a constant $C(t,R,\Phi_{1},\Phi_{2})$, that depends only on the norm of $\Phi_{1}$ and $\Phi_{2}$, but is otherwise independent of the choice of $\Phi_{1}$ and $\Phi_{2}$, such that $|\varphi(s,\xi)-\varphi(s,\tilde{\xi})|\leq C(t,R,\Phi_{1},\Phi_{2})|\xi-\tilde{\xi}|,$ for every $\xi,\tilde{\xi}\in[\rho(C_{1}),8R]$ and $s\in[0,t]$. We thus look at the set:
\begin{align*}
    \mathcal{K}_{\rho(C_{1}),R}:=&\{\chi\in\textup{C}([\rho(C_{1}),8R])| |\chi(\xi)-\chi(\tilde{\xi})|\leq C(t,R,\Phi_{1},\Phi_{2})|\xi-\tilde{\xi}|,\\
    &\textup{ for all } \xi, \tilde{\xi}\in[\rho(C_{1}),8R]\}\subset \textup{C}([\rho(C_{1}),8R]).
\end{align*}
The set $ \mathcal{K}_{\rho(C_{1}),R}$ is totally bounded, thus there exist $N\in\mathbb{N}$ and $\chi_{1},\ldots,\chi_{N}\in  \mathcal{K}_{\rho(C_{1}),R}$ such that $ \mathcal{K}_{\rho(C_{1}),R}\subseteq \cup_{i=1}^{N}B(\chi_{i}, \frac{\delta}{4N C_{norm}})$.

Then we obtain that:
\begin{align}
\int_{\mathbb R_*} \varphi(t,x) \tilde{\Phi}(t, \der x)&=\int_{\mathbb R_*} \varphi(0,x) [{\Phi_{\textup{in},1}}(\der x)-{\Phi_{\textup{in},2}}( \der x)] =:T_{1}. \label{conclusion}
\end{align} 
We can bound $T_{1}$ by
\begin{align*}
 T_{1}&\leq \min_{i=1}^{N}\int_{\mathbb R_*} |[\varphi(0,x)-\chi_{i}(x) ][{\Phi_{\textup{in},1}}(\der x)-{\Phi_{\textup{in},2}}( \der x)] | \\
 &+\max_{i=1}^{N}\big(\int_{\mathbb R_*} \chi_{i}(x) [{\Phi_{\textup{in},1}}(\der x)-{\Phi_{\textup{in},2}}( \der x)]\big)\\
 &\leq\frac{\delta}{4C_{norm}}\int_{\mathbb R_*}  [{\Phi_{\textup{in},1}}(\der x)+{\Phi_{\textup{in},2}}( \der x)]  + \max_{i=1}^{N}\int_{\mathbb R_*} \chi_{i}(x) [{\Phi_{\textup{in},1}}(\der x)-{\Phi_{\textup{in},2}}( \der x)]\leq \delta,
\end{align*}
where in the last step we used (\ref{sufficiently small}).
\end{proof}
\begin{proof}[Proof of Theorem \ref{thm:stationary truncated}]
By Proposition \ref{prop:weak star cont} we know that the operator $\Phi \mapsto S(t) \Phi$ is continuous in the  weak-$^{\ast}$ topology. 
Additionally, thanks to Proposition \ref{prop:inv region}, we know that $P$ is an invariant region for $S(t)$. 
Since $P$ is also convex and compact and since we have proven that the map $t \mapsto S(t)$ is continuous,  we apply Theorem 1.2 in \cite{escobedo2005self} to deduce that there exists a $\Phi $ such that $S(t)\Phi=\Phi$.  
\end{proof}

\subsection{Existence of the self-similar profile} \label{sec:existence self-similar profile}
To keep the notation lighter, in the previous sections we denoted with $\Phi$ the solution of the truncated problem. In this section, since we pass to the limit as $R$ tends to infinity, we add the label $R$ to $\Phi$ and we will denote with $\Phi$ the self-similar profile.  
For simplicity, we denote by $\rho:=\rho(C_{1}),$ where $\rho(C_{1})$ was defined in (\ref{delta1formula}). 
\begin{proof}[Proof of Theorem \ref{thm:existence}]
For every $R>1$, sufficiently large, there exist $T>0$ and $\Phi_{R}\in  C^1([0,T];\mathcal M_{+}(\mathbb R_*))$ such that
\begin{align*}
& \int_{\mathbb R_*} \varphi(x) \dot{\Phi}_{R}(t, \der x)- \frac{1+\gamma}{1-\gamma} \int_{\mathbb R_*}  \varphi(x) \Phi_{R}(t, \der x) +
\frac{2}{1-\gamma} \int_{\mathbb R_*}  \varphi'(x) x \Phi_{R}(t, \der x) \\
& - \frac{1}{M_{\gamma+\lambda}(\Phi_{R}(t))}  \int_{\mathbb R_*}  \varphi'(x) x^{\gamma + \lambda}\Phi_{R}(t, \der x)  \\
& =\frac{1}{2}\int_{\mathbb R_*} \int_{\mathbb R_*} K_R(x,y) \left[ \varphi(x+y)- \varphi(x) - \varphi(y) \right] \Phi_{R}(t, \der x) \Phi_{R}(t, \der y), 
\end{align*} 
where $K_{R}$ is the coagulation kernel defined in (\ref{truncated kernel}). Notice that the bounds in Proposition \ref{prop:inv region} are independent of $R>1$. 

Thanks to Theorem \ref{thm:stationary truncated} we know that a measure $\Phi_R \in \mathcal M_+(\mathbb R_*) $ satisfying equation \eqref{weak truncated steady state eq} and satisfying the bounds  
\eqref{uniform bounds self-sol} exists.
We notice also that, since $\gamma + \lambda>0$,
\begin{align}\label{everybound}
    \int_{(0,\infty)}\Phi_{R}(\der x)=\int_{[\rho,\infty)}\Phi_{R}(\der x)\leq \rho^{-\gamma-\lambda}\int_{[\rho,\infty)}x^{\gamma+\lambda}\Phi_{R}(\der x)\leq \rho^{-\gamma-\lambda}C_{1}.
\end{align}
Hence, Banach-Alaoglu Theorem implies that there exists $\Phi$ such that 
\begin{align}\label{convergence_fixed_points}
\Phi_R \rightharpoonup \Phi \text{ as } R \rightarrow \infty
\end{align}
in the weak-$^{\ast}$ topology.

We now prove that the measure $\Phi $ in (\ref{convergence_fixed_points}) satisfies equation \eqref{weak steady state eq} by taking the limit as $R \rightarrow \infty$. Fix $\varphi\in C_{c}(\mathbb{R}_{\ast})$.
We start with passing to the limit in the coagulation term.

Let $\varepsilon\in(0,1)$. We first show that 
\begin{align}\label{first step convergence kernel}
   \bigg| \int_{[\rho,\infty)}\int_{[\rho,\infty)}\left(K (\xi,z)- K_{R}(\xi,z) \right)\Phi_{R}(\der \xi)\Phi_{R}(\der z)[\varphi(z+\xi)-\varphi(z)-\varphi(\xi)] \bigg|\leq \varepsilon,
\end{align}
for sufficiently large $R$. We then prove that 
\begin{align}\label{second step convergence kernel}
    \int_{[\rho,\infty)}\int_{[\rho,\infty)} K(\xi,z) \Phi_{R}(\der \xi)\Phi_{R}(\der z)[\varphi(z+\xi)-\varphi(z)-\varphi(\xi)] 
\end{align}
converges to
\begin{align}\label{second step convergence kernel 2}
    \int_{[\rho,\infty)}\int_{[\rho,\infty)}K (\xi,z)\Phi(\der \xi)\Phi(\der z)[\varphi(z+\xi)-\varphi(z)-\varphi(\xi)] 
\end{align}
as $R \rightarrow \infty .$

For (\ref{first step convergence kernel}), we have that: 
\begin{align*}
   &\bigg| \int_{[\rho,\infty)}\int_{[\rho,\infty)}\left(K (\xi,z)- K_{R}(\xi,z) \right)\Phi_{R}(\der \xi)\Phi_{R}(\der z)[\varphi(z+\xi)-\varphi(z)-\varphi(\xi)] \bigg|  \\
   &\leq \int_{[\rho,\frac{R}{4}]}\int_{[\rho,\frac{R}{4}]}\bigg|\left(K (\xi,z)- K_{R}(\xi,z) \right)\Phi_{R}(\der \xi)\Phi_{R}(\der z)[\varphi(z+\xi)-\varphi(z)-\varphi(\xi)] \bigg| \\
   &+ 2\int_{(\frac{R}{4},\infty)}\int_{[\rho,\infty)}\bigg|\left(K (\xi,z)- K_{R}(\xi,z) \right)\Phi_{R}(\der \xi)\Phi_{R}(\der z)[\varphi(z+\xi)-\varphi(z)-\varphi(\xi)] \bigg|\\
   &=T_{1}+T_{2}.
\end{align*}
For the first term, we have  
\begin{align*}
    T_{1}\leq e^{-R}3||\varphi||_{\infty}\rho^{-2},
\end{align*}
where we used the definition of $K_{R}$ in \eqref{truncated kernel} and the fact that the total mass of the measures is equal to $1$.

For the region $(\frac{R}{4}, \infty)\times[\rho, \infty)$, we use $K_R \leq K$ to prove that 
\begin{align*}
    T_{2}&\leq 12||\varphi||_{\infty}
        \int_{(\frac{R}{4},\infty)}\int_{[\rho,\infty)} K(\xi,z) \Phi_{R}(\der z)\Phi_{R}(\der \xi)\\     & \leq c \left(  R^{-\gamma-2\lambda}M_{\gamma+\lambda}^{2}(\Phi_{R}) +R^{\gamma+\lambda-1}\rho^{-\gamma-2\lambda}  M_{\gamma+\lambda}(\Phi_{R}) \right),
\end{align*}
which gives a small contribution due to the uniform estimates for $\Phi_{R}$ if we make $R$ sufficiently large.

We now analyse (\ref{second step convergence kernel}). We consider a continuous function $g: \mathbb R_+ \rightarrow [0, 1] $ such that $g(x)=1$, when $x\leq 1$, and $g(x)=0$, when $x \geq 2$. 
We define the function $p_M $ as  
\begin{align} \label{p_M}
p_{M}(x,y)=g\left(\frac{x}{M}\right)g\left(\frac{y}{M}\right),
\end{align}
where $M$ is a positive constant.
By the construction of $p_{M}$ and given the fact that $\Phi_{R}$ is supported in the region $[\rho,\infty)$, for every $R>1$, we have that, given any function $\varphi \in C_c(\mathbb R_*)$, 
\begin{align*}
    &\int_{[\rho,\infty)}\int_{[\rho,\infty)}K(\xi,z)p_{M}(\xi,z)\Phi_{R}(\der \xi)\Phi_{R}(\der z)[\varphi(z+\xi)-\varphi(z)-\varphi(\xi)]=\\
     &\int_{[\rho,2M]}\int_{[\rho,2M]}K(\xi,z)p_{M}(\xi,z)\Phi_{R}(\der \xi)\Phi_{R}(\der z)[\varphi(z+\xi)-\varphi(z)-\varphi(\xi)].
\end{align*}
Therefore, since $\Phi_R \rightharpoonup \Phi $ in the weak-$^{\ast}$ topology, we can conclude that $\Phi_R \Phi_{R}\rightharpoonup \Phi \Phi $ in the weak-$^{\ast}$ topology as $R\rightarrow\infty$ if we work on the set $[\rho,2M]^{2}$ and we deduce that 
\begin{align*}
    \int_{[\rho,\infty)}\int_{[\rho,\infty)}K(\xi,z)p_{M}(\xi,z)\Phi_{R}(\der \xi)\Phi_{R}(\der z)[\varphi(z+\xi)-\varphi(z)-\varphi(\xi)]
\end{align*}
converges to
\begin{align*}
    \int_{[\rho,\infty)}\int_{[\rho,\infty)}K(\xi,z)p_{M}(\xi,z)\Phi(\der \xi)\Phi(\der z)[\varphi(z+\xi)-\varphi(z)-\varphi(\xi)]
\end{align*}
as $R $ tends to infinity. 

To conclude that, for every test function $\varphi \in C_c(\mathbb R_*), $ we have that (\ref{second step convergence kernel}) converges to (\ref{second step convergence kernel 2}) as $R  \rightarrow \infty$, we have to prove that the reminder terms, namely
\begin{equation}\label{reminder1}
\int_{[\rho,\infty)}\int_{[M,\infty)}K(\xi,z)[\varphi(z+\xi)-\varphi(z)-\varphi(\xi)]\Phi_{R}(\der \xi)\Phi_{R}(\der z),
\end{equation}
\begin{equation} \label{reminder2}
\int_{[M,\infty)}\int_{[\rho,\infty)}K(\xi, z) [\varphi(z+\xi)-\varphi(z)-\varphi(\xi)]\Phi_{R}(\der \xi)\Phi_{R}(\der z),
\end{equation}
\begin{equation}\label{reminder3}
\int_{[\rho,\infty)}\int_{[M,\infty)}K(\xi,z)[\varphi(z+\xi)-\varphi(z)-\varphi(\xi)]\Phi(\der \xi)\Phi(\der z),
\end{equation}
\begin{equation} \label{reminder4}
\int_{[M,\infty)}\int_{[\rho,\infty)}K(\xi, z) [\varphi(z+\xi)-\varphi(z)-\varphi(\xi)]\Phi(\der \xi)\Phi(\der z),
\end{equation}
tend to zero as $M\rightarrow\infty$.

Let us look at the term \eqref{reminder1}. In this case we are in the region where $\{\xi\geq M\}$, hence $\xi^{\gamma+\lambda-1}\leq M^{\gamma+\lambda-1}$ as $\gamma+\lambda<1$. Therefore, for every $\varphi\in\textup{C}_{c}(\mathbb{R}_{>0})$ 
\begin{align*}
\int_{[\rho,\infty)}\int_{[M,\infty)}\xi^{\gamma+\lambda}z^{-\lambda}\Phi_{R}(\der \xi)\Phi_{R}(\der z)|\varphi(z+\xi)-\varphi(z)-\varphi(\xi)|\lesssim M^{\gamma+\lambda-1}\rho ^{-\gamma-2\lambda} M_{\gamma+\lambda}(\Phi_{R}),
\end{align*}
where we remind that
 $M_{\gamma+\lambda}(\Phi_{R})$ is bounded uniformly in $R>1$ and that the mass of $\Phi_R$ is equal to one. 
 Similarly, the fact that $\xi\geq M$ implies that $\xi^{-\gamma-2\lambda}\leq M^{-\gamma-2\lambda}$ since $\gamma+2\lambda> 0$. 
 We then obtain that
\begin{align*}
\int_{[\rho,\infty)}\int_{[M,\infty)}z^{\gamma+\lambda}\xi^{-\lambda}\Phi_{R}(\der \xi)\Phi_{R}(\der z)|\varphi(z+\xi)-\varphi(z)-\varphi(\xi)|\leq  C M^{-\gamma-2\lambda}M_{\gamma+\lambda}^{2}(\Phi_{R}). 
\end{align*}
From these two inequalities and the fact that $\gamma + \lambda <1 $ and $\gamma + 2 \lambda >0$ we deduce that the term \eqref{reminder1} tends to zero as $M \rightarrow \infty.$
By a symmetric argument we prove that the term \eqref{reminder2} tends to zero as $M\rightarrow\infty$.
The fact that the two terms \eqref{reminder3} and \eqref{reminder4} tend to zero as $M \rightarrow \infty$ follows similarly by the fact that the $\gamma + \lambda $ moment of $\Phi$ is bounded. 

The linear terms for the self-similar profiles of the truncated problem will converge as $R\rightarrow\infty$ to the desired terms thanks to (\ref{convergence_fixed_points}).

To conclude, we need to prove the convergence of the remainder term. We thus have to prove that
\begin{align*}
 \frac{1}{\int_{[\rho,\infty)}x^{\gamma+\lambda}\Phi_{R}(\der x)}\int_{[\rho,\infty)}\partial_{x}\varphi(x)x^{\gamma+\lambda}\Phi_{R}(\der x)
   \end{align*}
   goes to 
   \begin{align*}
 \frac{1}{\int_{[\rho,\infty)}x^{\gamma+\lambda}\Phi(\der x)}\int_{[\rho,\infty)}\partial_{x}\varphi(x)x^{\gamma+\lambda}\Phi(\der x)
   \end{align*}
 as $R\rightarrow\infty$. We have that
   \begin{align*}
   \bigg| &\frac{1}{M_{\gamma+\lambda}(\Phi_{R})}\int_{[\rho,\infty)}\partial_{x}\varphi(x)x^{\gamma+\lambda}\Phi_{R}(\der x)-\frac{1}{M_{\gamma+\lambda}(\Phi)}\int_{[\rho,\infty)}\partial_{x}\varphi(x)x^{\gamma+\lambda}\Phi(\der x)\bigg|\\
   &\leq    \bigg| \frac{1}{\int_{[\rho,\infty)}x^{\gamma+\lambda}\Phi_{R}(\der x)}\bigg|\bigg|\int_{[\rho,\infty)}\partial_{x}\varphi(x)x^{\gamma+\lambda}\Phi_{R}(\der x)-\int_{[\rho,\infty)}\partial_{x}\varphi(x)x^{\gamma+\lambda}\Phi(\der x)\bigg|\\
   &+   \bigg| \frac{1}{\int_{[\rho,\infty)}x^{\gamma+\lambda}\Phi_{R}(\der x)}-\frac{1}{\int_{[\rho,\infty)}x^{\gamma+\lambda}\Phi(\der x)}\bigg|\bigg|\int_{[\rho,\infty)}\partial_{x}\varphi(x)x^{\gamma+\lambda}\Phi(\der x)\bigg|.
\end{align*}
We can analyse the two terms separately. For the first term, we know by \eqref{eq following from CS} that the $M_{\gamma+\lambda}(\Phi_{R})$ moment is uniformly bounded from below and thus
\begin{align*}
 \bigg|& \frac{1}{\int_{[\rho,\infty)}x^{\gamma+\lambda}\Phi_{R}(\der x)}\bigg|\bigg|\int_{[\rho,\infty)}\partial_{x}\varphi(x)x^{\gamma+\lambda}\Phi_{R}(\der x)-\int_{[\rho,\infty)}\partial_{x}\varphi(x)x^{\gamma+\lambda}\Phi(\der x)\bigg| \\
&\leq M_{2-\gamma-\lambda}(\Phi_{R})\bigg|\int_{[\rho,\infty)} \partial_{x}\varphi(x)x^{\gamma+\lambda}\Phi_{R}(\der x)-\int_{[\rho,\infty)}\partial_{x}\varphi(x)x^{\gamma+\lambda}\Phi(\der x)\bigg|\\
&\leq C_{2}\bigg|\int_{[\rho,\infty)} \partial_{x}\varphi(x)x^{\gamma+\lambda}\Phi_{R}(\der x)-\int_{[\rho,\infty)}\partial_{x}\varphi(x)x^{\gamma+\lambda}\Phi(\der x)\bigg|
\end{align*}
and then we use that $\bigg|\int_{[\rho,\infty)}\partial_{x}\varphi(x)x^{\gamma+\lambda}\Phi_{R}(\der x)-\int_{[\rho,\infty)} \partial_{x}\varphi(x)x^{\gamma+\lambda}\Phi(\der x)\bigg|\rightarrow 0$ by (\ref{convergence_fixed_points}).

For the second term, we split the integral into a region with compact support and the remainder regions. More precisely,
\begin{align*}
      \bigg| &\frac{\int_{[\rho,\infty)}x^{\gamma+\lambda}\Phi(\der x)-\int_{[\rho,\infty)}x^{\gamma+\lambda}\Phi_{R}(\der x)}{\int_{[\rho,\infty)}x^{\gamma+\lambda}\Phi_{R}(\der x)\int_{[\rho,\infty)}x^{\gamma+\lambda}\Phi(\der x)}\bigg|\bigg|\int_{[\rho,\infty)}\partial_{x}\varphi(x)x^{\gamma+\lambda}\Phi(\der x)\bigg|\\
     &\leq C\bigg|\int_{[\rho,\infty)}x^{\gamma+\lambda}\Phi(\der x)-\int_{[\rho,\infty)}x^{\gamma+\lambda}\Phi_{R}(\der x)\bigg|M_{2-\gamma-\lambda}(\Phi_{R})M_{2-\gamma-\lambda}(\Phi)M_{\gamma+\lambda}(\Phi)\\
     &\lesssim C\bigg|\int_{[\rho,\infty)}g\left(\frac{x}{M}\right)\left[x^{\gamma+\lambda}\Phi(\der x)-x^{\gamma+\lambda}\Phi_{R}(\der x)\right]\bigg|+2C\bigg|\int_{[M, \infty)}x^{\gamma+\lambda}\Phi(\der x)\bigg|\\
      &+2C\bigg|\int_{[M, \infty)}x^{\gamma+\lambda}\Phi_{R}(\der x)\bigg|.
\end{align*}
We have that $\bigg|\int_{[\rho, \infty)}g(\frac{x}{M})[x^{\gamma+\lambda}\Phi(\der x)-x^{\gamma+\lambda}\Phi_{R}(\der x)]\bigg|\rightarrow 0$ as $R\rightarrow\infty$ by (\ref{convergence_fixed_points}). For the regions close to infinity, we use
\begin{align*}
    \bigg|\int_{[M,\infty)}x^{\gamma+\lambda}\Phi_{R}(\der x)\bigg|\leq M^{\gamma+\lambda-1},
\end{align*}
as the total mass is uniformly bounded. Since $\gamma+\lambda<1,$ the remaining term goes to zero as $M\rightarrow\infty$.

As a last step we prove that, if $\Phi((0, \delta))=0$ for some $\delta >0$, then $\Phi((0, \rho(M_{\gamma+\lambda})))=0$ for $\rho(M_{\gamma+\lambda})$ given by \eqref{rhoM}. 
Since the proof is standard we only sketch it here. 
Consider $x \in [\frac{\delta}{2},\frac{1}{2}\rho(M_{\gamma+\lambda}))$, then, since $\Phi$ satisfies \eqref{eq: ss strong}, we have
\begin{align*}
c \Phi(x) & \leq \left( \frac{x^{\gamma+\lambda} }{M_{\gamma+\lambda}} - \frac{2}{1-\gamma} x \right) \Phi(x)\\
& \leq \frac{|1+\gamma |}{1-\gamma} \int_\delta^x \Phi(\eta ) \der \eta + \int_{\delta}^x  \int_\delta^{y-\delta} K(y-\eta,\eta) \Phi(y-\eta ) \Phi(\eta ) \der \eta \der y \\
&\leq \tilde{c} \int_\delta^x \Phi(\eta) \der \eta,
\end{align*}
for some positive constants $\tilde{c}, c >0$.
Using Gr\"onwall's lemma, this implies that $\Phi(x) \leq e^{\frac{\tilde{c}}{c} x }\Phi(\frac{\delta}{2}) =0$. Hence $\Phi([\delta, \frac{1}{2}\rho(M_{\gamma+\lambda})))=0$.
Iterating this argument, we deduce that $\Phi((0, \rho(M_{\gamma+\lambda} )))=0$.
The argument can be made rigorous by working with the weak formulation of equation \eqref{eq: ss strong}.
\end{proof}

\subsection{Properties of the constructed self-similar profile}
In this section we aim at proving Theorem \ref{thm:existence and properties}. In particular we prove that the self-similar profile, whose existence has been proven in Theorem \ref{thm:existence}, satisfies the properties stated in Theorem \ref{thm:existence and properties}. 
In other words we have to prove the following theorem. 
\begin{theorem}\label{thm:properties of the constructed s.p.}
Assume $K$ to be a homogeneous symmetric coagulation kernel, of homogeneity $\gamma$, satisfying
\eqref{kernel bounds},  \eqref{kernel cont},  with $\gamma , \lambda $ such that  \eqref{avoid_gelation_parameters} holds and such that
\[
-1 < \gamma, \quad \gamma+2\lambda \geq  1.
\] 
Let $\Phi$ be the self-similar profile constructed in Theorem \ref{thm:existence}. 
Then 
\[
\int_{\mathbb R_*} e^{Lx } \Phi(\der x) < \infty 
\] 
for some $L>0$ and $\Phi$ is absolutely continuous with respect to the Lebesgue measure. 
Then $\Phi(\der x )=\phi(x) \der x  $ and $\phi$ is such that 
\[
\limsup_{x  \rightarrow \infty} e^{ \overline{L} x} \phi(x) < \infty
\]
for a constant $\overline{L}>0$. 
\end{theorem}
To prove this theorem we start by proving that every self-similar profile $\Phi$ in the sense of Definition \ref{def:self-similar profile} and such that $\Phi((0, \delta))=0$, for a positive $\delta$,
is absolutely continuous with respect to the Lebesgue measure and satisfies some moment bounds. 
\begin{proposition}[Regularity] \label{prop:regularity}
Assume $K$ to be a homogeneous symmetric coagulation kernel, of homogeneity $\gamma$, satisfying
\eqref{kernel bounds},  \eqref{kernel cont},  with $\gamma , \lambda $ such that  \eqref{avoid_gelation_parameters} holds and such that $ \gamma+2\lambda \geq  1.$
	Let $\Phi \in \mathcal M_+(\mathbb R_*)$ be a self-similar profile as in Definition \ref{def:self-similar profile}.
	Assume additionally that $\Phi((0, \delta ))=0$ for $\delta >0$. 
	Then $\Phi$ is absolutely continuous with respect to the Lebesgue measure. 
	Its density $\phi \in L^1([\delta, \infty ))$ satisfies the following equation 
	\begin{align} \label{flux form} 
&\int_0^z x \phi(x) dx - \frac{2}{1-\gamma} z^2 \phi(z) + \frac{1}{M_{\gamma+\lambda}(\phi)} z^{\gamma +\lambda +1 } \phi(z) - \frac{1}{M_{\gamma+\lambda}(\phi) } \int_0^z x^{\gamma +\lambda }\phi(x) \der x \\
	& = - J_\phi(z), \quad a.e.\  z > \delta \nonumber
	\end{align} 
	where 
	\begin{equation}\label{flux def}
	J_\phi(z):= \int_0^z \int_{z-x}^\infty x K(x,y) \phi(x) \phi(y) \der y \der x.
	\end{equation}
\end{proposition}
\begin{proof}
We just sketch the proof as similar arguments have been repeatedly used in the analysis of coagulation equations, see for instance the proof of Lemma 4.9 in \cite{ferreira2021self}. An analogous proof will be presented in the proof of Proposition \ref{prop:support}. 
Let $\Phi$ be a solution of \eqref{eq: ss strong}, then 
\begin{align*} 
\partial_\xi \left[ \left( \frac{\xi^{\gamma + \lambda}}{M_{\gamma +\lambda} } - \frac{2}{1-\gamma} \xi \right) \Phi(\xi) \right] = \frac{1+\gamma}{1-\gamma} \Phi(\xi) + \mathbb K [\Phi](\xi).
\end{align*} 
Using the fact that 
\[
\mathbb K [\Phi] (\xi) \leq \int_{\delta}^{\xi-\delta} K(\xi-\eta,\eta) \Phi(\eta) \Phi(\xi-\eta) \der \eta,   \quad \xi > \delta
\]
and that $\int_{\delta}^{\xi-\delta} K(\xi-\eta,\eta) \Phi(\eta) \Phi(\xi-\eta) \der \eta $ is a bounded measure in every compact set, we deduce that the measure 
\[ 
\left( \frac{\xi^{\gamma + \lambda}}{M_{\gamma +\lambda} } - \frac{2}{1-\gamma} \xi \right) \Phi(\xi)
\] 
is  absolutely continuous with respect to the Lebesgue measure on the set $(\delta, \infty )$ if $\delta > \rho(M_{\gamma+\lambda})$, while it
is absolutely continuous with respect to the Lebesgue measure  in the set $(\delta,\rho(M_{\gamma +\lambda})) \cup (\rho(M_{\gamma+\lambda}), \infty )$ if $\delta \leq \rho(M_{\gamma+\lambda})$. In the latter case, we therefore have that $\Phi(\der x )=\phi(x) \der x + c_0 \delta_{\rho(M_{\gamma+\lambda})}(x) $
with density $\phi \in L^1(\mathbb R_*) $ and with $c_0 \in \mathbb R$.
To exclude that the self-similar solution $\Phi$ has a Dirac in $\rho(M_{\gamma+\lambda})$,
we notice that
\begin{align*}
&\int_\delta^{\xi-\delta } K(\xi- \eta , \eta) c_0\delta_{\rho(M_{\gamma+\lambda})}(\eta ) c_0 \delta_{\rho(M_{\gamma+\lambda})}(\xi -\eta) \der \eta \\
&= c_0^2 \int_\delta^{\xi-\delta} K(\xi-\eta , \eta) \delta_{\rho(M_{\gamma+\lambda})} (\eta) \delta_{\rho(M_{\gamma+\lambda})}(\xi-\eta) \der \eta
\end{align*} 
is non-zero only if $\xi= 2 \rho(M_{\gamma+\lambda}) $.
Hence the coagulation operator applied to measure $c_0 \delta_{\rho(M_{\gamma+\lambda})}$ produces a Dirac in $2 \rho(M_{\gamma+\lambda})$, contradicting the fact that $\Phi_{|(\rho(M_{\gamma+\lambda}), \infty)}$ is absolutely continuous with respect to the Lebesgue measure. As a consequence, we deduce that $c_0=0$. 
\end{proof} 

\begin{proposition}[Moment bounds] \label{prop:moment bounds} 
Assume $K$ to be a homogeneous symmetric coagulation kernel, of homogeneity $\gamma$, satisfying
\eqref{kernel bounds},  \eqref{kernel cont},  with $\gamma , \lambda $ such that  \eqref{avoid_gelation_parameters} holds and such that
$
\gamma+2\lambda \geq  1.$
	Let $\Phi \in \mathcal M_+(\mathbb R_*)$ be a self-similar profile as in Definition \ref{def:self-similar profile} and assume additionally that $\Phi((0,\delta))=0$, where $\delta>0$.
	Then 
	\[
	\int_{\mathbb R_*} x^\mu \Phi(\der x) = \int_{[\delta, \infty )} x^\mu \Phi( \der x) < \infty, \quad \forall \mu \in \mathbb R. 
	\]
\end{proposition}
\begin{proof}
	For $\mu \leq \gamma + \lambda$ the statement follows by the fact that $\Phi((0, \delta))=0$ and that by definition of self-similar profile we know that
	\[
	\int_{[\delta, \infty )} x^{\gamma + \lambda} \Phi(\der x) < \infty.
	\]
	We want to prove the statement for $\mu > \gamma + \lambda.$
	As a first step, we prove that there exists a $\overline \delta >0$ such that 
	\begin{equation}\label{moment 1+delta}
	\int_{\mathbb R_*} x^{\gamma+ \lambda +\overline \delta } \phi(x) \der x  < \infty, 
	\end{equation}
	where $\phi$ is the density of $\Phi$.
	As a second step we will prove that 
	\begin{equation}\label{moment 1+(n-1)delta}
	\int_{\mathbb R_*} x^{\gamma+ \lambda +n \overline \delta } \phi(x) \der x  < \infty 
	\end{equation}
	for $n \geq 1$ implies that
	\begin{equation}\label{moment 1+ndelta}
	\int_{\mathbb R_*} x^{\gamma+ \lambda +(n+1)\overline \delta } \phi(x) \der x  < \infty. 
	\end{equation}
	The desired conclusion then follows by induction. \\

	\textbf{Step 1} \\
	Consider an integrable function $\varphi$ and integrate both sides of equation \eqref{flux form} against the function $\psi(x)= \int_x^\infty \varphi(y) \der y$ to deduce that 
	\begin{align} \label{equality psi}
&	\int_0^\infty \psi(z) z \phi(z) \der z  -  \int_0^\infty \phi(z) \psi'(z) \left( \frac{z^{\gamma+\lambda+1}}{M_{\gamma + \lambda}} - \frac{2}{1-\gamma } z^2 \right) \der z - \frac{1}{M_{\gamma+\lambda} } \int_0^\infty \psi(z) z^{\gamma+\lambda} \phi(z) \der z \nonumber\\
	&= \int_0^\infty \int_0^\infty x K(x,y) \left( \psi(x+y)-\psi(x)\right) \phi(x) \phi(y) \der y \der x .
	\end{align} 
We remark that the above formulation requires that $\psi$ has to decay fast enough so that all the integrals in \eqref{equality psi} are finite. 

Select $\overline \delta >0$ such that 
$\max\{0, 1- 2 (\gamma+\lambda)\} < \overline \delta  < 1-\gamma - \lambda $ and consider $\psi_R(x)=x^{\gamma+\lambda + \overline \delta -1 }$ if $x \leq R$ while $\psi_R(x)=R^{\overline \delta} x^{\gamma+\lambda  -1 }$ if $x > R$. 
Using \eqref{equality psi} and the fact that $\gamma+ 2\lambda \geq 1 $ and that $\psi_R(x+y)-\psi_R(x) \leq 0$ we deduce that 
\begin{align*}
0< \frac{\gamma + 2\lambda -1 +2\overline \delta } {1-\gamma } \int_0^R z^{\gamma + \lambda + \overline \delta }  \phi(z) \der z \leq \frac{1}{M_{\gamma+\lambda} } \int_0^\infty z^{2(\gamma+\lambda)+ \overline \delta -1 } \phi(z) \der z < \infty,
\end{align*} 
where the moment $M_{2(\gamma+\lambda)+ \overline \delta -1 }$ is bounded because, due to the choice of $\overline \delta$, we have that $2(\gamma +\lambda ) + \overline \delta -1 < \gamma +\lambda $.
Passing to the limit as $R$ tends to infinity we deduce that \eqref{moment 1+delta} holds. \\
	
		\textbf{Step 2} \\
		We assume that inequality \eqref{moment 1+(n-1)delta} holds and we want to prove \eqref{moment 1+ndelta}.
		Taking $\psi_R(x)= x^{\gamma + \lambda + (n+1) \overline \delta-1 }$ when $x \leq R$ and $\psi_R(x)= R^{\overline \delta} x^{\gamma + \lambda + n \overline \delta-1 }$ when $x > R$ in \eqref{equality psi} we deduce that
			\begin{align*}
&0 <\frac{\gamma+2\lambda -1 +2\overline \delta (n+1 )}{1-\gamma } \int_0^R z^{\gamma + \lambda + (n+1)\overline \delta }  \phi(z) \der z \leq \frac{1}{M_{\gamma+\lambda} } \int_0^\infty z^{2(\gamma+\lambda)+(n+1) \overline \delta -1 } \phi(z) \der z \\
&+ \int_\delta^\infty \int_\delta^\infty  x K(x,y) \left[\psi_R(x+y)-\psi_R(x) \right] \phi(x) \phi(y) \der y \der x + \int_0^\infty \phi(z) \psi_R'(z) z^{\gamma+\lambda+1} \phi(z) \der z .
		\end{align*}
Thanks to the choice of $\overline \delta $, we have that $2(\gamma+\lambda)+(n+1) \overline \delta -1 \leq 	\gamma+\lambda+ n \overline \delta $
	and the desired conclusion follows for $n$ such that $\gamma+\lambda +(n+1)\overline \delta-1 < 1$ by the fact that in that case $\psi_R$ is decreasing and hence the coagulation term is negative as well as the term $\int_0^\infty \phi(z) \psi_R'(z) z^{\gamma+\lambda+1} \phi(z)$. 
	We examine now the case in which $\gamma+\lambda +(n+1)\overline \delta-1 \geq 1$.
	In this case we have that 
	\begin{align*}
&\int_\delta^\infty \int_\delta^\infty x K(x,y) \left[\psi_R(x+y)-\psi_R(x) \right] \phi(x) \phi(y) \der y \der x \\
&\leq c  \int_\delta^\infty \int_\delta^\infty  x K(x,y) \left[(x+y)^{\gamma+\lambda +(n+1) \overline \delta -1} -x^{\gamma+\lambda +(n+1) \overline \delta -1} \right] \phi(x) \phi(y) \der y \der x \\
& \leq c 
\iint_{\{(x,y) \in[\delta, \infty )^{2} : x \leq y\}}  x^{1-\lambda } y^{\gamma+\lambda}  \left[(x+y)^{\gamma+\lambda +(n+1) \overline \delta -1} -x^{\gamma+\lambda +(n+1) \overline \delta -1} \right] \phi(x) \phi(y) \der y \der x\\
&+ c 
\iint_{\{(x,y) \in [\delta, \infty )^{2} : x \geq y\}}  x^{1+\gamma +\lambda } y^{-\lambda}  \left[(x+y)^{\gamma+\lambda +(n+1) \overline \delta -1} -x^{\gamma+\lambda +(n+1) \overline \delta -1} \right] \phi(x) \phi(y) \der y \der x \\
& \leq c \int_\delta^\infty x^{1-\lambda} \phi(x) \der x \int_\delta^\infty y^{2(\gamma+\lambda) +(n+1) \overline \delta -1 } \phi(y) \der y < \infty.
\end{align*} 
Since we also have that
\begin{align*}
    \int_0^\infty \phi(z) \psi_R'(z) z^{\gamma+\lambda+1} \phi(z) \der z \leq c(n)   \int_0^\infty \phi(z)  z^{2(\gamma+\lambda) + (n+1)\delta -1} \phi(z) \der z < \infty
\end{align*}
the desired conclusion follows.
	\end{proof}

\begin{lemma}[Exponential bound] \label{lemma:expo bound}
Assume $K$ to be a homogeneous symmetric coagulation kernel, of homogeneity $\gamma$, satisfying
\eqref{kernel bounds},  \eqref{kernel cont},  with $\gamma , \lambda $ such that  \eqref{avoid_gelation_parameters} holds and such that
$ \gamma+2\lambda \geq  1.$
		Let $\Phi \in \mathcal M_+(\mathbb R) $ be as in Proposition \ref{prop:regularity}  and let $\phi$ be its density. 
	Then there exist two positive constants $L$ and $M$ such that 
	\begin{equation}\label{expo bound}
	\int_{M}^\infty \phi(z) e^{L z } \der z < \infty. 	
	\end{equation}
	\end{lemma} 
\begin{proof}

Adapting the approach used in \cite{fournier2006local} we define the function $\Psi_a $ as 
\[
\Psi_a(L):= \int_M^\infty
 \frac{e^{L \min\{ x, a \}}}{\min\{ x, a \}}   x^2 \phi(x) \der x  
\]
where $M > \delta$. 
Hence, by its definition
\[
\Psi_a'(L)=  \int_M^\infty
e^{L \min\{ x, a \}}x^2 \phi(x) \der x .
\]
We consider a function $\psi$ in \eqref{equality psi} with
$\psi'(x):=
  e^{L \min\{ x, a \}} $
 to deduce that
\begin{align*}
 \frac{2}{1-\gamma } \Psi_a' (L) & \leq \frac{M^{\gamma+\lambda-1}}{M_{\gamma+\lambda}} \int_M^\infty e^{L \min\{ x,a \}} x^{2} \phi(x) \der x +\frac{1} {M_{\gamma+\lambda}} \int_\delta^M  e^{L \min\{ x,a \}} x^{\gamma + \lambda +1 } \phi(x) \der x 
\\
&+\int_{\mathbb R_*}  \int_{\mathbb R_*} x K(x,y) \int_{y}^{x+y}   e^{L \min\{ w, a \}} \der w \phi(x) \phi(y) \der x \der y \\
&+ 
\int_\delta^{\overline z } \left(\frac{1}{M_{\gamma+\lambda}} x^{\gamma + \lambda } -x \right) \psi(x) \phi(x) \der x, 
\end{align*}
for $
\overline z := \left( \frac{1}{ M_{\gamma+\lambda}}\right)^{\frac{1}{1-\gamma - \lambda}}$, where we are assuming without loss of generality that $\overline z> \delta$ and since for every  $z \geq \overline z$ we have
\begin{equation} \label{positivity} 
z- \frac{z^{\gamma + \lambda}}{M_{\gamma + \lambda }} > 0.
\end{equation}
Thus, there exists a $c \geq 0$ such that
\begin{align*}
 \frac{2}{1-\gamma } \Psi_a' (L) \leq  \Psi_a' (L)  \frac{M^{\gamma+\lambda-1}}{M_{\gamma+\lambda}}+ c 
+\int_{\mathbb R_*}  \int_{\mathbb R_*} x K(x,y) \int_{y}^{x+y}   e^{L \min\{ w, a \}} \der w \phi(x) \phi(y) \der x \der y . 
\end{align*}

As in \cite{ferreira2021self} and \cite{fournier2006local} we can deduce, using Jensen's inequality together with the fact that $\gamma + \lambda < 1 $ and that $- \lambda < 1 $, that
\[
\int_{\mathbb R_*}  \int_{\mathbb R_*} x K(x,y) \int_{y}^{x+y}   e^{L \min\{ w, a \}} \der w \phi(x) \phi(y) \der x \der y \leq \Psi_a(L)^{1-\gamma -\lambda } \Psi_a'(L)^{\gamma +\lambda } .
\] 
This, together with the fact that we can take $M$ arbitrary large, implies that 
\begin{align}\label{horrible ODE}
c(\gamma,\lambda) \Psi_a' (L) \leq \Psi_a (L)^{1-\gamma -\lambda} \Psi_a' (L)^{\gamma+\lambda}
+ c,
\end{align}
for a positive constant $c$ and for $ c(\gamma, \lambda)= \frac{2}{1-\gamma }-\frac{M^{\gamma+\lambda-1}}{M_{\gamma + \lambda}}>0.$

By the definition of $\Psi_a$ we have that 
\[
\Psi_a(0) \leq \frac{2}{1-\gamma} \left( 1+ \frac{c_1}{a} \right) \rightarrow \frac{2}{1-\gamma} \text{ as } a \rightarrow \infty.
\]
If we prove that $ \limsup_{a \rightarrow \infty } \Psi_a(L) < \infty $ for some $L$, then we can conclude.
 Indeed we would have 
\[
M \int_M^\infty e^{Lx }  \phi(x) \der x  \leq \int_M^\infty e^{Lx } x  \phi(x) \der x  < \infty.
\] 
Let us prove that $ \limsup_{a \rightarrow \infty } \Psi_a(L) < \infty $ for some $L$.
First of all notice that this is true if $\Psi_a'(L) \leq 1 $ as in this case $\Psi_a (L) \leq L + \lim_{a \rightarrow \infty} \Psi_a(0) $ and the desired conclusion follows. 
If instead $\Psi_a' >1 $, then
\[
c(\gamma, \lambda) \Psi_a'(L) \leq \Psi_a (L)^{1-\gamma -\lambda} \Psi_a' (L)^{\gamma+\lambda}+c  \leq \alpha \Psi_a (L)^{1-\gamma -\lambda} \Psi_a' (L)+c.
\] 
Without loss of generality we can assume $\alpha >0$ such that $ \lim_{a \rightarrow \infty } \Psi_a(0) \neq \alpha$. This implies that a solution of the ODE 
\[
\Psi_a' \left( 1- \alpha  \Psi_a^{1-\gamma-\lambda} \right) =c
\] 
exists for small intervals around $0$ and is such that  $\lim_{a \rightarrow \infty } \Psi_a (L) < \infty $ for $L$ in that interval.
	\end{proof} 

\begin{proposition}[Exponential decay]\label{prop:expo decay}
Assume $K$ to be a homogeneous symmetric coagulation kernel, of homogeneity $\gamma$, satisfying
\eqref{kernel bounds},  \eqref{kernel cont},  with $\gamma , \lambda $ such that  \eqref{avoid_gelation_parameters} holds and such that $\gamma+2\lambda \geq  1.$
		Let $\Phi \in \mathcal M_+(\mathbb R) $ be as in Proposition \ref{prop:regularity}  and let $\phi$ be its density. 
	Then there exists a positive constant $\tilde{M}$ such that 
	\begin{equation}\label{expo decay}
\limsup_{z \rightarrow \infty } \phi(z) e^{\tilde{M} z } < \infty. 	
	\end{equation}
\end{proposition}
\begin{proof}
From equation \eqref{flux form} we deduce that 
\begin{align*}
\left( \frac{2}{1-\gamma} z^2 - \frac{z^{\gamma +\lambda +1 }}{M_{\gamma+\lambda}(\phi)} \right)  \phi(z)  
\leq \int_0^z x \phi(x) \der x - \frac{1}{M_{\gamma+\lambda}(\phi)}\int_0^z x^{\gamma +\lambda }\phi(x) \der x +  J_\phi(z). 
	\end{align*} 
We now show that $J_\phi$ decays exponentially and that the term 
\[
\int_0^z x \phi(x) \der x - \frac{1}{\int_0^\infty y^{\gamma + \lambda } \phi(y) \der y } \int_0^z x^{\gamma +\lambda }\phi(x) \der x 
\]
decays exponentially too. 

First of all, we prove that there exists a constant $M_1>0$ such that $ J_\phi(z) \leq e^{- M_1 z}$ for large $z$. 
To this end we notice that the bound \eqref{expo bound} implies that
\begin{align*}
 J_\phi(z)&= \int_0^z \int_{z-x}^\infty e^{- L(x+y) }e^{L(x+y) } x K(x,y) \phi(x) \phi(y) \der y \der x \\
& \leq  e^{- Lz  }  \int_0^z \int_{z-x}^\infty e^{L(x+y) } x K(x,y) \phi(x) \phi(y) \der y \der x  \leq e^{-M_1 z } . 
\end{align*} 

On the other side, for large values of $z$ we have that
\begin{align*}
&\int_0^z x \phi(x) \der x - \frac{1}{\int_0^\infty y^{\gamma + \lambda } \phi(y) \der y } \int_0^z x^{\gamma +\lambda }\phi(x) \der x  \\
&= \frac{1}{\int_0^\infty y^{\gamma + \lambda } \phi(y) \der y } \left(\int_0^z x \phi(x) \der x \int_0^\infty y^{\gamma + \lambda } \phi(y) \der y  - \int_0^z x^{\gamma +\lambda }\phi(x) \der x\right)   \\
& \leq \frac{1}{\int_0^\infty y^{\gamma + \lambda } \phi(y) \der y } \left( \int_0^\infty y^{\gamma + \lambda } \phi(y) \der y  - \int_0^z x^{\gamma +\lambda }\phi(x) \der x\right)   \\ 
& \leq  \frac{1}{\int_0^\infty y^{\gamma + \lambda } \phi(y) \der y }  \int_z^\infty y^{\gamma + \lambda } \phi(y) \der y \leq C(\rho(M_{\gamma + \lambda })) e^{- M_2 z },
\end{align*}
where $M_2 >0$.

We deduce that for large values of $z $
\begin{align*}
\left( \frac{2z^2 }{1-\gamma} - \frac{z^{\gamma +\lambda +1 }}{\int_0^\infty y^{\gamma + \lambda } \phi(y) \der y} \right)  \phi(z) \leq c \max\{e^{-z M_1}, e^{-z M_2}  \}.
\end{align*}
The conclusion follows since, for every $z > \left(\frac{1-\gamma}{2 \int_0^\infty x^{\gamma+\lambda} \Phi(\der x)}  \right)^{\frac{1}{1-\gamma-\lambda}}$, we have
\begin{align*}
     \frac{2z^2 }{1-\gamma} - \frac{z^{\gamma +\lambda +1 }}{\int_0^\infty y^{\gamma + \lambda } \phi(y) \der y}>0.
\end{align*}
\end{proof}

\section{Non-existence results and properties of the self-similar profiles} \label{sec:non existence}
To study the non-existence of the self-similar solutions we proceed by contradiction and start by assuming that a self-similar solution exists.
To find a contradiction we need to analyse the properties of each self-similar profile.
In the case $\gamma + 2 \lambda >1 $ the fundamental properties that we prove to be true are that $\Phi((0, \delta))=0$ for some positive $\delta>0$ and that $\Phi$ decays sufficiently fast for large sizes.
When $\gamma+2\lambda =1$, it is possible to prove that $\Phi$ decays fast for large sizes, but we have not been able to prove that $\Phi((0,\delta))=0$ for some $\delta>0$.
This is the reason why we require the additional condition \eqref{-lambda moment} in Theorem \ref{thm:non-existence}.

\subsection{Properties of the self-similar profiles}
\label{sec:properties eqch}
In this section we study the properties of each self-similar solution that do not rely on the existence of the self-similar solution and hence do not rely on the assumption $\gamma >-1$. 
\begin{theorem}[Properties of the self-similar profiles] \label{thm:properties}
Let $K$ be a homogeneous symmetric coagulation kernel, of homogeneity $\gamma$, satisfying
\eqref{kernel bounds},  \eqref{kernel cont},  with $\gamma , \lambda $ such that  \eqref{avoid_gelation_parameters} holds. 
\begin{enumerate}
\item If $\gamma+2\lambda > 1$, 
then every self-similar profile $\Phi$ as in Definition \ref{def:self-similar profile} is such that $\Phi((0, \rho(M_{\gamma+\lambda}) ))=0$ where $\rho(M_{\gamma+\lambda}) $ is given by \eqref{rhoM}.  
Additionally, $\Phi$ is such that
\begin{equation} \label{eq expo bound}
\int_{\mathbb R_*} e^{Lx } \Phi(\der x) < \infty 
\end{equation}
for $L>0$ and it is absolutely continuous with respect to the Lebesgue measure. 
Its density $\phi $ is such that 
\begin{equation}\label{eq expo decay}
\limsup_{x \in \mathbb R_*} \phi(x) e^{ \overline{L} x} \leq c 
\end{equation} 
for constants $\overline{L} , c>0$. 
\item If $\gamma+ 2\lambda =1 $ and if, in addition, $\Phi((0, \delta))=0$ for some $\delta>0$, then $\Phi([\delta, \rho(M_{\gamma+\lambda})))=0$ where $ \rho(M_{\gamma+\lambda})$ is given by \eqref{rhoM}. In addition, \eqref{eq expo bound} and \eqref{eq expo decay} hold. 
\end{enumerate} 
\end{theorem}

First of all we prove, see Proposition \ref{prop:support}, that when $\gamma + 2 \lambda >1$, each solution of equation \eqref{eq:ss} in the sense of Definition \ref{def:self-similar profile} is zero near the origin. 
The statement of Theorem \ref{thm:properties} then follows by Proposition \ref{prop:regularity} and Proposition \ref{prop:expo decay} and by the fact that when $\gamma+\lambda=1 $ we are assuming that there exists a $\delta>0$ such that $\Phi((0, \delta))=0$ .
\begin{proposition}[Support of the self-similar solution]\label{prop:support}
Let $K$ be a homogeneous symmetric coagulation kernel, of homogeneity $\gamma$, satisfying
\eqref{kernel bounds},  \eqref{kernel cont},  with $\gamma , \lambda $ such that  \eqref{avoid_gelation_parameters} holds and such that
$\gamma+2\lambda > 1.$
	Let $\Phi \in \mathcal M_+(\mathbb R_*)$ be a self-similar profile as in Definition \ref{def:self-similar profile}. 
	Then $	\Phi((0, \overline \xi ))=0$ for 
	\[
	\overline \xi := \min\left\{ \left( \frac{1}{M_{\gamma+\lambda}(\Phi)} \right)^{\frac{1}{1-\gamma-\lambda}} , \left( \frac{1-\gamma}{ 2M_{\gamma+\lambda}(\Phi)} \right)^{\frac{1}{1-\gamma-\lambda}} \right\}.
	 \]
\end{proposition}
\begin{proof}
We start explaining the proof in a heuristic manner without entering in the technical details that will be explained later. 
Equation \eqref{eq:ss} can be rewritten in the following flux form
\begin{equation} \label{eq:ss flux}
 \partial_x \left( J_\Phi (x) + \left( \frac{x^{\gamma+\lambda +1 } }{ M_{\gamma + \lambda } } - \frac{2 x^2 }{1- \gamma }\right) \Phi(x ) \right)  =\left(  \frac{x^{\gamma+\lambda}}{M_{\gamma+\lambda}}  -x \right) \Phi(x ), 
\end{equation} 
where $J_\phi$ is given by \eqref{flux def}.
For $x$ smaller than $\overline \xi $ we have that $ \frac{x^{\gamma+\lambda}}{M_{\gamma+\lambda}}  -x  \geq 0$
and hence the function 
\[
x \mapsto J_\Phi (x) + \left( \frac{x^{\gamma+\lambda +1 } }{ M_{\gamma + \lambda } } - \frac{2 x^2 }{1- \gamma }\right) \Phi(x ) 
\] 
is increasing and right-continuous at $x=0$.
This implies that 
\[
\lim_{ x \rightarrow 0^+ } \left( J_\Phi (x) + \left( \frac{x^{\gamma+\lambda +1 } }{ M_{\gamma + \lambda } } - \frac{2 x^2 }{1- \gamma }\right) \Phi(x ) \right) =L.
\] 
Since for every
$x \leq \overline \xi $ 
we also have that 
$ \frac{x^{\gamma+\lambda +1 } }{ M_{\gamma + \lambda } } - \frac{2 x^2 }{1- \gamma } \geq 0
$ 
we deduce that $L \geq 0.$

We make a scaling argument to identify the value of $L. $
To this end we notice that the units of the flux are given by $[J_\Phi(x)] =[\Phi]^2 [x]^{3+\gamma }$, where from now on we indicate with $[\cdot]$ the dimensional properties of a quantity, hence
\[
[\Phi ] =\left[x\right]^{-\frac{(3+ \gamma) }{2}}.
\] 
We deduce that 
\[
[x^{\gamma + \lambda +1} \Phi  ] = [x^{\gamma + \lambda +1} ]  [\Phi] =[x^{\frac{\gamma + 2 \lambda -1}{2}}  ].
\] 
Since $\gamma + 2 \lambda > 1 $ this implies that as $x \rightarrow 0 $ the dominant term in equation \eqref{eq:ss flux} is $J_\Phi$, hence 
$J_\Phi(x)\sim L$ as $x \rightarrow 0$. Finally, we prove that $L=0$, in agreement with the statement proven in \cite{ferreira2019stationary} that, when $\gamma + 2 \lambda \geq 1 $, there are no solutions to the constant flux equation. 
Integrating \eqref{eq:ss flux}, we deduce that
\[
 J_\Phi (x) + \left( \frac{x^{\gamma+\lambda +1 } }{ M_{\gamma + \lambda } } - \frac{2 x^2 }{1- \gamma }\right) \Phi(x )  = \int_0^x \left(  \frac{z^{\gamma+\lambda}}{M_{\gamma+\lambda}}  -z \right) \Phi(z) \der z .
\] 
A detailed analysis of this ODE for $\Phi$ implies that 
$\Phi=0$ on the interval $\left(0, \left( \frac{1-\gamma }{2 M_{\gamma +\lambda} } \right)^{\frac{1}{1-\gamma-\lambda } } \right) $, see Step 4. 

We now explain the proof in detail. 
Testing (\ref{weak steady state eq}) with a function of the form $\psi(\xi)=\xi\varphi(\xi),$ where $\varphi\in\textup{C}^{1}_{c}(\mathbb R_*),$ we get to the following equation:
\begin{align}\label{weak flux form}
\int_{(0,\infty)}&\left( \xi -\frac{1}{M_{\gamma+\lambda}} \xi^{\gamma+\lambda} \right) \varphi(\xi)\Phi(\der \xi )
 =\int_{(0,\infty)}\varphi'(\xi)W[\Phi](\der \xi ) ,
\end{align}
where
\begin{align} \label{def W}
   W[\Phi](\der x)= J_\Phi (x) \der x + \left( \frac{x^{\gamma+\lambda +1 } }{ M_{\gamma + \lambda } } - \frac{2 x^2 }{1- \gamma }\right) \Phi(\der x ).
\end{align}

\textbf{Step 1: Regularity of $W[\Phi]$ and an integral representation formula}\\

First of all we prove that, for every $\overline \delta >0$, the restriction of the measure $W[\Phi] $ on the interval $[\overline \delta, 1 ]$ is absolutely continuous with respect to the Lebesgue measure. 
This follows from the fact that for every function $\varphi \in C^1_c([\overline \delta, \overline \xi ])$ with $ 0 < \overline \delta < 1$ we have that 
\[
\int_{\mathbb R_*} \varphi'(\xi)W[\Phi](\der \xi ) \leq c(\Phi, \delta, \gamma, \lambda)  \| \varphi \|_\infty; 
\]
hence $W[\Phi]$, on the interval $[\overline \delta,\overline \xi]$, has a density $W_\phi \in BV([\overline \delta, \overline \xi] )$, see \cite{gariepy2001functions}. 
Since $\overline \delta $ is arbitrary we deduce that $W[\Phi]$ has a density $W_\phi \in L^{1}_{loc}((0, \overline \xi] )$.
This implies that the measure $\Phi_{| (0, \overline \xi ]}$ is absolutely continuous with respect to the Lebesgue measure, its density is $\phi$. 

We now prove that $ W_\phi$ is increasing on $(0, \overline \xi]$.
To this end we consider a sequence of functions $ \{\varphi_n \}_n \subset C^1_c(\mathbb R_* ) $ such that $\varphi_n \rightarrow \chi_{[\xi_1, \xi_2]} $ pointwise, with $0<\xi_1, \xi_2 < \overline \xi$, $\varphi_n '(\xi) \geq 0 $ for $\xi \in \left(0, \frac{\xi_1+\xi_2}{2}\right)$ and such that $\varphi_n '(\xi) \leq 0 $ for $\xi \in \left( \frac{\xi_1+\xi_2}{2}, \infty\right)$. 
Substituting $\varphi_n$ in \eqref{weak flux form} we deduce that for every $n\geq 1$ 
\begin{align*}
\int_0^\infty&\xi \varphi_n(\xi)\phi(\xi ) \der \xi 	-\frac{1}{M_{\gamma+\lambda}} \int_0^\infty \varphi_n(\xi) \xi^{\gamma+\lambda}\phi(\xi ) \der \xi 
 =\int_0^\infty \varphi_n'(\xi)W_\phi( \xi ) \der \xi. 
\end{align*}
Notice that 
\[
\int_0^\infty \xi \varphi_n(\xi)\phi(\xi ) \der \xi \rightarrow \int_{[\xi_1, \xi_2 ] } \xi \phi(\xi ) \der \xi  \quad \text{ as } \quad n \rightarrow \infty
\] 
and that 
\[
\int_0^\infty \xi^{\gamma+\lambda} \varphi_n(\xi) \phi(\xi ) \der \xi \rightarrow \int_{[\xi_1, \xi_2 ] } \xi^{\gamma+\lambda} \phi(\xi ) \der \xi  \quad \text{ as } \quad n \rightarrow \infty.
\]
Additionally, since $W_\phi$ is in $L^1_{loc}((0, \overline \xi))$ and satisfies \eqref{weak flux form} we deduce that $W_\phi \in W^{1}_{loc}(0, \overline \xi)$ and hence is continuous in $(0, \overline \xi)$. Hence we  have that 
\[
\int_{(0,\infty)}\varphi_n'(\xi)W_\phi( \xi ) \der \xi \rightarrow W_\phi (\xi_1) - W_\phi (\xi_2) \quad \text{ as } \quad n \rightarrow \infty.
\] 
As a consequence, we deduce that 
\begin{align} \label{monotonicity}
W_\phi (\xi_2) - W_\phi (\xi_1) = \int_{[\xi_1, \xi_2] } \left(\frac{1}{M_{\gamma +\lambda} } \xi^{\gamma + \lambda}  - \xi \right) \phi(\xi) \der \xi.
\end{align} 
Since if $\xi \leq \overline \xi $, then 
$\frac{1}{M_{\gamma +\lambda} } \xi^{\gamma + \lambda} - \xi \geq 0$,
we deduce that there exists a $\overline \xi $ such that $W_\phi $ is increasing in $(0, \overline \xi]$. 
Hence, since $W_\phi$ is bounded, we have that the following limit exists
\[
0 \leq L:= \lim_{\xi \rightarrow 0^+} W_\phi (\xi) < \infty.
\]

Due to \eqref{minimal boundary cond for ss} we can take the limit as $\xi_1 \rightarrow 0$ in \eqref{monotonicity} and deduce that 
\begin{align} \label{eq W}
W_\phi (\xi_2) - L  = \int_{(0, \xi_2] } \left(\frac{1}{M_{\gamma +\lambda} } \xi^{\gamma + \lambda} - \xi \right) \phi(\xi) \der \xi. 
\end{align} 

\textbf{Step 2: we prove that $\int_{\mathbb R_*} J_\phi(\varepsilon x) \varphi(x) \der x \rightarrow L  \int_{\mathbb R_*}\varphi(x) \der x$ }\\

Since $W_\phi$ is bounded  in $(0, \overline \xi ] $ and $\frac{1}{M_{\gamma+\lambda}} \xi^{\gamma+\lambda} -\frac{2}{1-\gamma} \xi \geq 0$ for every $\xi \leq \overline \xi$ we have that
 $J_\phi $ is also bounded in $(0, \overline \xi ] $.
Since 
\begin{equation} \label{Omega}
 [2/3 z , z ] \times [2/3 z , z ] \subset \left\{ x \in \mathbb R_*^2 : 0 <  x \leq z , z-x \leq  y < \infty    \right\}=:\Omega_z
\end{equation}
we deduce that for every $z \in (0, \overline \xi ] $ we have that
\[
z^{\gamma + 1 } \left( \int_{2z/3 }^z \phi(x) \der x \right)^2 \leq J_\phi(z) \leq C 
\] 
hence 
\begin{equation}\label{bound in mean}
\frac{1}{z}  \int_{2z/3}^z \phi(x) \der x \leq  \frac{C}{z^{\frac{\gamma + 3}{2} }} \quad z \in (0, \overline \xi ]. 
\end{equation} 

Equation \eqref{eq W} implies that 
\begin{align} \label{eq for Fepsilon}
\int_{\mathbb R_*} \varphi\left(x\right) J_{\phi}(\varepsilon x) \der x  & =\frac{1}{\varepsilon}\int_{\mathbb R_*} \varphi\left(\frac{x}{\varepsilon}\right) J_{\phi}( x) \der x = - \frac{1}{\varepsilon} \int_{\mathbb R_*} \varphi\left(\frac{x}{\varepsilon}\right) \left[ \frac{x^{\gamma+\lambda+1}}{M_{\gamma+\lambda}}- \frac{2}{1-\gamma}x^2\right] \phi(x) \der x \\
&  + \frac{1}{\varepsilon} \int_{\mathbb R_*} \varphi\left(\frac{x}{\varepsilon}\right) \int_0^x \left[ \frac{z^{\gamma+\lambda}}{M_{\gamma+\lambda}}- z\right] \phi(z) \der z  \der x  + L \int_{\mathbb R_*} \varphi(x) \der x \nonumber 
\end{align}
where $\varphi \in C_c((0, \overline \xi))$ and where $\varepsilon >0$. 

Thanks to the bounds \eqref{minimal boundary cond for ss} for $\phi$ and to the fact that $\varphi$ is compactly supported we can pass to the limit as $\varepsilon \rightarrow 0$ in equation \eqref{eq for Fepsilon} and deduce that 
\begin{align*}
\lim_{\varepsilon \rightarrow 0 } \int_{\mathbb R_*} J_{\phi}(\varepsilon \xi)  \varphi(\xi) \der \xi = L \int_{\mathbb R_*} \varphi(x) \der x.
\end{align*}
for every $\varphi \in C^1_c((0, \overline \xi))$. \\

\textbf{Step 3: we prove that $L=0$} \\

We want to prove that for every $\varphi \in C^1_c((0, \overline \xi))$
\[
\lim_{\varepsilon \rightarrow 0 } \int_{\mathbb R_*} J_{\phi}(\varepsilon \xi)  \varphi(\xi) \der \xi =0. 
\]
Using the formula for the fluxes we obtain that 
\begin{align*}
 \frac{1}{\varepsilon}  \int_{\mathbb R_*} J_{\phi}( \xi)  \varphi\left(\frac{\xi}{\varepsilon}\right) \der \xi &= \int_{\mathbb R_*}  \varphi(\xi)  J_{\phi}(\varepsilon \xi) \der \xi 
\\
&= \int_{\mathbb R_*}  \int_{\mathbb R_*} x K(x,y) \left( \psi\left(\frac{x+y}{\varepsilon}\right) - \psi\left(\frac{x}{\varepsilon}\right) \right) \Phi (\der x )  \Phi( \der y )\\
&=I_1 + I_2
\end{align*}
where $\psi(x)= \int_0^x \varphi(y) \der y$
and where 
\[
I_1 = \int_{\mathbb R_*}  \int_{(y,\infty)} x K(x,y) \left( \psi\left(\frac{x+y}{\varepsilon}\right) - \psi\left(\frac{x}{\varepsilon}\right) \right) \Phi (\der x )  \Phi( \der y )
\]
and
\[
I_2 = \int_{\mathbb R_*}  \int_{(0, y)} x K(x,y) \left( \psi\left(\frac{x+y}{\varepsilon}\right) - \psi\left(\frac{x}{\varepsilon}\right)\right) \Phi (\der x )  \Phi( \der y ).
\]

Since $\varphi$ is compactly supported we have, by definition, that $ \supp \left[ \psi\left(x+y\right) - \psi\left(x\right) \right] \subset \left\{ (x,y) \in \mathbb R_*^2 : x \leq  R , \ x+y \geq \overline \delta >0 \right\} $
for suitable $\overline \delta >0$ and $R>0.$
As a consequence, using \eqref{minimal boundary cond for ss} we deduce that 
\begin{align*}
 |I_1 | & \leq c\int_{(0, \varepsilon R)} \int_{(\overline \delta\varepsilon/2, \varepsilon R)} y^{-\lambda} x^{1+\gamma+\lambda } \left| \psi\left(\frac{x+y}{\varepsilon}\right) - \psi\left(\frac{x}{\varepsilon}\right) \right| \Phi(\der x) \Phi(\der y) \\
&\leq C(\psi) \frac{1}{\varepsilon}
 \int_{(0, \varepsilon R)} y^{1-\lambda} \Phi(\der y) \int_{(\overline \delta\varepsilon/2, \varepsilon R)}x^{1+\gamma+\lambda } \Phi(\der x) \rightarrow 0 \text{ as } \varepsilon \rightarrow 0.
\end{align*} 
Using again \eqref{minimal boundary cond for ss} we have
\begin{align*}
    | I_2 |& =   \int_{\mathbb R_*}  \int_{(0,y )} x K(x,y) \left| \psi\left(\frac{x+y}{\varepsilon}\right) - \psi\left(\frac{x}{\varepsilon}\right)\right| \Phi (\der x )  \Phi( \der y ) \\
    &\leq C(\psi) \int_{(\varepsilon\delta/2, \infty)} y^{\gamma+\lambda} \Phi( \der y )  \int_{(0, \varepsilon R)} x^{1-\lambda}  \Phi (\der x ) \rightarrow 0 \text{ as } \varepsilon \rightarrow 0. 
\end{align*}
Hence $
\lim_{\varepsilon \rightarrow 0 } \int_{\mathbb R_*} J_{\phi}(\varepsilon \xi)  \varphi(\xi) \der \xi =0$ and therefore $L=0.$\\

\textbf{Step 4: $\Phi$ is zero near zero}\\

Since $L=0$ we deduce by \eqref{eq W} that 
\begin{align} \label{eq W2} 
 W_\phi (z)  = \int_0^z  \left( \frac{1}{M_{\gamma +\lambda} } \xi^{\gamma + \lambda} - \xi \right) \phi(\xi) \der \xi \quad 
 \text{ for } z \in(0, \overline \xi).
\end{align} 
Or, equivalently, since $J_\phi \in BV(\mathbb R_*)$
\begin{align} \label{ODE for the support}
J_\phi (z) + \left( \frac{z^{\gamma + \lambda +1} }{M_{\gamma + \lambda} }- \frac{2 z^2 }{1-\gamma }\right) \phi(z) = \int_0^z  \left( \frac{1}{M_{\gamma +\lambda} } \xi^{\gamma + \lambda} - \xi \right) \phi(\xi) \der \xi \quad 
 \text{ for } a.e. z \in (0, \overline \xi).
\end{align} 

We now rewrite \eqref{ODE for the support} in the following way
\begin{align*} 
&\int_0^z \int_{z-x}^z x K(x,y) \Phi(\der y) \Phi( \der x)  + \int_0^z \int_z^\infty x K(x,y) \Phi(\der y) \Phi(\der x)  \\
&+ \left( \frac{z^{\gamma + \lambda +1} }{M_{\gamma + \lambda} }- \frac{2 z^2 }{1-\gamma }\right) \phi(z) = \int_0^z  \left( \frac{1}{M_{\gamma +\lambda} } \xi^{\gamma + \lambda} - \xi \right) \phi(\xi) \der \xi.
\end{align*} 
If $\Phi$ is different to zero near the origin, then
  \begin{align*}\int_0^z \int_z^\infty x K(x,y)  \Phi( \der x) \Phi(\der y ) & \geq c_1 \int_0^z x^{1-\lambda}  \Phi( \der x) \int_z^\infty y^{\gamma+\lambda}  \Phi( \der y) \\
  &+ c_1 \int_0^z x^{1+\gamma+\lambda}  \Phi(\der x)   \int_z^\infty y^{-\lambda}  \Phi(\der y ) \geq C_+ \int_0^z x^{1-\lambda} \Phi( \der x)
  \end{align*} 
where we have used that for sufficiently small $z$ we have that
\[
\int_z^\infty y^{\gamma+\lambda}  \Phi(\der y)  \geq  C_+
\] 
for a strictly positive constant $C_+$. 

This implies that for $z$ small we have that
  \begin{align*}
0 & =J_\phi(z) + \left( \frac{z^{\gamma + \lambda +1} }{M_{\gamma + \lambda} }- \frac{2 z^2 }{1-\gamma }\right) \phi(z) - \int_0^z  \left( \frac{1}{M_{\gamma +\lambda} } \xi^{\gamma + \lambda} - \xi \right) \phi(\xi) \der \xi \\
& \geq   C_+ \int_0^z x^{1-\lambda} \phi( x) \der x+ \left( \frac{z^{\gamma + \lambda +1} }{M_{\gamma + \lambda} }- \frac{2 z^2 }{1-\gamma }\right) \phi(z) - \int_0^z  \left( \frac{1}{M_{\gamma +\lambda} } \xi^{\gamma + \lambda} - \xi \right) \phi(\xi) \der \xi \\
& \geq \left( \frac{z^{\gamma + \lambda +1} }{M_{\gamma + \lambda} }- \frac{2 z^2 }{1-\gamma }\right) \phi(z) + \int_0^z  \left(  C_+ \xi^{1-\lambda} - \frac{1}{M_{\gamma +\lambda} } \xi^{\gamma + \lambda} \right)  \phi(\xi) \der \xi.
  \end{align*} 
Since we are assuming that $\Phi $ is different from zero near the origin and since $\gamma+2\lambda>1$, we deduce that
$\int_0^z  \left(  C_+ \xi^{1-\lambda} - \frac{1}{M_{\gamma +\lambda} } \xi^{\gamma + \lambda} \right)  \phi(\xi) \der \xi>0$. 
This, together with the fact that for $z \leq \overline \xi$ we have $\left( \frac{z^{\gamma + \lambda +1} }{M_{\gamma + \lambda} }- \frac{2 z^2 }{1-\gamma }\right) \geq 0$ leads to a contradiction. Hence there exists a $\delta>0$ such that $\phi(x)=0$ for $x \in (0, \delta)$. 
We now prove that $\phi(x)=0$ for every $x < \overline \xi$. 
To this end we define the function $H$ as 
\[
H(z):= \int_0^z  \left( \frac{1}{M_{\gamma +\lambda} } \xi^{\gamma + \lambda} - \xi \right) \phi(\xi) \der \xi.
\] 
Due to \eqref{ODE for the support} we have that $H$ satisfies 
\begin{equation} \label{ODE H} 
\frac{d}{dz} H(z)= H(z)Q(z)-J_\phi(z)Q(z) , \quad H(\delta)=0 \quad \text{ for }  \ a.e. \ z\in (\delta, \overline \xi) 
\end{equation}
where 
\begin{equation}\label{def of Q} 
Q(z):=\frac{(1-\gamma ) ( z^{\gamma + \lambda} - M_{\gamma + \lambda}z) }{(1-\gamma ) z^{\gamma + \lambda+1} - 2 M_{\gamma + \lambda } z^2} > 0, \quad   z  \in (\delta ,\overline \xi). 
\end{equation}

The solution of equation \eqref{ODE H} is the function
\begin{equation} \label{H}
H(z)= - \int_\delta^z \frac{J_\phi(s) }{Q(s)} e^{ \int_s^z Q(x) \der x } \der s. 
\end{equation}
 Expression \eqref{H} for $H$ implies that $H$ is negative for $z \in (\delta, \overline \xi)$.
 On the other hand, if we assume that $\phi$ is different from zero in $(0, \overline \xi)$ we deduce that $H$ is positive on $z \in (\delta, \overline \xi)$, by definition. This contradiction implies that $H(z)=0$ for $z \in (\delta,\overline \xi).$
\end{proof}
	
 \begin{remark}
We remark that equation \eqref{eq W2} is valid both for $\gamma + 2 \lambda >1$ and for $\gamma+ 2 \lambda =1$, but the same argument we used to prove that $\Phi((0, \delta))=0$ does not work in the case $\gamma + 2 \lambda =1$.
 \end{remark}

\begin{proof}
[Proof of Theorem \ref{thm:properties}]
As in the proof of Theorem \ref{thm:existence} we prove that if $\Phi((0, \delta))=0$ for some $\delta>0$ then $\Phi((0, \rho(M_{\gamma+\lambda}) ))=0$ for $\rho(M_{\gamma+\lambda}) $ given by \eqref{rhoM}. 
To conclude the proof combine this with  Proposition \ref{prop:regularity} and Proposition \ref{prop:expo decay} with Proposition \ref{prop:support}. 
\end{proof}

\subsection{Non-existence results}
\begin{proof}[Proof of Theorem \ref{thm:non-existence}]
We consider now the case $\gamma + 2\lambda >1$ and $\gamma \leq -1 $ in which $\gamma + \lambda$ could be bigger than zero or smaller than zero, see Figure \ref{fig1}. 
We will provide a proof that relies on the technical proof of Proposition \ref{prop:support}. 
We proceed by contradiction. Assume that a self-similar solution exists. Then thanks to Proposition \ref{prop:support} we know that $\Phi((0, \delta))=0$ for some $\delta>0$. 
Additionally, from Proposition \ref{prop:expo decay} we know that $\lim_{x \rightarrow \infty } \Phi(x)=0$ exponentially.  

We can therefore consider the test function $\varphi\equiv 1$ in equation \eqref{weak steady state eq}. 
This implies 
\begin{align*}
0\leq- \frac{1+\gamma}{1-\gamma} \int_{\mathbb R_*}  \Phi( \der x) 
 =-\frac{1}{2}\int_{\mathbb R_*} \int_{\mathbb R_*} K(x,y)\Phi( \der x) \Phi( \der y)<0,
\end{align*}
which is a contradiction. 

Assume now that $\gamma + 2 \lambda =1 $ and that $\gamma \leq -1$. Notice that in this case we necessarily have that $\gamma+\lambda \leq 0 $. 
Let us assume by contradiction that a self-similar solution exists and is such that
\[
\int_{(0,1]} x^{-\lambda }\Phi(\der x) < \infty.
\]
We show that
\[
\int_{\mathbb R_*} \Phi(\der x ) < \infty.
\] 
Notice that this is true by assumption when $\gamma+\lambda =0$, therefore we restrict the attention to the case $\gamma +\lambda <0.$
To this end we adapt the argument used in the proof of Proposition \ref{prop:moment bounds}, we refer there for the technical aspects, and prove that 
\[
\int_{\mathbb R_*} x^{\gamma+\lambda} \Phi(\der x) < \infty \quad \text{ implies } \quad 
\int_{\mathbb R_*} x^{\gamma+\lambda+\overline \delta } \Phi(\der x) < \infty
\]
when $ 0 < -\gamma-\lambda <\overline \delta < 1-\gamma-\lambda$. To see this we consider $\psi(x)=x^{\gamma+\lambda+\overline \delta -1 } $ in equation \eqref{equality psi} to deduce that
\begin{align*}
     \frac{2\overline \delta }{1-\gamma} M_{\gamma+\lambda+\overline \delta} \leq \frac{M_{2(\gamma+\lambda)+\overline \delta-1}}{M_{\gamma+\lambda}}.
\end{align*}
Now we notice that our choice of $\overline \delta$ implies that 
\[
- \lambda \leq 2(\gamma+\lambda)+\overline \delta-1 \leq \gamma + \lambda.
\]
Hence $M_{2(\gamma+\lambda)+\overline \delta-1}< \infty$ and the desired conclusion follows. 

Similarly, we can prove that 
\[
\int_{\mathbb R_*} x^{\gamma+\lambda+ n \overline \delta} \Phi(\der x) < \infty
\quad \text{ implies } \quad
\int_{\mathbb R_*} x^{\gamma+\lambda+(n+1)\overline \delta } \Phi(\der x) < \infty.
\]
As a consequence, taking $n$ sufficiently large, we deduce that $M_0(\Phi)< \infty$. 
We can therefore consider the test function $\varphi\equiv 1$ in equation \eqref{weak steady state eq}. 
This implies 
\begin{align*}
0\leq- \frac{1+\gamma}{1-\gamma} \int_{\mathbb R_*}  \Phi( \der x) 
 =-\frac{1}{2}\int_{\mathbb R_*} \int_{\mathbb R_*} K(x,y)\Phi( \der x) \Phi( \der y)<0
\end{align*}
which is a contradiction. 
\end{proof}

\subsubsection*{Acknowledgements}
The authors are grateful to A. Nota for many interesting mathematical discussions that led to several of the main ideas contained in this paper.  
The authors would also like to thank
B. Kepka for his help clarifying some technical points of this paper. 
The authors gratefully acknowledge the financial support of the Hausdorff Research Institute for Mathematics (Bonn), through the \textit{Junior Trimester Program on Kinetic Theory}, of the CRC 1060, Project-ID 211504053,
\textit{The mathematics of emergent effects} at the University of Bonn funded through the German
Science Foundation (DFG), of the Bonn International Graduate School of Mathematics at the Hausdorff Center for Mathematics (EXC 2047/1, Project-ID 390685813) funded through the Deutsche Forschungsgemeinschaft (DFG, German Research Foundation),
as well as of the \textit{Atmospheric Mathematics}
(AtMath) collaboration of the Faculty of Science of University of Helsinki, of the ERC Advanced Grant 741487, and of the Academy of Finland, via the  \textit{Finnish centre of excellence in Randomness and STructures} (project No. 346306). 
The funders had no role in study
design, analysis, decision to publish, or preparation of the manuscript.

%\eu{The time dependent equation corresponding to \eqref{eq:ss} is
%\begin{align*} 
%&\partial_{t} \Phi \left( t, \xi \right) - \frac{3 + \gamma }{1-\gamma } \Phi(t,\xi) - \frac{2}{1-\gamma} \xi \partial_\xi \Phi(t, \xi) +  \partial_\xi \left[ \xi^{\gamma + \lambda}  \Phi \left( t, \xi \right) \right] \frac{ c_0}{\int_{0}^{\infty} z^{\gamma +\lambda} \Phi \left(t, z\right)  dz}  = \mathbb K \Phi. 
%\end{align*} 
%which can also be written as 
%\begin{align*} 
%&\partial_{t} \Phi \left( t, \xi \right) - \frac{3 + \gamma }{1-\gamma } \Phi(t,\xi) + \left( \frac{ c_0}{\int_{0}^{\infty} z^{\gamma +\lambda} \Phi \left(t, z\right)  dz} \xi^{\gamma + \lambda} - \frac{2}{1-\gamma} \xi \right)  \partial_\xi \Phi(t, \xi) \\
%&+ \frac{ c_0  \left( \gamma + \lambda \right) }{\int_{0}^{\infty} z^{\gamma +\lambda} \Phi \left(t, z\right)  dz}   \xi^{\gamma + \lambda- 1}  \Phi \left( t, \xi \right)
%= \mathbb K \Phi. \nonumber 
%\end{align*}
%We can use the method of integration along the characteristics to rewrite the equation as 
%\begin{align*} 
%&\frac{d}{dt}\Phi \left( t, X(t,y) \right) - \frac{3 + \gamma }{1-\gamma } \Phi(t,X(t,y) ) = \mathbb K \Phi(t, X(t,y)) 
%\end{align*}
%where $X(t,y)$ is the solution of the ODE 
%\[
%\frac{dx}{dt} = \frac{ c_0}{\int_{0}^{\infty} z^{\gamma +\lambda} \Phi \left(t, z\right)  dz} x^{\gamma + \lambda} - \frac{2}{1-\gamma} x. 
%\]
%But how do we deal with the time dependent moment in the denominator ?}

\bibliographystyle{siam}

\bibliography{bib}

\end{document}